\documentclass[]{siamart220329}
\headers{Hybrid hierarchical matrices with adaptive precision storage}{Ritesh Khan, Erin Carson}

\title{Hybrid hierarchical matrices with adaptive mixed precision storage}

\author{Ritesh Khan\thanks{Department of Numerical Mathematics, Faculty of Mathematics and Physics, Charles University, Prague, Czech Republic. 
Emails: \email{ritesh.khan@matfyz.cuni.cz}, \email{carson@karlin.mff.cuni.cz}.  \\
Both authors are supported by the European Union (ERC, inEXASCALE, 101075632) and the Charles University Research Centre program No. UNCE/24/SCI/005}
\and Erin Carson\footnotemark[1]}

\usepackage[T1]{fontenc}
\usepackage{xr}
\usepackage[utf8]{inputenc}
\usepackage{amsfonts}
\usepackage{amsmath,amssymb}
\allowdisplaybreaks
\usepackage[caption=false]{subfig} 
\usepackage{tikz,pgfplots}
\usepackage{multirow}
\usepackage{makecell}
\usepackage{tikz-3dplot}
\usetikzlibrary{math}
\usepackage{bm}
\usepackage{algorithm}
\usepackage{algpseudocode}
\usepackage{graphicx,epstopdf} 
\usetikzlibrary{fadings}
\usepackage{enumerate}
\usepackage{svg}
\usepackage{graphicx}
\usepackage{mathtools}

\usepackage{cleveref}
\usepackage{hyperref}
\usepackage{relsize}
\usepackage{adjustbox}
\usepackage{tcolorbox}
\usepackage{dsfont}
\usepackage{accents}

\usepackage{multirow}
\usepackage{booktabs}
\usepackage{graphicx}

\usepackage{mathrsfs}

\ifpdf
  \DeclareGraphicsExtensions{.eps,.pdf,.png,.jpg}
\else
  \DeclareGraphicsExtensions{.eps}
\fi

\newsiamremark{remark}{Remark}

\usetikzlibrary{arrows} 
\usetikzlibrary{decorations.markings}

\pgfplotsset{compat=1.16} 

\definecolor{matlabblue}{rgb}{0.0,0.45,0.74}

\newcommand{\dsum}{\displaystyle\sum}

\newcommand{\subsquare}{\scalebox{0.6}{$\blacksquare$}}
\newcommand{\subcircle}{\scalebox{1.2}{$\circ$}}

\newcommand{\bkt}[1]{\left(#1\right)}
\newcommand{\abs}[1]{\lvert#1\rvert}
\newcommand{\magn}[1]{\lVert#1\rVert}

\newcommand{\ceil}[1]{\left\lceil#1\right\rceil}

\newcommand{\Rb}{\mathbb{R}}

\newcommand{\mvp}{matrix-vector product}

\newcommand{\cmt}[1]{\iffalse {#1} \fi}

\makeatletter
\newcommand*{\addFileDependency}[1]{
  \typeout{(#1)}
  \@addtofilelist{#1}
  \IfFileExists{#1}{}{\typeout{No file #1.}}
}
\makeatother

\DeclareUnicodeCharacter{2212}{-}

\begin{document}

\maketitle

\begin{abstract}
Hierarchical matrices are data-sparse approximations of dense matrices that are widely used for fast matrix computations. Hierarchical matrices are built using a tree data structure, with low-rank blocks identified by various admissibility conditions, such as standard admissibility and weak admissibility. This paper introduces a novel hierarchical matrix framework, namely $\mathcal{H}_h$, based on a \emph{hybrid} admissibility condition: we use the standard admissibility at the coarser levels (larger blocks) and the weak admissibility at the finer levels (smaller blocks). This hybrid strategy confines dense blocks only along the diagonal. We provide a criterion that ensures lower storage cost for $\mathcal{H}_h$-matrices compared to $\mathcal{H}$-matrices under the standard admissibility condition. As special cases of $\mathcal{H}_h$-matrices, purely standard admissibility-based $\mathcal{H}$-matrices, as well as purely weak admissibility-based HODLR matrices, can be easily generated. We carry out a rounding error analysis of $\mathcal{H}_h$-matrices and show that the admissible blocks of $\mathcal{H}_h$-matrices can be represented in low precision (precision lower than the working precision) without degrading the overall approximation quality. We provide an explicit rule for dynamically selecting the precision of a given admissible block, thereby proposing an adaptive mixed precision algorithm for constructing and storing $\mathcal{H}_h$-matrices. Furthermore, we show that the use of mixed precision does not compromise the numerical stability and accuracy of the resulting $\mathcal{H}_h$-matrix-vector product. We perform a range of numerical experiments to validate our theoretical findings. Our numerical results show that the proposed adaptive mixed precision $\mathcal{H}_h$-matrices achieve significant storage reductions (up to $11 \times$ for $3$D kernel problems) compared with uniform double precision standard admissibility-based $\mathcal{H}$-matrices, without compromising accuracy. 
\end{abstract}


\section{Introduction}
In many areas of computational science, including PDEs \cite{ho2013hierarchical,massei2022hierarchical}, machine learning \cite{gray2000,huang2006extreme}, graph theory \cite{kriege2020survey}, matrix operations such as matrix vector products, matrix factorizations, and solving linear systems are ubiquitous. However, as the matrix size $N$ associated with a problem increases, the corresponding storage and computational costs become increasingly prohibitive. Over the past few decades, substantial research effort \cite{greengard1987fast,hackbusch1999sparse,borm2003hierarchical,xia2010fast,gillman2012direct,amestoy2017complexity} has been devoted to developing techniques for accelerating matrix operations by exploiting rank-structuredness; among these, a popular framework is hierarchical matrices. Hierarchical matrices (in short, $\mathcal{H}$-matrices) are data-sparse approximations of the original matrix, introduced in \cite{hackbusch1999sparse,hackbusch2000sparse}. Hierarchical matrices can be categorized based on the admissibility condition (e.g., standard admissibility \cite{borm2003hierarchical} and weak admissibility \cite{weak_hackbusch2004}) and on how the bases of the admissible blocks are constructed (e.g., non-nested \cite{hackbusch2000sparse} and nested \cite{borm2009construction}). Hierarchical matrix formats such as $\mathcal{H}$-matrices \cite{hackbusch1999sparse,hackbusch2000sparse}, $\mathcal{H}^2$-matrices \cite{borm2009construction}, HODLR (Hierarchically Off-Diagonal Low-Rank) \cite{ambikasaran2013mathcal,ambikasaran2019hodlrlib}, HSS (Hierarchically Semi Separable) \cite{xia2010fast,chandrasekaran2007fast}, and HBS (Hierarchically Block Separable) are widely used to accelerate computations with matrices arising from boundary element methods (BEM) and kernel methods. In addition to fast solvers \cite{chandrasekaran2007fast,xia2010fast,ambikasaran2013mathcal}, hierarchical matrices are also employed in matrix factorization \cite{bebendorf2005hierarchical}, preconditioning \cite{grasedyck2009domain, bebendorf2005hierarchical}, RBF interpolation \cite{ambikasaran2013mathcal}, $N$-body problems \cite{khan2022numerical} and related applications. We refer the reader to \cite{hmatrix_book,bebendorf_book} and the references therein for further details on hierarchical matrices. In this work, we focus on $\mathcal{H}$-matrices constructed using a geometrically balanced $2^d$-tree \cite{grasedyck2003construction} (i.e., quad-tree in $2$D, oct-tree in $3$D, etc.), with the bases formed in a non-nested manner. 

Recently, there has been growing interest in mixed precision algorithms due to the increasing availability of low precision arithmetic on modern hardware. For example, vendors such as NVIDIA Tensor Core GPUs, Google TPUs, and AMD Instinct accelerators (MI300 series) provide hardware capable of performing floating-point computations in 16-bit (IEEE fp16 and bfloat16) formats. The use of low precision arithmetic can lead to significant performance improvements, including faster computations, reduced memory usage, and lower energy consumption, which are valuable for large-scale simulations, data-intensive applications, artificial intelligence and machine learning. However, careless use of low precision can potentially lead to loss of accuracy. To address this, mixed precision or adaptive precision strategies have been developed, where appropriate precisions are used in different parts of a computation while maintaining accuracy or stability. We direct the reader to the recent surveys \cite{higham2022mixed,abdelfattah2021survey} for a detailed overview of mixed precision algorithms in numerical linear algebra.

We would like to highlight a few related works that use mixed precision schemes for rank-structured matrices. In \cite{abdulah2021accelerating}, a mixed precision approach is used for Cholesky factorization in the context of geostatistical modeling, utilizing fp64, fp32, and fp16 floating-point formats. \cite{ooi2020effect} investigates the impact of mixed precision computation on $\mathcal{H}$-matrix–vector products, where fp64 and fp32 arithmetic are used, and the resulting matrix-vector products are applied to accelerate iterative solvers. These studies are primarily application-driven and performance-oriented. The work \cite{amestoy2023mixed} demonstrates that singular vectors associated with small singular values can be represented in lower precision without sacrificing accuracy and presents a mixed precision single-level BLR matrix and its LU factorization. \cite{kriemann2024performance} studies the advantages of $\mathcal{H}$-matrix-vector products with floating point compression. An accumulator-based approach to $\mathcal{H}$-matrix arithmetic is presented in \cite{doi:10.1137/24M1649009}. The paper \cite{carson2025mixed} presents a rounding error analysis of a special class of $\mathcal{H}$-matrices, HODLR matrices, and proposes mixed precision HODLR matrices. It also analyzes the stability of the matrix–vector products and LU decompositions of the resulting mixed precision HODLR matrix representation. However, a rounding error analysis for more general $\mathcal{H}$-matrices and mixed precision $\mathcal{H}$-matrices remains absent. 

In this work, we aim to bridge this gap by proposing a more efficient adaptive mixed precision $\mathcal{H}$-matrix representation. This work differs from \cite{carson2025mixed} in many aspects. While \cite{carson2025mixed} focuses on binary-tree-based HODLR matrices, which primarily rely on an algebraic block splitting of the original matrix, we consider $2^d$-tree-based hierarchical structures that are more complex and more general. In addition, we introduce a novel class of hybrid hierarchical matrices, of which standard admissibility-based $\mathcal{H}$-matrices as well as HODLR matrices can be obtained as special cases. Moreover, HODLR matrices can become inefficient in higher dimensions $(d>1)$, as the numerical rank grows as a positive power of the matrix size $N$ \cite{khan2022numerical}, and therefore it may not be a truly \emph{quasi-linear} complexity algorithm. By exploiting the geometric information of the clusters, our hierarchical representations provide better control over the numerical rank of the admissible blocks, which is important for efficiently representing them in lower precision. Thus, this work could be attractive for large-scale multidimensional problems.

The main contributions of this paper are as follows.
\begin{enumerate}
    \item We propose a novel hybrid admissibility condition, which exploits the standard admissibility condition at coarser levels and the weak admissibility condition at finer levels. Based on this idea, we present a new class of hierarchical matrices, referred to as $\mathcal{H}_h$-matrices. The standard admissibility-based $\mathcal{H}$-matrices and the weak admissibility-based HODLR matrices can be obtained as special cases of $\mathcal{H}_h$-matrices. We demonstrate the potential storage advantages of $\mathcal{H}_h$-matrices.
    \item We conduct a rounding error analysis of $\mathcal{H}_h$-matrix approximation, which provides an explicit rule for selecting the appropriate precision for each admissible block without compromising accuracy. This motivates us to propose an adaptive mixed precision $\mathcal{H}_h$-matrix representation.
    \item Using the adaptive mixed precision $\mathcal{H}_h$-matrix representation, we analyze the error in computing matrix–vector products and present a corresponding backward error bound.
    \item We perform a series of experiments with dense kernel matrices arising from BEM and kernel methods. The numerical results confirm our theoretical findings and demonstrate that adaptive mixed precision $\mathcal{H}_h$-matrices can significantly reduce storage requirements compared to uniform precision schemes while maintaining the same level of accuracy. Our implementation is made publicly available at \url{https://github.com/riteshkhan/ampHmat}.
\end{enumerate}

The rest of this paper is organized as follows. \Cref{sec:prelim} reviews hierarchical matrices under various admissibility conditions and presents a comparative analysis. In \Cref{sec:hybrid_hmat}, we introduce hybrid hierarchical matrices and discuss the benefits of this approach. In \Cref{sec:adaptive_prec_Hmatrix}, we discuss the proposed adaptive mixed precision hybrid hierarchical matrices and the matrix–vector product using the resulting representations. Numerical experiments validating the theoretical results and showcasing the advantages of the adaptive mixed precision hybrid hierarchical matrices are presented in \Cref{sec:num_results}, followed by a conclusion in \Cref{sec:conclusion}.

Throughout this paper, we adopt the following notation. Unless stated otherwise, the unsubscripted norm $\magn{\cdot}$ refers to the Frobenius norm. We use the standard model of floating-point arithmetic \cite{higham2002accuracy}. The notation $\widehat{X}$ denotes $X$ computed in finite precision. We use the notation $\text{fl} \bkt{\cdot}$ to represent the computed value of the corresponding expression. The notation $\lesssim$ is used in the rounding error analysis when safely dropping the negligible terms of second and higher orders. We refer to precision with unit roundoff $u$ as ``precision $u$'', as is standard in the literature. A uniform precision scheme is one in which only one precision is used throughout the computation. \Cref{tab:precisions} lists the parameters corresponding to the various floating-point formats used in this work. Unless stated otherwise, we consider fp64 (double) as the working precision.

\begin{table}[H]
\resizebox{\textwidth}{!}{%
\centering
\begin{tabular}{llllllll}
 \quad fp format               & \text{bits} & $\bkt{1-e-m}$ & \qquad $u$ & \quad $x_{\min}$ & \quad $x_{\max}$ & $e_{\min}$ & $e_{\max}$ \\ \hline \hline
double (fp64)   &    64 & $\bkt{1-11-52}$  & $1.11\times 10^{-16}$  & $2.23 \times 10^{-308}$   &  $1.80\times 10^{308}$       &  -1022      &  1023      \\ \hline
single (fp32)   &    32 & $\bkt{1-8-23}$  & $5.96 \times  10^{-8}$  & $1.18 \times 10^{-38}$   &  $3.40\times 10^{38}$        &   -126     &    127    \\ \hline
half (fp16)     &    16 &  $\bkt{1-5-10}$    & $4.88 \times  10^{-4}$   & $6.10 \times 10^{-5}$        &  $6.55 \times 10^{4}$      & -14  &   15     \\ \hline
bfloat16 (bf16) &    16 & $\bkt{1-8-7}$      & $3.91 \times  10^{-3}$  &   $1.18 \times 10^{-38}$     &  $3.39 \times 10^{38}$      & -126  &  127           \\ \hline
quarter (q43) &  8  & $\bkt{1-4-3}$       &  $2.5 \times 10^{-1}$  &  $1.56 \times 10^{-2}$      &  $2.4 \times 10^{2}$      & -6  &   7       \\ \hline
\end{tabular}}
\caption{Parameters for various floating-point formats used in this work. The symbols $e$ and $m$ represent the number of bits for exponent and mantissa, respectively. The symbol $u$ refers to the unit roundoff. The symbols $x_{\min}$ and $x_{\max}$ correspond to the smallest and largest representable positive floating-point numbers, with their exponents indicated by $e_{\min}$ and $e_{\max}$, respectively.}
\label{tab:precisions}
\end{table}

\section{Hierarchical matrices} \label{sec:prelim}
Hierarchical matrices ($\mathcal{H}$-matrices) are multilevel block low-rank representations of dense matrices. These representations are constructed using a tree data structure, and the low-rank blocks are chosen according to an admissibility condition. In this section, we review hierarchical matrices based on various admissibility conditions and compare their storage costs. 

\subsection{Tree construction} Let $\Omega \subset \mathbb{R}^d$ denote a domain containing $N$ particles (or points), whose interactions are represented by a matrix $H \in \mathbb{R}^{N \times N}$. Typically, $H$ is a dense matrix. We further assume that $\Omega$ is contained within a $d$-dimensional hypercube $C$. $\mathcal{H}$-matrices can be constructed using various tree structures, such as block cluster tree \cite{hmatrix_book}, balanced \cite{grasedyck2003construction, hmatrix_book}, $k$-d tree \cite{ambikasaran2019hodlrlib}, and others \cite{massei2022hierarchical}. However, to simplify the analysis and exposition, we consider a (geometrically) balanced $2^d$-tree \cite{grasedyck2003construction} (i.e., quad-tree in $2$D, oct-tree in $3$D, etc.) throughout this paper. 

Let $\mathcal{T}^L$ represent a balanced $2^d$-tree of depth $L$. At level $0$ (root level) of the tree is the hypercube $C$ itself. Each hypercube $C_i^{\bkt{l}}$ at level $l < L$ is partitioned into $2^d$ finer hypercubes $C_{\alpha}^{\bkt{l+1}}$, where $\alpha = 1,2,\dots, 2^d$, which belong to level $l+1$ of the tree. The former is the parent of the latter, and the latter are its children. The cluster $\mathcal{C}_i^{\bkt{l}}$ is defined as the collection of the indices corresponding to the particles in the hypercube $C_i^{\bkt{l}}$. The partitioning continues recursively up to level $L$ (the leaf level), where each leaf hypercube contains at most $n_{\max}$ particles. It can be noted that for a \emph{perfectly} balanced $2^d$-tree, $N = n_{\max} 2^{dL} \implies L = \ceil{\log_{2^d}\bkt{N/{n_{\max}}}}$.

\subsection{Various admissibility conditions}
The low-rank block matrices of an $\mathcal{H}$-matrix are identified based on geometrical information known as admissibility conditions. Let $\mathcal{C}_i$ and $\mathcal{C}_j$ be two clusters at the same level of a balanced $2^d$-tree. For notational convenience, the superscript $``l"$ is occasionally omitted when both clusters are at the same level, and their index sets are denoted by $I$ and $J$, respectively. If the admissibility condition is satisfied, the block matrix of $H$ corresponding to $I$ and $J$, $H_{I,J} \in \Rb^{\abs{I} \times \abs{J}}$, is approximated by a low-rank matrix with a target accuracy $\epsilon$, i.e.,

\begin{align}
    \magn{H_{I,J} - U_I \bkt{V_J}^*} \leq \epsilon \magn{H_{I,J}}.
\end{align}

In this work, the truncated SVD is employed as the low-rank approximation technique in the numerical experiments.

Two commonly used admissibility conditions in the literature are the standard (or strong) admissibility condition \cite{borm2003hierarchical} and the weak admissibility condition \cite{weak_hackbusch2004}. 

The $\mathcal{H}$-matrix formats based on standard and weak admissibility conditions are referred to as $\mathcal{H}_s$-matrices (the subscript $``s"$ indicates the use of standard admissibility condition) and HODLR matrices, respectively.

\subsection{\texorpdfstring{$\mathcal{H}_s$}{Hs}-matrices}
$\mathcal{H}_s$-matrices are based on standard (or strong) admissibility conditions \cite{borm2003hierarchical}, where the admissible clusters are well-separated or far-field (at least one box away). 
Let $\mathcal{C}_i$ and $\mathcal{C}_j$ be two distinct clusters at the same level of a balanced $2^d$-tree, with index sets $I$ and $J$, respectively. The block $I \times J$, or the pair of clusters $\mathcal{C}_i$ and $\mathcal{C}_j$, is said to be admissible under the standard admissibility condition, denoted by $\text{adm}_s \bkt{\mathcal{C}_i, \mathcal{C}_j} = \text{true}$, if the following condition holds:
\begin{align} \label{eq:H_admiss}
    \text{adm}_s \bkt{\mathcal{C}_i, \mathcal{C}_j} = \text{true} \iff \min \bkt{\text{diam} \bkt{C_i}, \text{diam} \bkt{C_j}} \leq \eta \text{ dist} \bkt{C_i,C_j},
\end{align}
where $$\text{diam}(X)= \sup\{\|u-v\|_{2}:u,v\in X\}$$ and $$\text{dist}(X,Y) = \inf\{\|x-y\|_{2}:x\in X, y\in Y\}.$$

The admissibility parameter $\eta>0$ controls the separation distance and the number of admissible pairs. From \cref{eq:H_admiss}, it follows that the smaller the $\eta$, the further apart the admissible clusters; conversely, the larger the $\eta$, the closer the admissible clusters.

If the admissibility condition is not satisfied, the corresponding cluster pair is inadmissible. The self-cluster and the neighboring clusters are inadmissible clusters.

The interaction list and the neighbor list of a cluster $\mathcal{C}_i$ under the standard admissibility condition are defined as follows:

The interaction list of $\mathcal{C}_i$ is defined as
\begin{align*} \label{eq:IL_strong}
    \mathcal{IL}_{s} \bkt{\mathcal{C}_i} := \big \{ \mathcal{C}_j : \text{adm}_s \bkt{\mathcal{C}_i, \mathcal{C}_j} = \text{true} \text{ and } \text{adm}_s \bkt{\text{parent} \bkt{\mathcal{C}_i}, \text{parent} \bkt{\mathcal{C}_j}} = \text{false} \big \}.
\end{align*}

The neighbors list of $\mathcal{C}_i$ is given by
\begin{align*} \label{eq:N_strong}
    \mathcal{N}_{s} \bkt{\mathcal{C}_i} := \big \{ \mathcal{C}_j : \text{adm}_s \bkt{\mathcal{C}_i, \mathcal{C}_j} = \text{false} \text{ and } j \neq i \big \}.
\end{align*}

\Cref{fig:h_interaction} shows the neighbors and interaction lists of a cluster (self-cluster) at different levels of the tree in $2$D, when $\eta = \sqrt{2}$ in \cref{eq:H_admiss}. The corresponding $\mathcal{H}_s$-matrix at different levels is illustrated in \Cref{fig:hmat_2d_figs}.

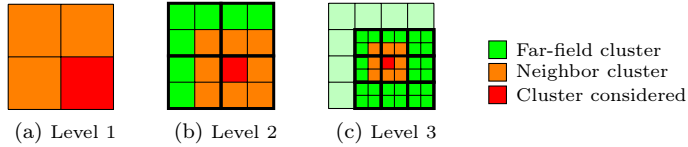
\begin{figure}[H]
\begin{center}
\subfloat[\scriptsize Level $1$]{
    \begin{tikzpicture}[scale=0.7]
        \fill [red] (1,0) rectangle (2,1);
        \fill [orange] (0,0) rectangle (1,1);
        \fill [orange] (1,1) rectangle (2,2);
        \fill [orange] (0,1) rectangle (1,2);
        \draw[step=1cm, black] (0, 0) grid (2, 2);
    \end{tikzpicture}
    \label{h_interaction_1}
    }\quad 
    \subfloat[\scriptsize Level $2$]{
    \begin{tikzpicture}[scale=0.35]
        \draw[black] (0, 0) grid (4, 4);
        \fill [green] (0,0) rectangle (1,1);
        \fill [green] (0,1) rectangle (1,2);
        \fill [green] (0,2) rectangle (1,3);
        \fill [green] (0,3) rectangle (1,4);
        \fill [green] (1,3) rectangle (2,4);
        \fill [green] (2,3) rectangle (3,4);
        \fill [green] (3,3) rectangle (4,4);
        \fill [red] (2,1) rectangle (3,2);
        \fill [orange] (1,0) rectangle (2,1);
        \fill [orange] (1,1) rectangle (2,2);
        \fill [orange] (1,2) rectangle (2,3);
        \fill [orange] (2,2) rectangle (3,3);
        \fill [orange] (3,2) rectangle (4,3);
        \fill [orange] (2,0) rectangle (3,1);
        \fill [orange] (3,0) rectangle (4,1);
        \fill [orange] (3,1) rectangle (4,2);
        \node[anchor=north] at (1.5,2.8) {};
        \node[anchor=north] at (1.5,1.8) {};
        \node[anchor=north] at (.5,1.8) {};
        \draw[black] (0, 0) grid (4, 4);
        \draw[line width=0.4mm,  black] (0, 0) rectangle (2, 2);
        \draw[line width=0.4mm,  black] (2, 2) rectangle (4, 4);
        \draw[line width=0.4mm,  black] (2, 0) rectangle (4, 2);
        \draw[line width=0.4mm,  black] (0, 2) rectangle (2, 4);
    \end{tikzpicture}
    \label{h_interaction_2}
    }\quad 
    \subfloat[\scriptsize Level $3$]{
    \begin{tikzpicture}[scale=0.35]
        \fill [green!25] (0,0) rectangle (1,1);
        \fill [green!25] (0,1) rectangle (1,2);
        \fill [green!25] (0,2) rectangle (1,3);
        \fill [green!25] (0,3) rectangle (1,4);
        \fill [green!25] (1,3) rectangle (2,4);
        \fill [green!25] (2,3) rectangle (3,4);
        \fill [green!25] (3,3) rectangle (4,4);
        \node[anchor=north] at (1.5,2.8) {};
        \node[anchor=north] at (1.5,1.8) {};
        \node[anchor=north] at (.5,1.8) {};
        \draw[black] (0, 0) grid (4, 4);
        \draw[line width=0.4mm,  fill=green] (1, 0) rectangle (2, 1);
        \draw[line width=0.4mm,  fill=green] (1, 1) rectangle (2, 2);
        \draw[line width=0.4mm,  fill=green] (2, 2) rectangle (3, 3);
        \draw[line width=0.4mm,  fill=green] (3, 2) rectangle (4, 3);
        \draw[line width=0.4mm,  fill=green] (2, 0) rectangle (3, 1);
        \draw[line width=0.4mm,  fill=green] (3, 0) rectangle (4, 1);
        \draw[line width=0.4mm,  fill=green] (3, 1) rectangle (4, 2);
        \draw[line width=0.4mm,  fill=green] (1, 2) rectangle (2, 3);
        \fill [orange] (1.5,1) rectangle (3,2.5);
        \fill [red] (2,1.5) rectangle (2.5,2);
        \draw[step=0.5cm, black] (1, 0) grid (4, 3);
        \draw[line width=0.4mm] (1, 0) rectangle (2, 1);
        \draw[line width=0.4mm] (1, 1) rectangle (2, 2);
        \draw[line width=0.4mm] (2, 2) rectangle (3, 3);
        \draw[line width=0.4mm] (3, 2) rectangle (4, 3);
        \draw[line width=0.4mm] (2, 0) rectangle (3, 1);
        \draw[line width=0.4mm] (3, 0) rectangle (4, 1);
        \draw[line width=0.4mm] (3, 1) rectangle (4, 2);
        \draw[line width=0.4mm] (1, 2) rectangle (2, 3);
    \end{tikzpicture}
    \label{h_interaction_3}
    }\quad
    \subfloat{
        \begin{tikzpicture}
            [
            box/.style={rectangle,draw=black, minimum size=0.1cm},scale=0.1
            ]
            \node[box,fill=red,,font=\tiny,label=right: \scriptsize Cluster considered,  anchor=west] at (-8,8){};
            \node[box,fill=orange,,font=\tiny,label=right:\scriptsize Neighbor cluster,  anchor=west] at (-8,11){};
            \node[box,fill=green,,font=\tiny,label=right:\scriptsize Far-field cluster,  anchor=west] at (-8,14){};
        \end{tikzpicture}
    }
    \caption{Interaction and neighbors lists under the standard admissibility condition $(\eta = \sqrt{2})$ at different levels in $2$D. Lighter shade colors denote the interaction list at the coarser levels.}
    \label{fig:h_interaction}
\end{center}
\end{figure}

    \begin{figure}[H]
        \centering
        \subfloat[\scriptsize  At level $1$.]{
        \includegraphics[scale=.2]{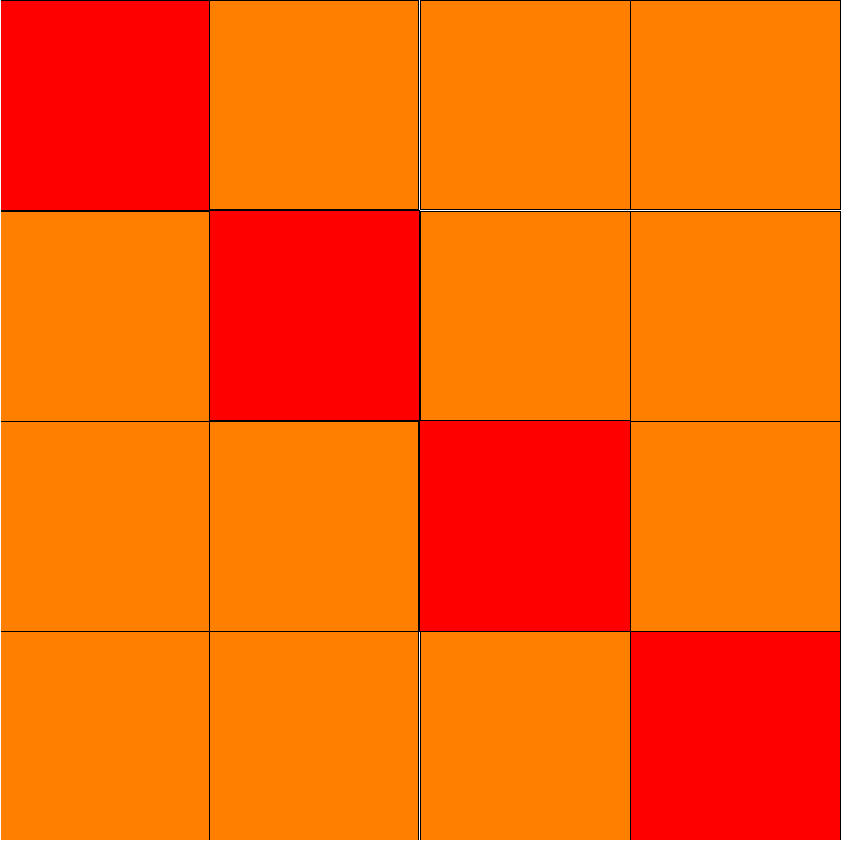}}
        \subfloat[\scriptsize  At level $2$.]{
        \includegraphics[scale=.205]{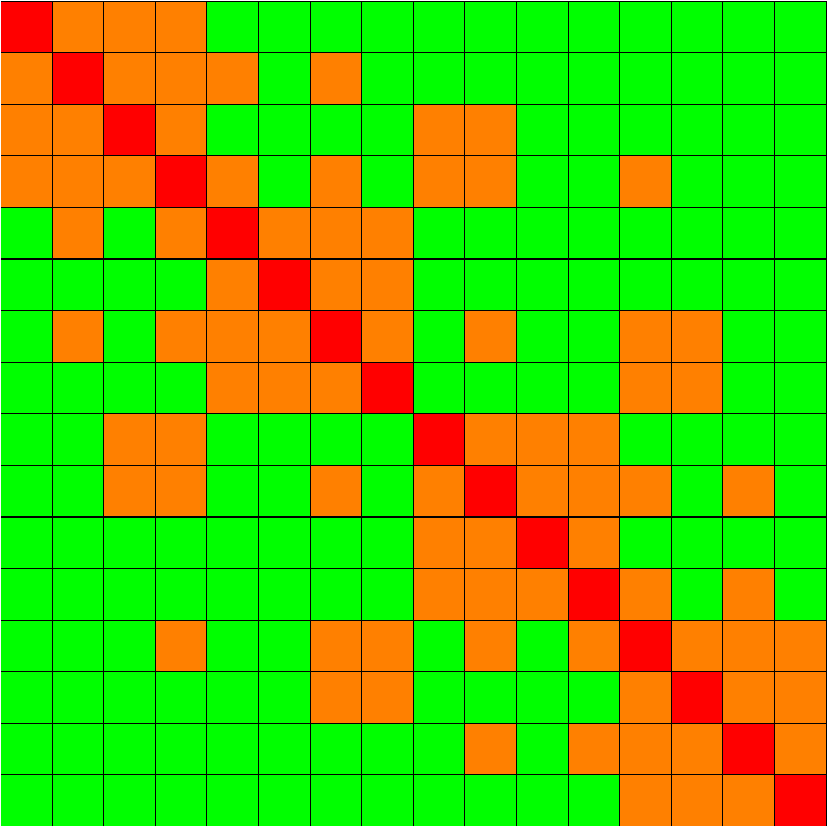}}
        \subfloat[\scriptsize  At level $3$.]{
        \includegraphics[scale=.2]{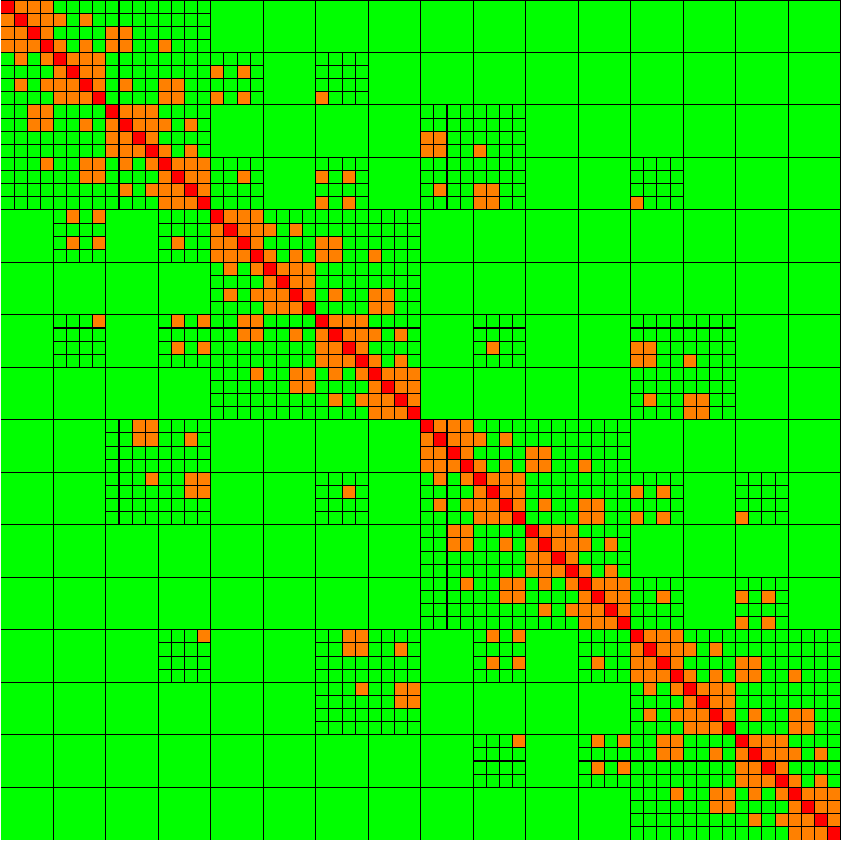}} \quad
        \subfloat{
        \begin{tikzpicture}
            [
            box/.style={rectangle,draw=black, minimum size=0.2cm},scale=0.1
            ]
            \node[box,fill=red,,font=\tiny,label=right:\scriptsize Diagonal (dense) block,  anchor=west] at (8,0){};
            \node[box,fill=orange,,font=\tiny,label=right:\scriptsize Neighbor (dense) block,  anchor=west] at (8,5){};
            \node[box,fill=green,,font=\tiny,    label={right:\scriptsize \makecell[l]{Far-field\\(compressed) block}},  anchor=west] at (8,10){};
        \end{tikzpicture}
        }
        \caption{$\mathcal{H}_{s}$-matrix with $\eta = \sqrt{2}$ at different levels in $2$D (no compression at level $1$).}
        \label{fig:hmat_2d_figs}
    \end{figure}
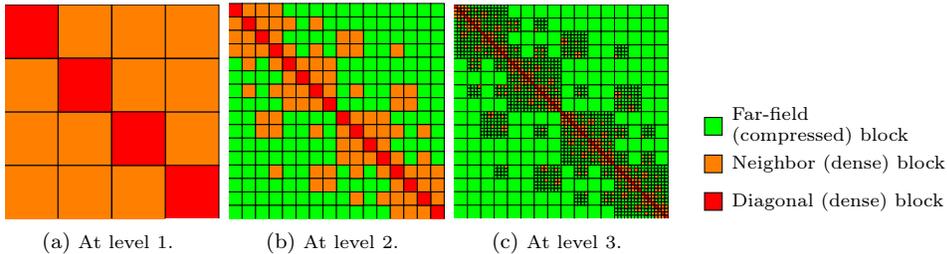

One can obtain an upper bound on the sizes of the interaction list and the neighbor list of a cluster in a balanced $2^d$-tree \cite{weak_hackbusch2004}.
The maximum sizes of the interaction list and the neighbor list of a cluster at \emph{any level} in $d$ dimensions are, respectively,
\begin{align} \label{eq:standard_IL_size}
   \abs{\mathcal{IL}_{s} \bkt{\mathcal{C}_i}} \leq \bkt{2^d-1} \bkt{1 + {2\sqrt{d}}/{\eta}}^d = C^{\prime},
\end{align}
\begin{align} \label{eq:standard_N_size}
    \abs{\mathcal{N}_{s} \bkt{\mathcal{C}_i}} \leq \bkt{1 + {2\sqrt{d}}/{\eta}}^d - 1 = C^{\prime\prime}.
\end{align}

It is worth noting that $\eta = \sqrt{d}$ generates a structure identical to the balanced FMM's interaction list \cite{greengard1987fast}. In particular, for the two-dimensional case with $\eta=\sqrt{2}$ (\Cref{fig:h_interaction}), we obtain $\abs{\mathcal{IL}_{s} \bkt{\mathcal{C}_i}} \leq 27 \text{ and } \abs{\mathcal{N}_{s} \bkt{\mathcal{C}_i}} \leq 8.$

\subsection{HODLR matrices}
HODLR matrices are a class of $\mathcal{H}$-matrices based on weak admissibility conditions \cite{weak_hackbusch2004}. The notion of weak admissibility was introduced in the context of one-dimensional problems, where the block matrices corresponding to adjacent intervals are compressed. More precisely, under this notion, all the off-diagonal block matrices are approximated by low-rank matrices, while the leaf-level diagonal blocks are only full-rank blocks. In the literature, HODLR matrices are typically constructed using a binary tree \cite{massei2020hm,ambikasaran2013mathcal}, regardless of the dimension of the underlying domain. 
However, since a balanced $2^{d}$-tree is employed in this work, the HODLR matrices considered herein are constructed using the same $2^{d}$-tree structure for consistency. 
One advantage of using a $2^d$-tree over a binary tree is that it yields smaller size admissible blocks, resulting in lower numerical ranks in higher dimensions.

Under the weak admissibility condition, all the non-self or adjacent clusters are admissible. For two clusters $\mathcal{C}_i$ and $\mathcal{C}_j$ at the same level of the $2^d$-tree, the weak admissibility condition is given by

\begin{align} \label{eq:HODLR_admiss}
    \text{adm}_w \bkt{\mathcal{C}_i, \mathcal{C}_j} = \text{true} \iff i \neq j.
\end{align}

The interaction list of $\mathcal{C}_i$ under the weak admissibility condition is defined as
\begin{align*}
    \mathcal{IL}_{w} \bkt{\mathcal{C}_i} := \big \{ \mathcal{C}_j : \text{adm}_w \bkt{\mathcal{C}_i, \mathcal{C}_j} = \text{true} \big \} =  \big \{ \mathcal{C}_j :  j \neq i \big \}.
\end{align*}

\begin{figure}[H]
\begin{center}
\subfloat[\scriptsize Level $1$]{
    \begin{tikzpicture}[scale=0.7]
        \fill [red] (1,0) rectangle (2,1);
        \fill [matlabblue] (0,0) rectangle (1,1);
        \fill [matlabblue] (1,1) rectangle (2,2);
        \fill [matlabblue] (0,1) rectangle (1,2);
        \draw[step=1cm, black] (0, 0) grid (2, 2);
        \draw[line width=0.4mm,  black] (0, 0) rectangle (2, 2);
    \end{tikzpicture}
    \label{hodlr_interaction_1}
    }\quad 
    \subfloat[\scriptsize Level $2$]{
    \begin{tikzpicture}[scale=0.35]
        \draw[black] (0, 0) rectangle (4, 4);
        \fill [matlabblue!20] (0,0) rectangle (4,4);
        \fill [red] (2,1) rectangle (3,2);
        \fill [matlabblue] (2,0) rectangle (3,1);
        \fill [matlabblue] (3,0) rectangle (4,1);
        \fill [matlabblue] (3,1) rectangle (4,2);
        \draw[black] (2, 0) grid (4, 2);
        \node[anchor=north] at (1.5,2.8) {};
        \node[anchor=north] at (1.5,1.8) {};
        \node[anchor=north] at (.5,1.8) {};
        \draw[line width=0.4mm,  black] (2, 0) rectangle (4, 2);
    \end{tikzpicture}
    \label{hodlr_interaction_2}
    }\quad 
    \subfloat[\scriptsize Level $3$]{
    \begin{tikzpicture}[scale=0.35]
        \draw[black] (0, 0) rectangle (4, 4);
        \fill [matlabblue!20] (0,0) rectangle (4,4);
        \fill [matlabblue] (2,1) rectangle (3,2);
        \fill [red] (2,1.5) rectangle (2.5,2); 
        \draw[step=0.5cm, black] (2, 1) grid (3, 2);
        \node[anchor=north] at (1.5,2.8) {};
        \node[anchor=north] at (1.5,1.8) {};
        \node[anchor=north] at (.5,1.8) {};
        \draw[line width=0.4mm,  black] (2, 1) rectangle (3, 2);
    \end{tikzpicture}
    \label{hodlr_interaction_3}
    }\quad
    \subfloat{
        \begin{tikzpicture}
            [
            box/.style={rectangle,draw=black, minimum size=0.1cm},scale=0.1
            ]
            \node[box,fill=red,,font=\tiny,label=right: \scriptsize Cluster considered,  anchor=west] at (-8,8){};
            \node[box,fill=matlabblue,,font=\tiny,label=right:\scriptsize Adjacent cluster,  anchor=west] at (-8,11){};
        \end{tikzpicture}
    }
    \caption{Interaction list under the weak admissibility condition at different levels in $2$D (quad-tree). Lighter shade colors denote the interaction list at the coarser levels.}
    \label{fig:hodlr_interaction}
\end{center}
\end{figure}
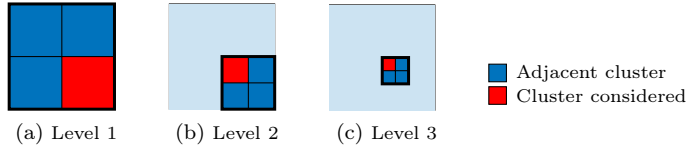

    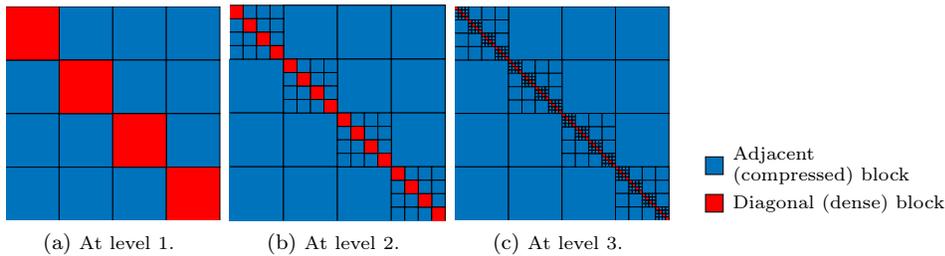
\begin{figure}[H]
        \centering
        \subfloat[\scriptsize At level $1$.]{
        \includegraphics[scale=.2]{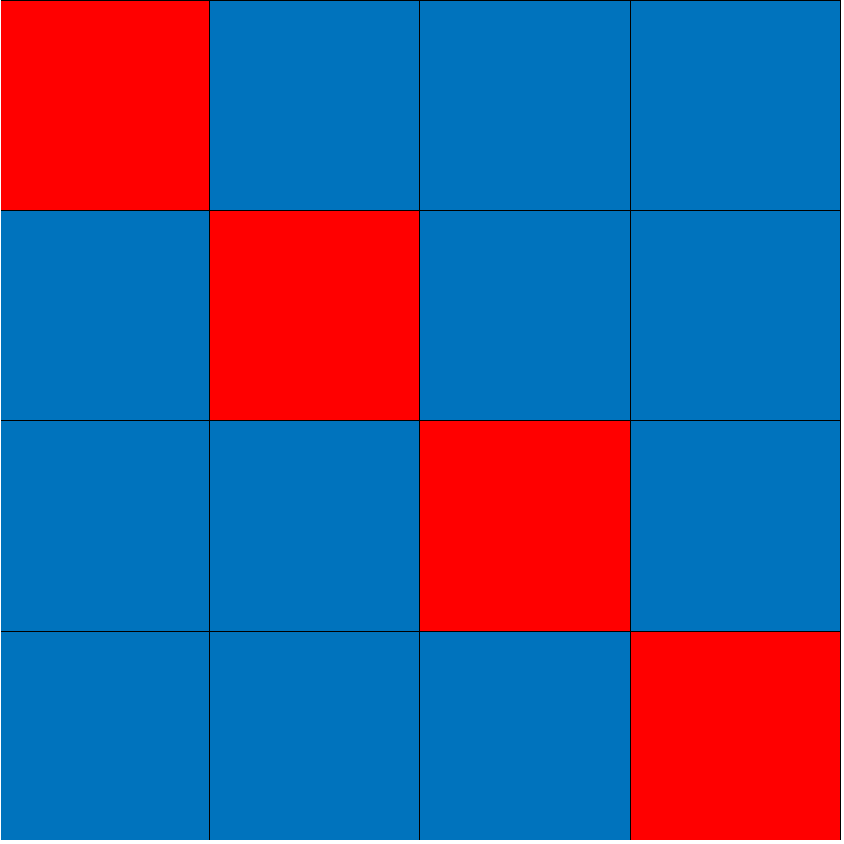}}
        \subfloat[\scriptsize  At level $2$.]{
        \includegraphics[scale=.205]{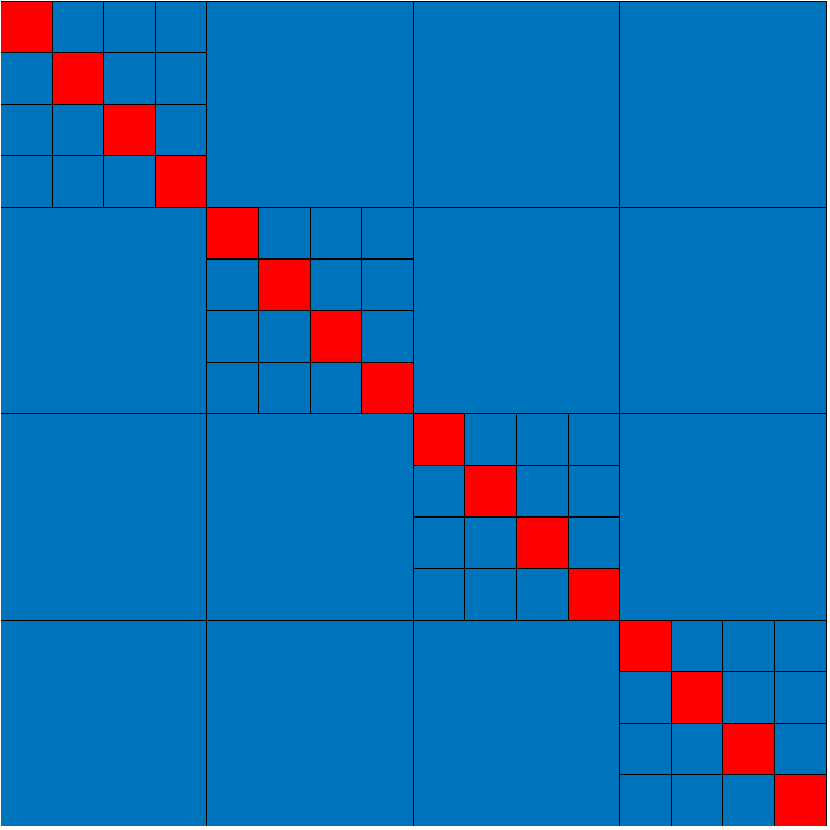}}
        \subfloat[\scriptsize  At level $3$.]{
        \includegraphics[scale=.2]{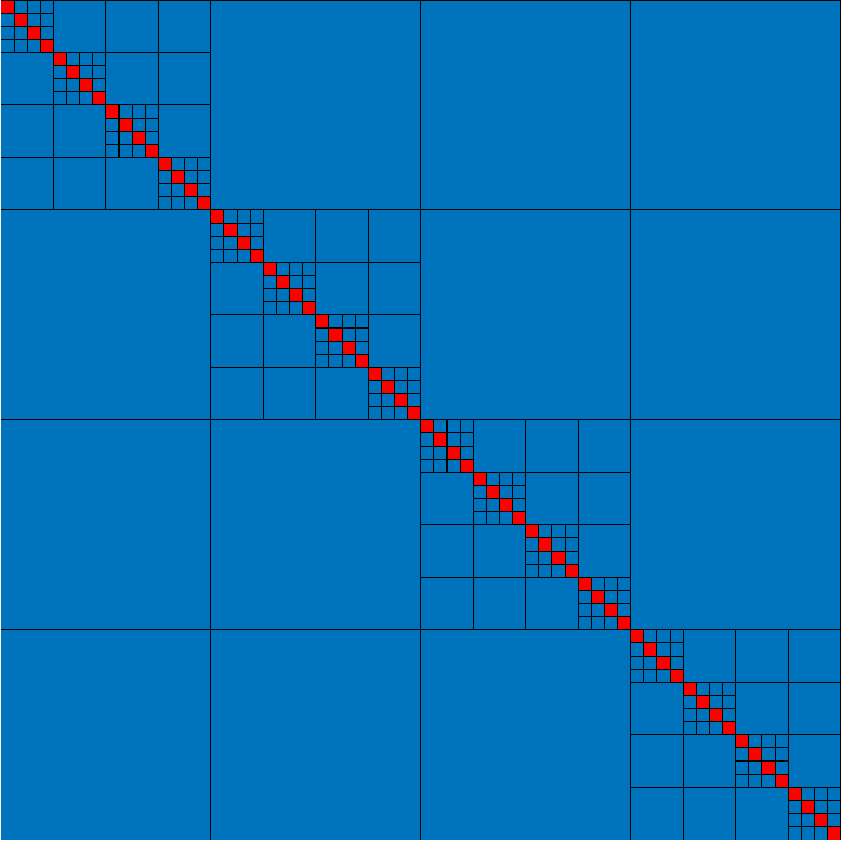}} \quad
        \subfloat{
        \begin{tikzpicture}
            [
            box/.style={rectangle,draw=black, minimum size=0.2cm},scale=0.1
            ]
            \node[box,fill=red,,font=\tiny,label=right:\scriptsize Diagonal (dense) block,  anchor=west] at (8,0){};
            \node[box,fill=matlabblue,,font=\tiny,label={right:\scriptsize \makecell[l]{Adjacent\\(compressed) block}},  anchor=west] at (8,5){};
        \end{tikzpicture}
        }
        \caption{HODLR matrix at different levels in $2$D.}
        \label{fig:hodlr_2d_figs}
    \end{figure}

\cref{fig:hodlr_interaction} and \cref{fig:hodlr_2d_figs} illustrate the interaction list of a cluster and the corresponding HODLR matrix representation, respectively.

The size of the interaction list of a cluster at any level in $d$ dimensions is

\begin{align} \label{eq:weak_IL_size}
   \abs{\mathcal{IL}_{w} \bkt{\mathcal{C}_i}} = \bkt{2^d-1} = C^{\prime\prime\prime}.
\end{align}

\subsection{Storage comparison of \texorpdfstring{$\mathcal{H}_s$}{Hs} and HODLR matrices} \label{subsec:H_HODLR}
Let $\mathcal{C}_i^{\bkt{l}}$ and $\mathcal{C}_j^{\bkt{l}}$ be admissible clusters at level $l$, and let $p_{ij}^{\bkt{l}}$ represent the numerical rank of the admissible block matrix for a target accuracy $\epsilon$. The admissible block is stored in low-rank format, and its storage requirement is given by 

\begin{align} \label{eq:LR_storage}
    p_{ij}^{\bkt{l}} (\abs{\mathcal{C}_i^{\bkt{l}}} + \abs{\mathcal{C}_j^{\bkt{l}}}).
\end{align}
The non-admissible blocks are stored in dense matrix format. The total storage cost is the sum of low-rank storage and dense storage requirements.
The storage cost of $\mathcal{H}_s$-matrix representation is thus given by
 \begin{equation} \label{eq:hs_storage}
 \resizebox{11.9cm}{!}{$
\begin{aligned}
  \mathcal{S}_{\mathcal{H}_s} = \overbrace{\dsum_{l=1}^{L} \dsum_{i=1}^{2^{dl}} \dsum_{\mathcal{C}_j^{\bkt{l}} \in \mathcal{IL}_s (\mathcal{C}_i^{\bkt{l}})} {p}_{ij}^{\prime \bkt{l}} (\abs{\mathcal{C}_i^{\bkt{l}}} + \abs{\mathcal{C}_j^{\bkt{l}}})}^{\text{\scriptsize Low-rank storage}} + \overbrace{\dsum_{i=1}^{2^{d L}} \dsum_{\mathcal{C}_j^{\bkt{L}} \in \mathcal{N}_s (\mathcal{C}_i^{\bkt{L}})} \abs{\mathcal{C}_i^{\bkt{L}}} \abs{\mathcal{C}_j^{\bkt{L}}}  + \dsum_{i=1}^{2^{d L}}  \abs{\mathcal{C}_i^{\bkt{L}}} \abs{\mathcal{C}_i^{\bkt{L}}}}^{\text{\scriptsize Dense storage}},  
\end{aligned}
$}
\end{equation}
and the storage cost of the HODLR matrix representation is given by
\begin{align} \label{eq:hodlr_storage}
  \mathcal{S}_{HODLR} = \overbrace{\dsum_{l=1}^{L} \dsum_{i=1}^{2^{dl}} \dsum_{\mathcal{C}_j^{\bkt{l}} \in \mathcal{IL}_w (\mathcal{C}_i^{\bkt{l}})} {p}_{ij}^{\prime\prime\prime\bkt{l}} (\abs{\mathcal{C}_i^{\bkt{l}}} +\abs{\mathcal{C}_j^{\bkt{l}}})}^{\text{\scriptsize Low-rank storage}} +  \overbrace{\dsum_{i=1}^{2^{d L}}  \abs{\mathcal{C}_i^{\bkt{L}}} \abs{\mathcal{C}_i^{\bkt{L}}}}^{\text{\scriptsize Dense storage}}.
\end{align}

The storage cost of the $\mathcal{H}$-matrix representation of $H \in \mathbb{R}^{N \times N}$ is $\mathcal{O} \bkt{p N \log \bkt{N}}$, where $p$ is the maximum numerical rank of the admissible blocks. For $\mathcal{H}_s$-matrices, $p$ scales roughly as $\mathcal{O} \bkt{\log^d \bkt{1/\epsilon}}$, since the admissible clusters are the far-field (well-separated) clusters. In contrast, for HODLR matrices, $p$ scales roughly as $\mathcal{O} \bkt{N^{(d-1)/d} \log \bkt{N/\epsilon}}$ \cite{khan2022numerical}. Hence, HODLR matrix does not achieve truly \emph{quasi-linear} complexity and is computationally less efficient than $\mathcal{H}_s$-matrix for large $N$ in higher dimensions $(d>1)$.

As a motivating example, we perform the following numerical experiment in two dimensions to compare the storage requirements of $\mathcal{H}_s$ and HODLR matrices for different target accuracies $\epsilon$.

    \begin{figure}[H]
        \centering
        \subfloat[\scriptsize  $N = 262144$ and $n_{\max} = 256$ $(L=5)$.]{
        \includegraphics[scale=.34]{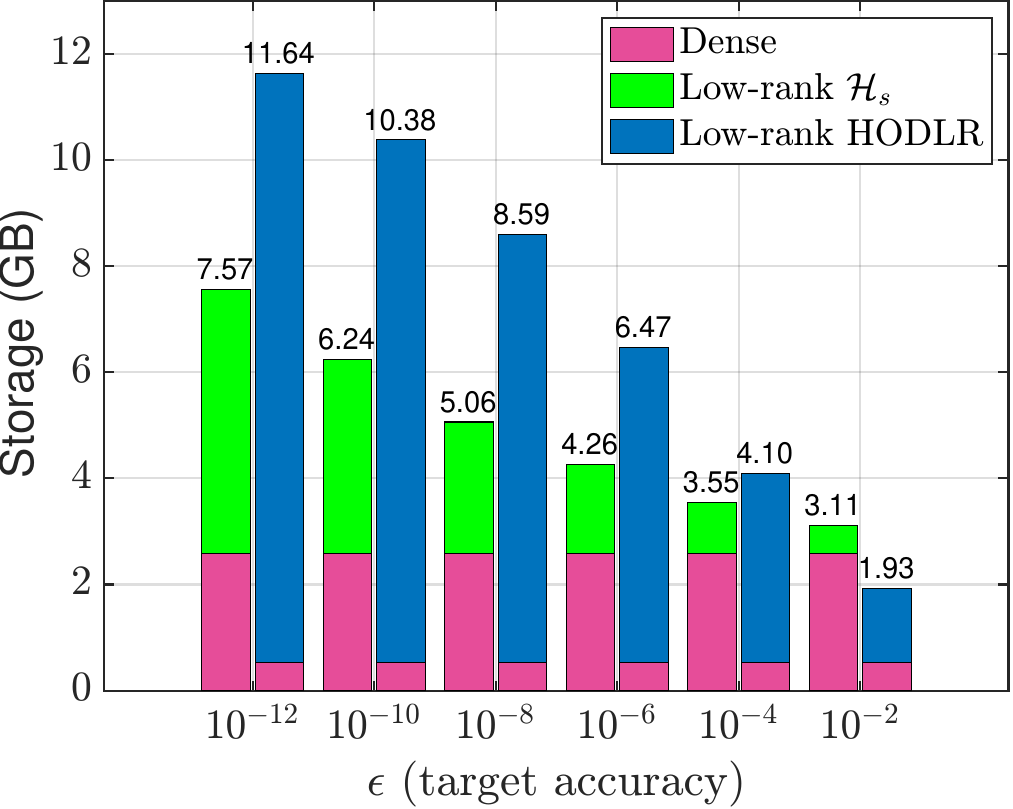}\label{fig:HODLR_H_comp_mem2}}
        \subfloat[\scriptsize  $N = 262144$ and $n_{\max} = 64$ $(L=6)$.]{
        \includegraphics[scale=.34]{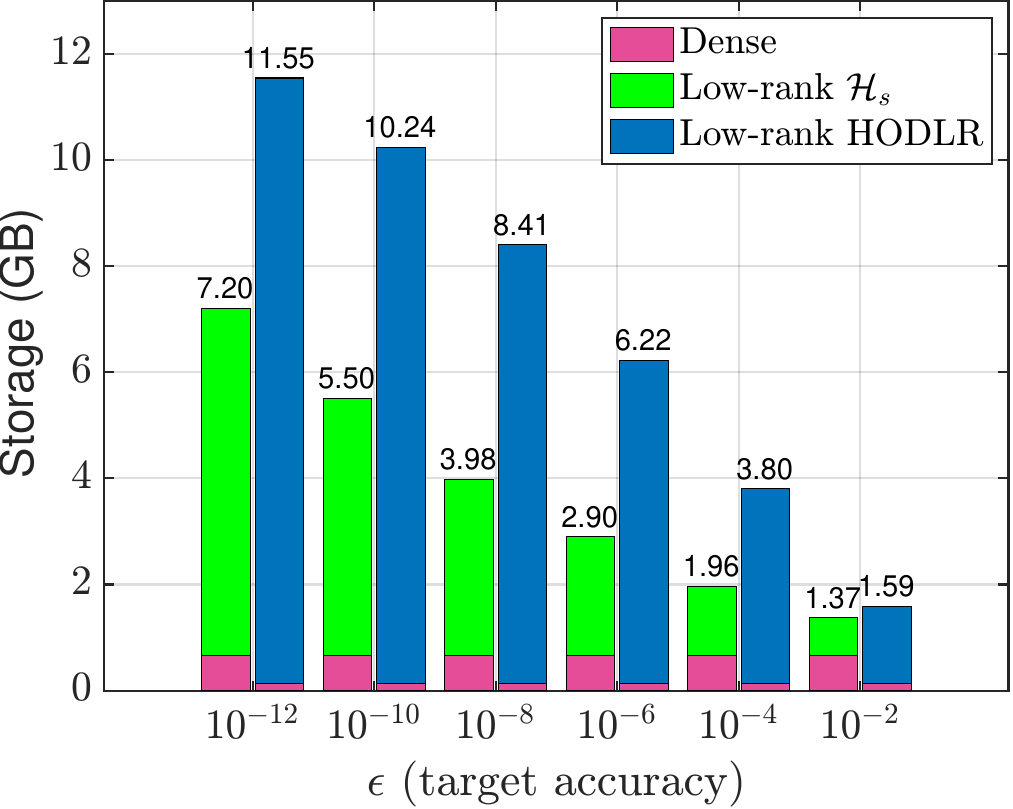}\label{fig:HODLR_H_comp_mem1}}
        \caption{Comparison of the storage costs of $\mathcal{H}_s$ and HODLR matrices for the kernel $1/r$ in $2$D.}
        \label{fig:HODLR_H_mem_comp}
    \end{figure}
Let $N=262144$ points be uniformly distributed inside the square domain $[-1,1]^2$. The dense kernel matrix $H \in \mathbb{R}^{N \times N}$ is formed using the kernel function $1/r$ as described in \cref{eq:kernel_mat}. We construct the $\mathcal{H}_s$ and HODLR representations of the dense kernel matrix $H$ for different values of target accuracy $\epsilon$ and $L$. The corresponding storage costs are compared in \Cref{fig:HODLR_H_mem_comp}. \Cref{fig:HODLR_H_mem_comp} shows that the HODLR matrix requires more low-rank storage than the $\mathcal{H}_{s}$-matrix, due to its higher numerical ranks. The difference increases as $N$ grows. However, the HODLR matrix has an advantage that dense blocks are confined to the diagonal, whereas the $\mathcal{H}_s$-matrix contains dense blocks along both the diagonal and neighboring blocks (cf. \cref{eq:hs_storage,eq:hodlr_storage}). Thus, the HODLR matrix requires less dense storage.

\section{Hybrid hierarchical matrices} \label{sec:hybrid_hmat}
Motivated by the previous experiment, we introduce a novel hybrid admissibility condition that exploits both standard and weak admissibility conditions. We then discuss the resulting hybrid hierarchical matrices, namely $\mathcal{H}_h$-matrices, and derive a criterion under which $\mathcal{H}_h$-matrices require less storage than $\mathcal{H}_s$-matrices.

\subsection{Hybrid admissibility condition}
The only disadvantage of $\mathcal{H}_s$-matrices compared to HODLR matrices is that they have a higher number of admissible blocks (cf. \cref{eq:standard_IL_size} and \cref{eq:weak_IL_size}) and more dense blocks at the leaf level. The rapidly growing number of admissible blocks can become a computational bottleneck for multidimensional problems as the tree depth increases. To mitigate this drawback of the $\mathcal{H}_s$-matrices, while keeping the leaf-level block size (or the tree depth $L$) unchanged, we exploit the weak admissibility condition at finer levels, where the admissible blocks have relatively low numerical rank due to their smaller sizes.

We propose a hybrid admissibility condition that uses standard admissibility at \emph{coarser levels}, where blocks are \textbf{large}, and weak admissibility at \emph{finer levels}, where blocks are relatively \textbf{small}. This approach reduces the number of admissible clusters at the \emph{finer levels}, and self-clusters are the only inadmissible clusters at the leaf level. 

Let $\mathcal{C}_i$ and $\mathcal{C}_j$ be a pair of clusters at level $l$, and let $\ell < L$ denote an intermediate level of the tree, which we referred to as the switching level. We define the hybrid admissibility condition as follows:

\begin{align} \label{eq:hybrid_admiss}
       \text{adm}_h \bkt{\mathcal{C}_i, \mathcal{C}_j} = \text{true} \iff
\begin{cases}
    \text{adm}_s \bkt{\mathcal{C}_i, \mathcal{C}_j} = \text{true} ,& \text{if } 1 \leq l \leq \ell \\
    \text{adm}_s \bkt{\mathcal{C}_i, \mathcal{C}_j} = \text{false} \text{ and } i \neq j,& \text{if } l = \ell \\
    \text{adm}_w \bkt{\mathcal{C}_i, \mathcal{C}_j} = \text{true} ,& \text{if } \ell +1 \leq l \leq L.
\end{cases}
\end{align}

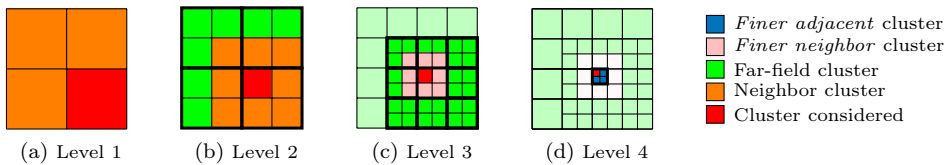
\begin{figure}[H]
\begin{center}
\subfloat[\scriptsize Level $1$]{
    \begin{tikzpicture}[scale=0.8]
        \fill [red] (1,0) rectangle (2,1);
        \fill [orange] (0,0) rectangle (1,1);
        \fill [orange] (1,1) rectangle (2,2);
        \fill [orange] (0,1) rectangle (1,2);
        \draw[step=1cm, black] (0, 0) grid (2, 2);
    \end{tikzpicture}
    \label{hybrid_interaction_1}
    }\quad 
    \subfloat[\scriptsize Level $2$]{
    \begin{tikzpicture}[scale=0.4]
        \draw[black] (0, 0) grid (4, 4);
        \fill [green] (0,0) rectangle (1,1);
        \fill [green] (0,1) rectangle (1,2);
        \fill [green] (0,2) rectangle (1,3);
        \fill [green] (0,3) rectangle (1,4);
        \fill [green] (1,3) rectangle (2,4);
        \fill [green] (2,3) rectangle (3,4);
        \fill [green] (3,3) rectangle (4,4);
        \fill [red] (2,1) rectangle (3,2);
        \fill [orange] (1,0) rectangle (2,1);
        \fill [orange] (1,1) rectangle (2,2);
        \fill [orange] (1,2) rectangle (2,3);
        \fill [orange] (2,2) rectangle (3,3);
        \fill [orange] (3,2) rectangle (4,3);
        \fill [orange] (2,0) rectangle (3,1);
        \fill [orange] (3,0) rectangle (4,1);
        \fill [orange] (3,1) rectangle (4,2);
        \node[anchor=north] at (1.5,2.8) {};
        \node[anchor=north] at (1.5,1.8) {};
        \node[anchor=north] at (.5,1.8) {};
        \draw[black] (0, 0) grid (4, 4);
        \draw[line width=0.4mm,  black] (0, 0) rectangle (2, 2);
        \draw[line width=0.4mm,  black] (2, 2) rectangle (4, 4);
        \draw[line width=0.4mm,  black] (2, 0) rectangle (4, 2);
        \draw[line width=0.4mm,  black] (0, 2) rectangle (2, 4);
    \end{tikzpicture}
    \label{hybrid_interaction_2}
    }\quad 
    \subfloat[\scriptsize Level $3$]{
    \begin{tikzpicture}[scale=0.4]
        \fill [green!25] (0,0) rectangle (1,1);
        \fill [green!25] (0,1) rectangle (1,2);
        \fill [green!25] (0,2) rectangle (1,3);
        \fill [green!25] (0,3) rectangle (1,4);
        \fill [green!25] (1,3) rectangle (2,4);
        \fill [green!25] (2,3) rectangle (3,4);
        \fill [green!25] (3,3) rectangle (4,4);
        \node[anchor=north] at (1.5,2.8) {};
        \node[anchor=north] at (1.5,1.8) {};
        \node[anchor=north] at (.5,1.8) {};
        \draw[black] (0, 0) grid (4, 4);
        \draw[line width=0.4mm,  fill=green] (1, 0) rectangle (2, 1);
        \draw[line width=0.4mm,  fill=green] (1, 1) rectangle (2, 2);
        \draw[line width=0.4mm,  fill=green] (2, 2) rectangle (3, 3);
        \draw[line width=0.4mm,  fill=green] (3, 2) rectangle (4, 3);
        \draw[line width=0.4mm,  fill=green] (2, 0) rectangle (3, 1);
        \draw[line width=0.4mm,  fill=green] (3, 0) rectangle (4, 1);
        \draw[line width=0.4mm,  fill=green] (3, 1) rectangle (4, 2);
        \draw[line width=0.4mm,  fill=green] (1, 2) rectangle (2, 3);
        \fill [pink] (1.5,1) rectangle (3,2.5);
        \fill [red] (2,1.5) rectangle (2.5,2);
        \draw[step=0.5cm, black] (1, 0) grid (4, 3);
        \draw[line width=0.4mm] (1, 0) rectangle (2, 1);
        \draw[line width=0.4mm] (1, 1) rectangle (2, 2);
        \draw[line width=0.4mm] (2, 2) rectangle (3, 3);
        \draw[line width=0.4mm] (3, 2) rectangle (4, 3);
        \draw[line width=0.4mm] (2, 0) rectangle (3, 1);
        \draw[line width=0.4mm] (3, 0) rectangle (4, 1);
        \draw[line width=0.4mm] (3, 1) rectangle (4, 2);
        \draw[line width=0.4mm] (1, 2) rectangle (2, 3);
    \end{tikzpicture}
    \label{hybrid_interaction_3}
    }\quad
    \subfloat[\scriptsize Level $4$]{
    \begin{tikzpicture}[scale=0.4]
        \fill [green!25] (0,0) rectangle (1,1);
        \fill [green!25] (0,1) rectangle (1,2);
        \fill [green!25] (0,2) rectangle (1,3);
        \fill [green!25] (0,3) rectangle (1,4);
        \fill [green!25] (1,3) rectangle (2,4);
        \fill [green!25] (2,3) rectangle (3,4);
        \fill [green!25] (3,3) rectangle (4,4);
        \node[anchor=north] at (1.5,2.8) {};
        \node[anchor=north] at (1.5,1.8) {};
        \node[anchor=north] at (.5,1.8) {};
        \draw[black] (0, 0) grid (4, 4);
        \fill [fill=green!25] (1, 0) rectangle (2, 1);
        \fill [fill=green!25] (1, 1) rectangle (2, 2);
        \fill [fill=green!25] (2, 2) rectangle (3, 3);
        \fill [fill=green!25] (3, 2) rectangle (4, 3);
        \fill [fill=green!25] (2, 0) rectangle (3, 1);
        \fill [fill=green!25] (3, 0) rectangle (4, 1);
        \fill [fill=green!25] (3, 1) rectangle (4, 2);
        \fill [fill=green!25] (1, 2) rectangle (2, 3);
        \fill [pink!10] (1.5,1) rectangle (3,2.5);
        \fill [red] (2,1.5) rectangle (2.5,2);
        \draw[step=0.5cm, black] (1, 0) grid (4, 3);
        \fill [matlabblue] (2.25,1.75) rectangle (2.5,2);
        \fill [matlabblue] (2,1.5) rectangle (2.25,1.75);
        \fill [matlabblue] (2.25,1.5) rectangle (2.5,1.75);
        \draw[line width=0.4mm,  black] (2, 1.5) rectangle (2.5, 2);
        \draw[black] (0, 2) rectangle (2, 4);
        \draw[step=1cm, black] (0, 0) grid (2, 2);
        \draw[black] (0, 0) grid (4, 4);
        \draw[step=0.25cm, black] (2, 1.5) grid (2.5, 2);
    \end{tikzpicture}
    \label{hybrid_interaction_4}
    }\quad
    \subfloat{
        \begin{tikzpicture}
            [
            box/.style={rectangle,draw=black, minimum size=0.1cm},scale=0.1
            ]
            \node[box,fill=red,,font=\tiny,label=right: \scriptsize Cluster considered,  anchor=west] at (-8,8){};
            \node[box,fill=orange,,font=\tiny,label=right:\scriptsize Neighbor cluster,  anchor=west] at (-8,11){};
            \node[box,fill=green,,font=\tiny,label=right:\scriptsize Far-field cluster,  anchor=west] at (-8,14){};
            \node[box,fill=pink,,font=\tiny,label=right:\scriptsize \emph{Finer neighbor} cluster,  anchor=west] at (-8,17){};
            \node[box,fill=matlabblue,,font=\tiny,label=right:\scriptsize \emph{Finer adjacent} cluster,  anchor=west] at (-8,20){};
        \end{tikzpicture}
    }
    \caption{Interaction list under the hybrid admissibility condition $(\eta = \sqrt{2})$ when $\ell=3$ and $L=4$. At level $3$, finer neighbor clusters are considered admissible. The admissible clusters are highlighted in deep green, pink and blue colored clusters enclosed within a noticeable \textbf{black} border.}
    \label{fig:hybrid_interaction}
\end{center}
\end{figure}

\Cref{fig:hybrid_interaction} illustrates the admissible clusters at different levels under the hybrid admissibility condition with $\eta = \sqrt{2}$ in two dimensions, for $\ell = 3$ and $L = 4$.

It follows from \cref{eq:hybrid_admiss} that the admissible clusters can be obtained as follows:
\begin{enumerate}
    \item The standard admissibility condition is applied from level $1$ to level $\ell$ of the tree (see \textbf{deep green} clusters in \Cref{hybrid_interaction_1,hybrid_interaction_2,hybrid_interaction_3}).
    \item At level $\ell$, the \emph{finer neighbor} clusters resulting from the standard admissibility condition are treated as admissible (see \textbf{deep pink} clusters in \Cref{hybrid_interaction_3}).
    \item The weak admissibility condition is applied from level $\ell + 1$ to level $L$ of the tree (see \textbf{deep blue} clusters in \Cref{hybrid_interaction_4}).
\end{enumerate}

\subsection{\texorpdfstring{$\mathcal{H}_h$}{Hh}-matrices}
Hierarchical matrices based on \cref{eq:hybrid_admiss} are referred to as hybrid hierarchical matrices and denoted by $\mathcal{H}_h$-matrices. We present the construction of $\mathcal{H}_h$-matrix representation in \Cref{alg:algo1}. It takes as input the matrix $H$, tree $\mathcal{T}^L$, the switching level $\ell$, and the target accuracy $\epsilon$. It returns the hybrid hierarchical representation $\widetilde{H}$ as output. 

 \begin{algorithm}
 \small
	\caption{$\mathcal{H}_h$-matrix representation.}\label{alg:algo1}
	\begin{algorithmic}[1]
        \State \textbf{Input:} $H \in \Rb^{N \times N}$, $\mathcal{T}^L$, $0 \leq \ell \leq L$, $\epsilon$. \qquad \textbf{Output:} $\widetilde{H}$.
		\State {\emph{Standard admissibility condition is applied levels} $1 : \ell$.}
\For{\texttt{$l=1:\ell$}} 
				\For{\texttt{$i=1:2^{dl}$}}
					\For{$\mathcal{C}^{\bkt{l}}_{j} \in$ $\mathcal{IL}_s (\mathcal{C}^{\bkt{l}}_{i})$}
                        \State $\big[\widetilde{U}^{\bkt{l}}_I, \widetilde{V}^{\bkt{l}}_J  \big] = \texttt{LRcompression} (H^{\bkt{l}}_{I,J}, \epsilon)$
                        \Comment{$I \gets$ index set of $\mathcal{C}^{\bkt{l}}_i$ and $J \gets$ index set of $\mathcal{C}^{\bkt{l}}_j$. Low-rank compression with target accuracy $\epsilon$.}
                        \State $\widetilde{H}^{\bkt{l}}_{I,J} \gets \widetilde{U}^{\bkt{l}}_I (\widetilde{V}^{\bkt{l}}_J)^*$
					\EndFor
				\EndFor
			\EndFor
                \If{$\ell == L$} \Comment{If $\ell=L$, it reduces to $\mathcal{H}_s$-matrix.}
              		\For{\texttt{$i=1:2^{d\ell}$}}
    					\For{$\mathcal{C}^{\bkt{\ell}}_{j} \in$ $\mathcal{N}_s (\mathcal{C}^{\bkt{\ell}}_{i})$}
                            \State $\widetilde{H}^{\bkt{\ell}}_{I,J} \gets H^{\bkt{\ell}}_{I,J}$ \Comment{Dense neighbor blocks at leaf level.}
    					\EndFor
    				\EndFor  
                \Else
    					\For{$\mathcal{C}^{\bkt{\ell}}_{j} \in$ $\mathcal{N}_s (\mathcal{C}^{\bkt{\ell}}_{i})$}
                            \State $\big[\widetilde{U}^{\bkt{\ell}}_I, \widetilde{V}^{\bkt{\ell}}_J  \big] = \texttt{LRcompression} (H^{\bkt{\ell}}_{I,J}, \epsilon)$, \quad $\widetilde{H}^{\bkt{\ell}}_{I,J} \gets \widetilde{U}^{\bkt{\ell}}_I (\widetilde{V}^{\bkt{\ell}}_J)^*$
                            \EndFor
                \EndIf

            \State{\emph{Weak admissibility condition is applied levels} $\ell+1 : L$.}
            \For{\texttt{$l=\ell+1:L$}} 
				\For{\texttt{$i=1:2^{dl}$}}
					\For{$\mathcal{C}^{\bkt{l}}_{j} \in$ $\mathcal{IL}_w (\mathcal{C}^{\bkt{l}}_{i})$}
                        \State $\big[\widetilde{U}^{\bkt{l}}_I, \widetilde{V}^{\bkt{l}}_J  \big] = \texttt{LRcompression} (H^{\bkt{l}}_{I,J}, \epsilon)$, \quad $\widetilde{H}^{\bkt{l}}_{I,J} \gets \widetilde{U}^{\bkt{l}}_I (\widetilde{V}^{\bkt{l}}_J)^*$
					\EndFor
				\EndFor
			\EndFor

            \State{\emph{Dense diagonal blocks at leaf level}.}
            \For{\texttt{$i=1:2^{dL}$}}
                \State $\widetilde{H}^{\bkt{L}}_{I,I} \gets H^{\bkt{L}}_{I,I}$ \Comment{$I \gets$ index set of $\mathcal{C}^{\bkt{L}}_i$}
            \EndFor
	\end{algorithmic}
\end{algorithm}

In the implementation, a dense matrix $H$ is not stored explicitly; instead, a function is provided that evaluates the $(i,j)^{th}$ entry of $H$. The admissible blocks are generated on the fly and compressed with the prescribed accuracy $\epsilon$.

From \Cref{alg:algo1}, we can see that when $\ell < L$, dense blocks are confined along the diagonal.

\subsubsection*{Special cases}
   As special cases of $\mathcal{H}_h$-matrices, purely $\mathcal{H}_s$-matrices and purely HODLR matrices can be obtained easily as follows:
   \begin{itemize}
       \item Setting $\ell = L$ in \Cref{alg:algo1} generates the $\mathcal{H}_s$-matrices (see lines $3-16$ \& lines $31-33$ of \Cref{alg:algo1}).
       \item Setting $\ell = 0$ in \Cref{alg:algo1} generates the HODLR matrices (see lines $23-33$ of \Cref{alg:algo1}). 
   \end{itemize}

$\mathcal{H}_h$-matrices corresponding to different choices of $\ell$ are illustrated in \Cref{fig:hmat_figs}. A \textbf{zoomed-in view} (see online color version) of \Cref{fig:hmat_figs} shows the $\mathcal{H}_s$-matrix ($\ell=L=4$) and $\mathcal{H}_h$-matrix with $\ell (= 3) < L (= 4)$. If $\ell = 3$ and $L = 4$, the far-field blocks from levels 1 through 3 are identified using the standard admissibility condition (green blocks in \Cref{fig:hybrid_hmat}), while the finer neighbor blocks are compressed at level 3 (pink blocks in \Cref{fig:hybrid_hmat}). At level 4, off-diagonal blocks are compressed according to the weak admissibility condition (blue blocks in \Cref{fig:hybrid_hmat}). \Cref{fig:strong_hmat} illustrates $\mathcal{H}_s$-matrix (orange and red are dense blocks at leaf level). Comparing \Cref{fig:hybrid_hmat} and \Cref{fig:strong_hmat} level by level helps in understanding how the construction is done.

    \begin{figure}[H]
        \centering
        \subfloat[\scriptsize  $\mathcal{H}_h$-matrix with $\ell=3$ and $L=4$.]{
        \includegraphics[scale=.44]{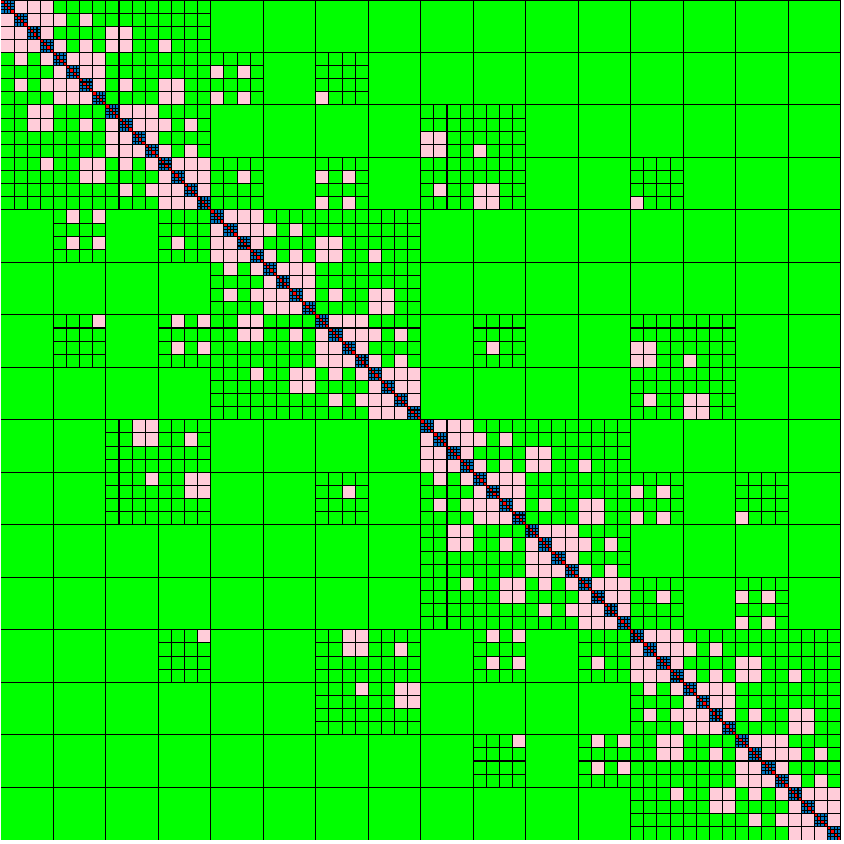}\label{fig:hybrid_hmat}}
        \subfloat[\scriptsize  $\mathcal{H}_h$-matrix with $\ell=L=4$, i.e., $\mathcal{H}_{s}$-matrix.]{
        \includegraphics[scale=.44]{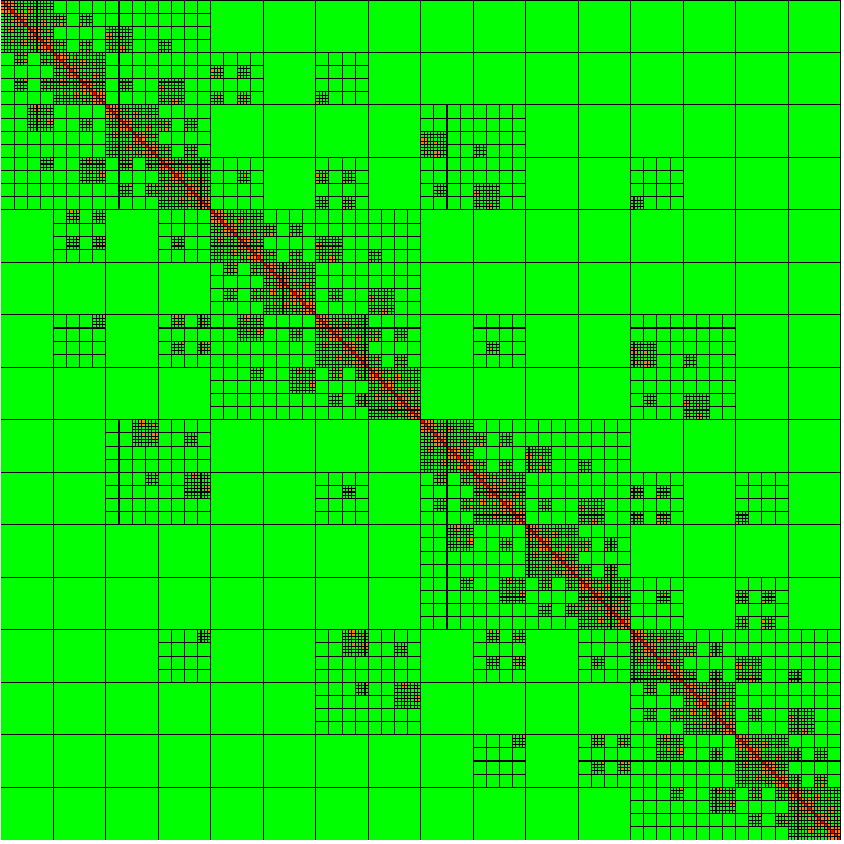}\label{fig:strong_hmat}} \quad
        \subfloat{
        \begin{tikzpicture}
            [
            box/.style={rectangle,draw=black, minimum size=0.1cm},scale=0.1
            ]
            \node[box,fill=green,,font=\tiny,label=right:\scriptsize Far-field,  anchor=west] at (-60,0){};
            \node[box,fill=pink,,font=\tiny,label=right:\scriptsize \emph{Finer neighbor},  anchor=west] at (-44,0){};
            \node[box,fill=matlabblue,,font=\tiny,label=right:\scriptsize \emph{Finer adjacent},  anchor=west] at (-20,0){};
            \node[box,fill=orange,,font=\tiny,label=right:\scriptsize Neighbor (only for $\mathcal{H}_s$-matrix),  anchor=west] at (5,0){};
            \node[box,fill=red,,font=\tiny,label=right: \scriptsize Diagonal,  anchor=west] at (50,0){};
        \end{tikzpicture}
    }
        \caption{A \textbf{zoomed-in view} (in online version) shows how the $\mathcal{H}_h$-matrices are constructed for $\ell < L$ and $\ell=L$. The green, pink and blue colored blocks are compressed in \Cref{fig:hybrid_hmat}.}
        \label{fig:hmat_figs}
    \end{figure}
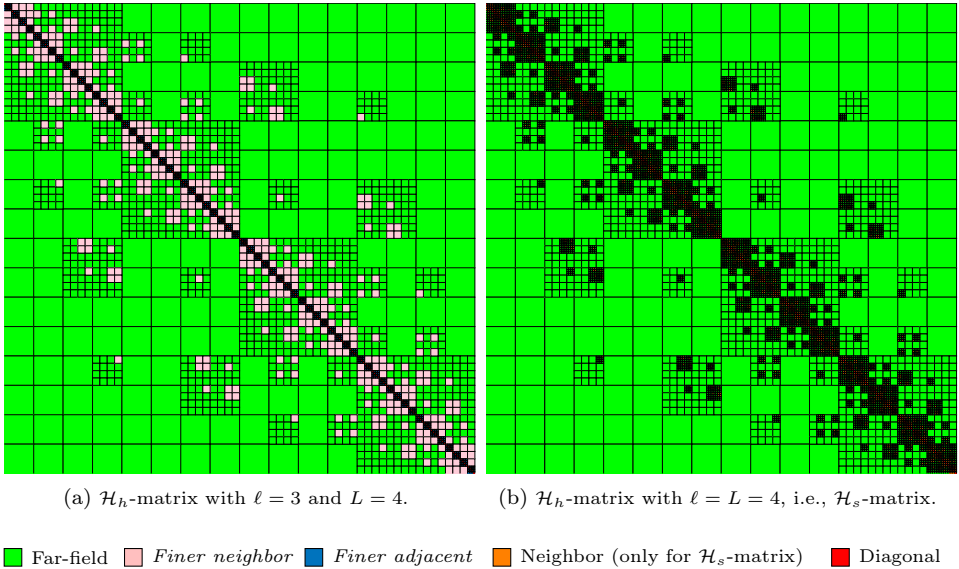

\subsection{Storage comparison of \texorpdfstring{$\mathcal{H}_s$}{Hs} and \texorpdfstring{$\mathcal{H}_h$}{Hh} matrices} \label{subsec:Hs_vs_Hh}
In the $\mathcal{H}_h$-matrix representation, the standard admissibility condition is applied from level $1$ to level $\ell$. Let $\mathscr{L}$  denote the total low-rank storage of the admissible blocks selected based on the standard admissibility condition up to level $\ell$, and let $\mathscr{D}$ denote the total storage of the leaf-level dense diagonal blocks. 

The storage cost of $\mathcal{H}_s$-matrix representation is given by

\begin{align} \label{eq:storage_Hs}
    \mathcal{S}_s = \mathscr{L} + S_1 \bkt{\ell} + \mathscr{D}, \qquad \text{ where}
\end{align}

\begin{equation} \label{eq:S1}
\resizebox{11.9cm}{!}{$
\begin{aligned}
  S_1 \bkt{\ell} &= \overbrace{\dsum_{l=\ell+1}^{L} \dsum_{i=1}^{2^{dl}} \dsum_{\mathcal{C}_j^{\bkt{l}} \in \mathcal{IL}_s (\mathcal{C}_i^{\bkt{l}})} {p^{\prime}}_{ij}^{\bkt{l}} (\abs{\mathcal{C}_i^{\bkt{l}}} + \abs{\mathcal{C}_j^{\bkt{l}}})}^{\text{\scriptsize scales as } \mathcal{O} (C^{\prime} p^{\prime} N \bkt{L-\ell})} + \overbrace{\dsum_{i=1}^{2^{d L}} \dsum_{\mathcal{C}_j^{\bkt{L}} \in \mathcal{N}_s (\mathcal{C}_i^{\bkt{L}})} \abs{\mathcal{C}_i^{\bkt{L}}} \abs{\mathcal{C}_j^{\bkt{L}}}}^{\text{\scriptsize scales as } \mathcal{O} (C^{\prime\prime} N n_{\max}^2)}\\ & \equiv S_1^{\prime} + S_1^{\prime \prime}. 
\end{aligned}
$}
\end{equation}

The storage cost of the $\mathcal{H}_h$-matrix representation \textbf{with $\ell<L$} is given by

\begin{align} \label{eq:storage_Hh}
    \mathcal{S}_h = \mathscr{L} + S_2 \bkt{\ell} + \mathscr{D}, \qquad \text{ where}
\end{align} 

\begin{equation} \label{eq:S2}
\resizebox{11.9cm}{!}{$
\begin{aligned}
S_2\bkt{\ell} &= 
\overbrace{
\sum_{i=1}^{2^{d \ell}} 
\sum_{\mathcal{C}_j^{(\ell)} \in \mathcal{N}_s(\mathcal{C}_i^{(\ell)})} 
p_{ij}^{\prime\prime (\ell)} (\abs{\mathcal{C}_i^{(\ell)}} + \abs{\mathcal{C}_j^{(\ell)}})
}^{\text{\scriptsize scales as } \mathcal{O}\bkt{C^{\prime \prime} p^{\prime\prime} N}} + 
\overbrace{
\sum_{l=\ell+1}^{L} 
\sum_{i=1}^{2^{dl}} 
\sum_{\mathcal{C}_j^{(l)} \in \mathcal{IL}_w(\mathcal{C}_i^{(l)})} 
p_{ij}^{\prime\prime\prime (l)} (\abs{\mathcal{C}_i^{(l)}} + \abs{\mathcal{C}_j^{(l)}})
}^{\text{\scriptsize scales as } \mathcal{O}\bkt{C^{\prime\prime\prime} p^{\prime\prime\prime} N (L-\ell)}} 
 \\ 
 & \equiv S_2^{\prime \prime} + S_2^{\prime \prime \prime}.
\end{aligned}
$}
\end{equation}

From \cref{eq:storage_Hs} and \cref{eq:storage_Hh}, it follows that the terms $\mathscr{L}$ and $\mathscr{D}$ are common to both $\mathcal{S}_s$ and $\mathcal{S}_h$. The only difference lies in the terms $S_1\bkt{\ell}$ and $S_2\bkt{\ell}$. Deriving exact closed-form expressions for $S_1\bkt{\ell}$ and $S_2\bkt{\ell}$ is difficult. Instead, we show their asymptotic behavior by specifying the scaling of each contributing term. The constants $C^{\prime}$, $C^{\prime\prime}$, and $C^{\prime\prime\prime}$ are defined in \cref{eq:standard_IL_size}, \cref{eq:standard_N_size}, and \cref{eq:weak_IL_size}, respectively. The maximum numerical ranks $p^{\prime}$, $p^{\prime\prime}$, and $p^{\prime\prime\prime}$ are defined as follows:

\begin{align*}
    p^{\prime} = \max_{\substack{\ell+1 \leq l \leq L \\ i,j}} {p^{\prime}}_{ij}^{\bkt{l}}, \quad p^{\prime\prime} = \max_{\substack{l = \ell \\ i,j}} {p}_{ij}^{\prime \prime \bkt{l}}, \quad p^{ \prime\prime\prime} = \max_{\substack{\ell+1 \leq l \leq L \\ i,j}} {p}_{ij}^{\prime\prime\prime \bkt{l}}.
\end{align*}

From \cref{eq:S1} and \cref{eq:S2} we can see that:
\begin{itemize}
    \item The term $S_1^{\prime}$ in $S_1 \bkt{\ell}$ scales as $\mathcal{O} \bkt{C^{\prime} p^{\prime} N \bkt{L-\ell}}$, while the term $S_2^{\prime \prime \prime}$ in $S_2$ scales as $\mathcal{O} \bkt{C^{\prime\prime\prime} p^{\prime\prime\prime} N \bkt{L-\ell}}$. 
    Note that $C^{\prime} > C^{\prime\prime\prime}$ for $d>1$ and $C^{\prime}$ grows more rapidly than $C^{\prime\prime\prime}$ as the dimension $d$ increases. Generally, $p^{\prime} < p^{\prime\prime\prime}$ because $p^{\prime}$ corresponds to the numerical rank of far-field blocks. However, since the weak admissibility condition is applied at \emph{finer} levels where block sizes are smaller, $p^{\prime\prime\prime}$ is not expected to be substantially larger than $p^{\prime}$. In fact, for moderate to higher values of target accuracy $\epsilon$, the difference between $p^{\prime}$ and $p^{\prime\prime\prime}$ is very small.
    \item In $S_1 \bkt{\ell}$, the term $S_1^{\prime \prime}$, corresponding to the dense storage of neighbor blocks, scales as $\mathcal{O} \bkt{C^{\prime\prime} N n_{\max}^2}$, whereas the term $S_2^{\prime \prime}$ in $S_2 \bkt{\ell}$ scales as $\mathcal{O}\bkt{C^{\prime \prime} p^{\prime\prime} N}$.
\end{itemize}

The storage cost of the $\mathcal{H}_h$-matrix at finer levels is controlled by the choice of the switching level $\ell$. We present the following condition under which $\mathcal{H}_h$-matrices achieve lower storage costs.

\subsubsection*{Criterion for lower storage cost in \texorpdfstring{$\mathcal{H}_h$}{Hh}-matrix}
From \cref{eq:storage_Hs} and \cref{eq:storage_Hh}, it follows that $\mathcal{S}_s > \mathcal{S}_h$ iff $S_1 \bkt{\ell} > S_2 \bkt{\ell}$, i.e., $S_1 \bkt{\ell}/S_2 \bkt{\ell} > 1$.

Here, an interesting question that arises is how to choose the switching level $\ell$ to maximize the storage gains over $\mathcal{H}_s$-matrices, i.e.,

\begin{align}
    \ell^* = \arg \max_{0 \leq \ell \leq L} \bkt{S_1 \bkt{\ell} - S_2 \bkt{\ell}}
\end{align}

The optimal switching level $\ell^*$ depends on many factors such as the matrix $H$, tree depth $L$, and target accuracy $\epsilon$. We propose \Cref{alg:algo_optimal_level} to determine $\ell^*$. 

First, the $\mathcal{H}_s$-matrix representation of $H$ is constructed, and the values of $S_1 \bkt{\ell}$ are recorded for $0 \leq \ell \leq L$. The values of $S_2 \bkt{\ell}$ can be computed with the help of \Cref{alg:algo1}. Finally, the tree is traversed from the bottom to the top direction using \Cref{alg:algo_optimal_level} in order to identify $\ell^{*}$.

 \begin{algorithm}[H]
 \small
	\caption{A choice of $\ell$ that gives maximum storage gains over $\mathcal{H}_s$-matrices.}\label{alg:algo_optimal_level}
    \begin{algorithmic}[1]
        \State \textbf{Output:} $\ell^{*}$
        \State $\ell \gets L$, \quad Compute $S_2(\ell)$
        \State $s \gets S_1(\ell) - S_2(\ell)$ \Comment{$S_1(L) = S_2(L)$.}
        
        \While{$\ell > 0$}
            \State Compute $S_2(\ell-1)$, \quad $s_{\text{new}} \gets S_1(\ell-1) - S_2(\ell-1)$
        
            \If{$s_{\text{new}} \leq s$}
                \State \textbf{break}
            \EndIf
            
            \State $s \gets s_{\text{new}}$, \quad $\ell \gets \ell-1$
        \EndWhile
        \State $\ell^{*} = \ell$
	\end{algorithmic}
\end{algorithm}

Note that $\ell^*$ obtained in \emph{uniform} precision scheme may differ when $\mathcal{H}$-matrices are represented in \emph{mixed} precision scheme, since the storage estimates $\bkt{S_1(\ell), S_2(\ell)}$ must consider the varying numbers of precision bits. A detailed investigation of the behavior of $\ell^*$ is beyond the scope of the current work. In the numerical experiments, unless stated otherwise, $\mathcal{H}_h$-matrices are constructed with $\ell = L-1$.

We perform two numerical experiments to compare the storage costs of the $\mathcal{H}_s$ and $\mathcal{H}_h$ matrices. 

First, we repeat the same numerical experiment in two dimensions as described in \Cref{subsec:H_HODLR}. For various choices of the target accuracy $\epsilon$ and the tree depth $L$, we construct the $\mathcal{H}_s$ and $\mathcal{H}_h$ representations of the dense kernel matrix $H$ and compare their storage costs, as illustrated in \Cref{fig:bar_mem_comp_1}. The bar plots show that the $\mathcal{H}_h$-matrix requires less dense storage and achieves a lower overall storage cost than the $\mathcal{H}_s$-matrix. Additionally, it can be observed that while the dense storage requirement of the $\mathcal{H}_h$-matrix is the same as that of the HODLR matrix, the $\mathcal{H}_h$-matrix achieves a significantly lower total storage requirement (see \Cref{fig:HODLR_H_mem_comp} and \Cref{fig:bar_mem_comp_1}).

    \begin{figure}
        \centering
        \subfloat[\scriptsize  $N = 262144$ and $n_{\max} = 256$ $(L=5)$.]{
        \includegraphics[scale=.34]{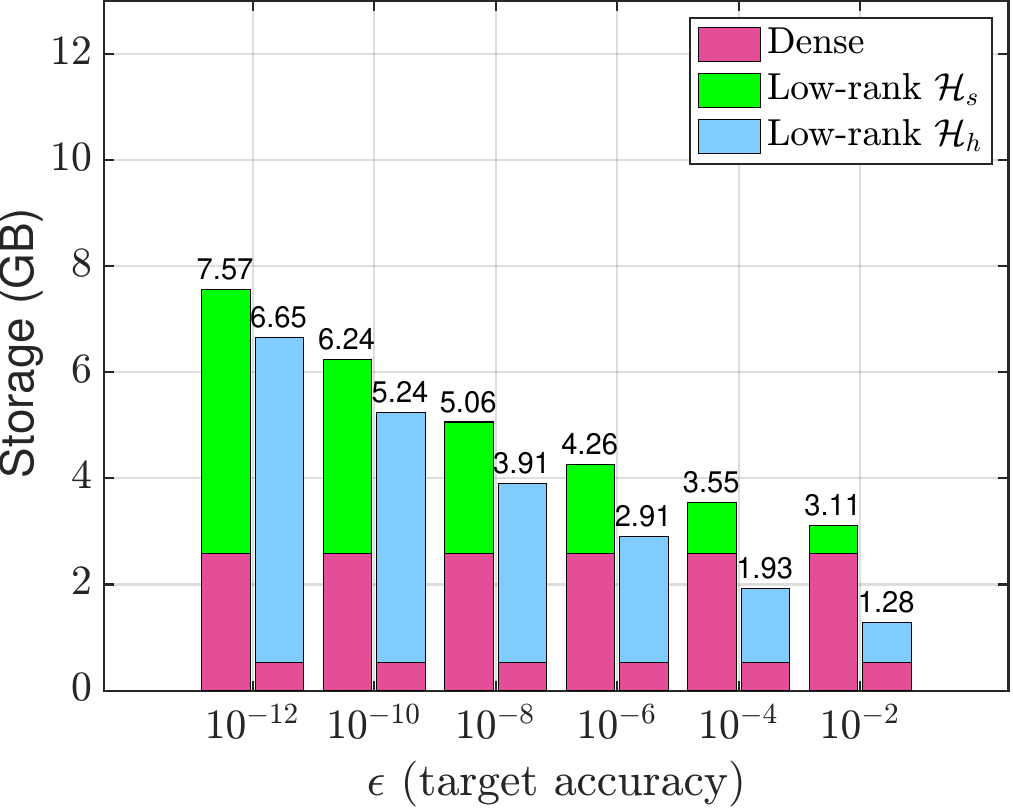}\label{fig:Hs_Hh_comp_mem2_2d}}
        \subfloat[\scriptsize  $N = 262144$ and $n_{\max} = 64$ $(L=6)$.]{
        \includegraphics[scale=.34]{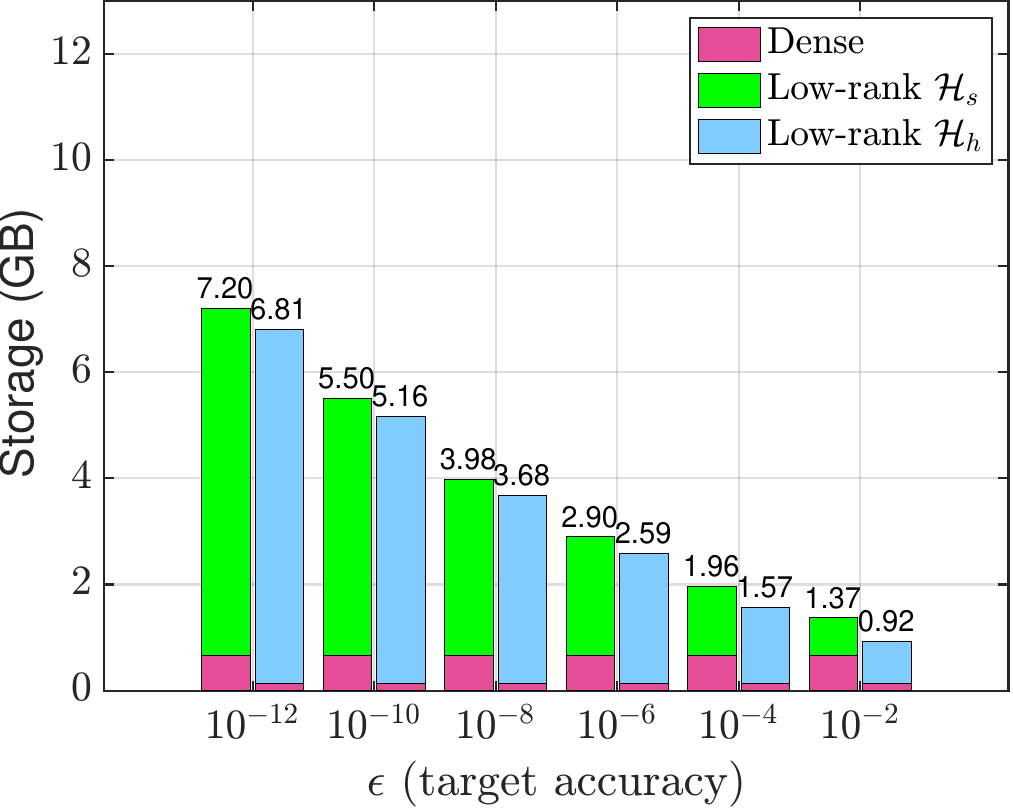}\label{fig:Hs_Hh_comp_mem1_2d}}
        \caption{Comparison of the storage costs of various $\mathcal{H}$-matrices for the kernel $1/r$ in $2$D.}
        \label{fig:bar_mem_comp_1}
    \end{figure}

    \begin{figure}
        \centering
        \subfloat[\scriptsize  $N = 64000$, $n_{\max} = 125$ $(L=3)$.]{
        \includegraphics[scale=.34]{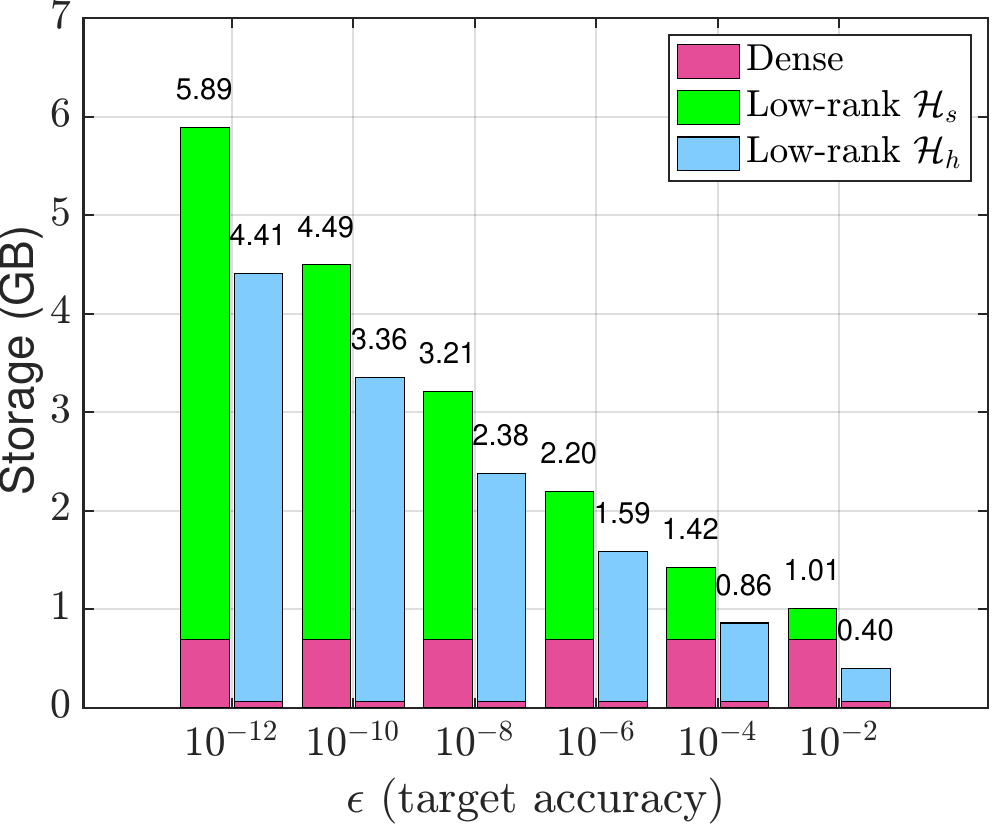}\label{fig:Hs_Hh_comp_mem1_3d}}
        \subfloat[\scriptsize  $N = 64000$, $n_{\max} = 64$ $(L=4)$.]{
        \includegraphics[scale=.34]{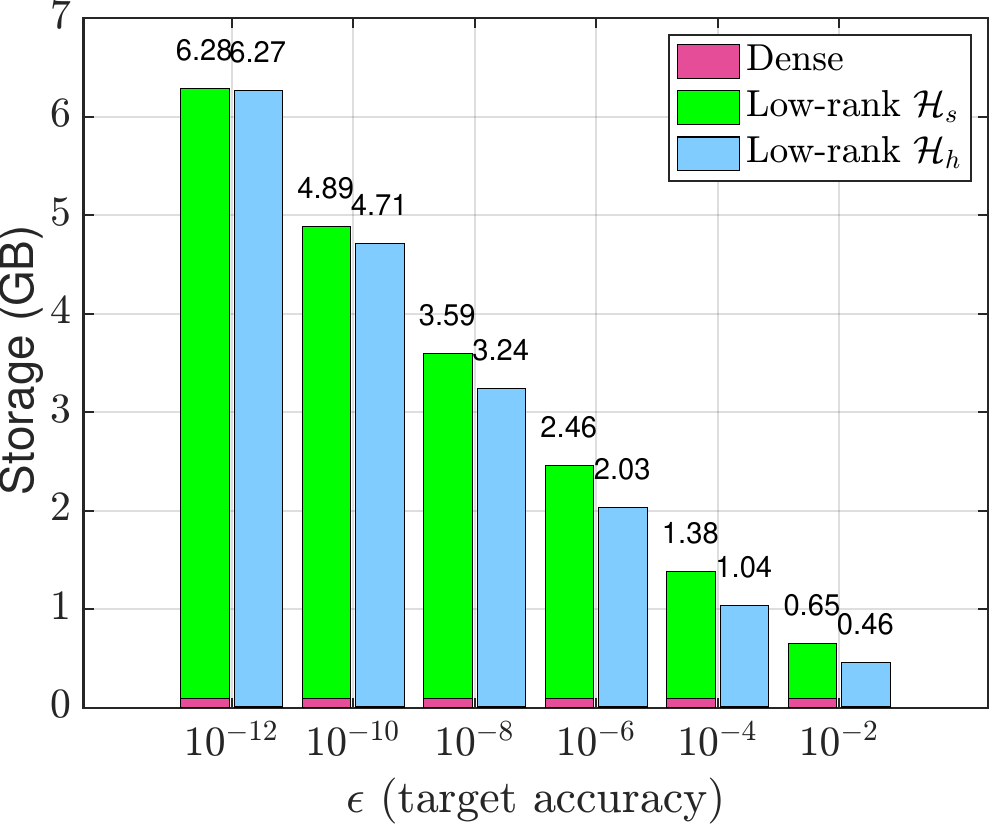}\label{fig:Hs_Hh_comp_mem2_3d}}
        \caption{Comparison of the storage costs of various $\mathcal{H}$-matrices for the kernel $1/r$ in $3$D.}
        \label{fig:bar_mem_comp_2}
    \end{figure}

In the second experiment, we consider $N=64000$ points uniformly distributed in the cube $[-1,1]^3$ and the dense kernel matrix $H \in \mathbb{R}^{N \times N}$ is formed using the kernel function $1/r$ as described in \cref{eq:kernel_mat}. We show the storage costs of the $\mathcal{H}_s$ and $\mathcal{H}_h$ representations of the dense kernel matrix $H$ in \Cref{fig:bar_mem_comp_2}. The bar plots demonstrate that the $\mathcal{H}_h$-matrix requires less dense storage and achieves a lower overall storage cost than the $\mathcal{H}_s$-matrix. We also tabulate the the values of $S_1 \bkt{\ell}$, $S_2 \bkt{\ell}$, and $S_1 \bkt{\ell}/S_2 \bkt{\ell}$ corresponding to various choices of $\epsilon$ and $L$ in \Cref{tab:s1_s_2}. The results show that $S_1 \bkt{\ell}/S_2 \bkt{\ell} > 1$ for all cases, and therefore $\mathcal{S}_s > \mathcal{S}_h$.

\begin{table}[H]
\centering
\resizebox{10cm}{!}{%
\begin{tabular}{|l|lll||lll|}
\hline
\multirow{2}{*}{\quad $\epsilon$} & \multicolumn{3}{l||}{$N = 64000$, $n_{\max} = 125$ $(L=3)$}                                          & \multicolumn{3}{l|}{$N = 64000$, $n_{\max} = 64$ $(L=4)$}                                          \\ \cline{2-7} 
                  & \multicolumn{1}{l|}{$S_1 \bkt{\ell}$} & \multicolumn{1}{l|}{$S_2\bkt{\ell}$} & $S_1\bkt{\ell}/S_2\bkt{\ell}$ & \multicolumn{1}{l|}{$S_1\bkt{\ell}$} & \multicolumn{1}{l|}{$S_2\bkt{\ell}$} & $S_1\bkt{\ell}/S_2\bkt{\ell}$ \\ \hline
                  $10^{-12}$& \multicolumn{1}{l|}{3.86} & \multicolumn{1}{l|}{2.38} & 1.62      & \multicolumn{1}{l|}{1.08} & \multicolumn{1}{l|}{1.06} & 1.02      \\ 
                  $10^{-10}$& \multicolumn{1}{l|}{3.06} & \multicolumn{1}{l|}{1.92} & 1.59      & \multicolumn{1}{l|}{1.08} & \multicolumn{1}{l|}{0.90} & 1.20      \\ 
                  $10^{-8}$& \multicolumn{1}{l|}{2.27} & \multicolumn{1}{l|}{1.43} & 1.59      & \multicolumn{1}{l|}{1.07} & \multicolumn{1}{l|}{0.71} & 1.51      \\ 
                  $10^{-6}$& \multicolumn{1}{l|}{1.63} & \multicolumn{1}{l|}{1.02} & 1.60      & \multicolumn{1}{l|}{0.95} & \multicolumn{1}{l|}{0.52} & 1.83      \\ 
                  $10^{-4}$& \multicolumn{1}{l|}{1.12} & \multicolumn{1}{l|}{0.56} & 2.00      & \multicolumn{1}{l|}{0.65} & \multicolumn{1}{l|}{0.31} & 2.10      \\ 
                  $10^{-2}$& \multicolumn{1}{l|}{0.84} & \multicolumn{1}{l|}{0.23} & 3.65      & \multicolumn{1}{l|}{0.33} & \multicolumn{1}{l|}{0.14} & 2.36      \\ \hline
\end{tabular}}
\caption{The values of $S_1 \bkt{\ell}/S_2 \bkt{\ell}$ for the kernel $1/r$ in $3$D are shown when $\ell=L-1$.}
\label{tab:s1_s_2}
\end{table}

\section{Adaptive mixed precision \texorpdfstring{$\mathcal{H}_h$}{Hh}-matrices} \label{sec:adaptive_prec_Hmatrix}
We analyze the rounding error of $\mathcal{H}_h$-matrices and present an adaptive mixed precision $\mathcal{H}_h$-matrix representation. 

\subsection{\texorpdfstring{$\mathcal{H}_h \bkt{p,\epsilon}$}{H(p,epsilon)}-matrices}
The $\mathcal{H}_h \bkt{p,\epsilon}$-matrix is defined for a target accuracy $\epsilon$ as follows.

\begin{definition}
\emph{($\mathcal{H}_h \bkt{p,\epsilon}$-matrix)}.
Let $\widetilde{H}$ be the hybrid hierarchical matrix representation of the dense matrix $H$. $\widetilde{H}$ is said to be an $\mathcal{H}_h \bkt{p,\epsilon}$-matrix if, for each admissible block $H^{\bkt{l}}_{I,J}$, where $I$ and $J$ indicate the index sets of the admissible clusters $\mathcal{C}_i^{\bkt{l}}$ and $\mathcal{C}_j^{\bkt{l}}$, respectively, the following conditions hold:

\begin{enumerate}
    \item Each admissible block has rank at most $p$, i.e., $\text{rank } (\widetilde{H}^{\bkt{l}}_{I,J}) \leq p$.
    \item Each admissible block satisfies the error bound 
    \begin{align} \label{eq:adm_block_norm}
        \magn{H^{\bkt{l}}_{I,J} - \widetilde{H}^{\bkt{l}}_{I,J}} \leq \epsilon \magn{H^{\bkt{l}}_{I,J}}.
    \end{align}
\end{enumerate}

\end{definition}

The following lemma gives a bound on the error of $\mathcal{H}_h \bkt{p,\epsilon}$-matrices.
\begin{lemma} \label{lemma:global_error}
Let $\widetilde{H}$ be an $\mathcal{H}_h \bkt{p,\epsilon}$-matrix format of the matrix $H$. Then, with respect to the Frobenius norm, the approximation error satisfies $\magn{H - \widetilde{H}} \leq \epsilon \magn{H}$.
\end{lemma}

\begin{proof}
The proof can be found in \Cref{sec:app_proof_lemma}.
\end{proof}

\subsection{Error analysis of mixed precision \texorpdfstring{$\mathcal{H}_h$}{Hh}-matrices}
If $I \times J$ is an admissible block at level $l$ of an $\mathcal{H}_h$-matrix, it admits the following low-rank representation

\begin{align*}
    \widetilde{H}^{\bkt{l}}_{I,J} = \widetilde{U}^{\bkt{l}}_I (\widetilde{V}^{\bkt{l}}_J)^*,
\end{align*}
where $\widetilde{U}^{\bkt{l}}_I \in \mathbb{R}^{\abs{I} \times p}$ and $\widetilde{V}^{\bkt{l}}_J \in \mathbb{R}^{\abs{J} \times p}$. We assume that $\widetilde{U}^{\bkt{l}}_I$ has columns orthonormal to precision $u$. Our goal is to represent the admissible blocks of $\mathcal{H}_h$-matrix in a precision that is typically lower than the working precision $u$, without degrading the overall approximation quality. In the mixed precision representation of an $\mathcal{H}_h$-matrix, $\widehat{H}$, we store the low-rank factors $\widetilde{U}^{\bkt{l}}_I$ and $\widetilde{V}^{\bkt{l}}_J$ corresponding to an admissible block in precision $u_{i,j}^{(l)}$, while the diagonal blocks at level $L$ are stored in working precision $u$. As a result, a mixed precision scheme may achieve significant storage gains compared to the uniform precision scheme. However, the precisions $u_{i,j}^{(l)}$ must be chosen carefully to ensure that the error in the mixed precision representation remains controlled. 

We propose the following theorem, which provides guidance on selecting the right precision associated with each admissible block in order to maintain the global error bound at a satisfactory level.

\begin{theorem} \label{thm:amp_throrem}
\emph{(Global error in mixed precision $\mathcal{H}_h$ representation)}.
Let $\widetilde{H}$ be an $\mathcal{H}_h \bkt{p,\epsilon}$-matrix format of the matrix $H$, and let $\widehat{H}$ denote its mixed precision hierarchical matrix representation, in which the low-rank factors corresponding to the admissible block $I \times J$ at level $l$ are stored in precision $u_{i,j}^{\bkt{l}}$. Define $\xi_{i,j}^{\bkt{l}} := {\magn{\widetilde{H}_{I,J}^{\bkt{l}}}}/{\magn{\widetilde{H}}}$, where $1 \leq i \neq j \leq 2^{dl}$ and $1 \leq l \leq L$. If the precision $u_{i,j}^{\bkt{l}}$ satisfies $u_{i,j}^{\bkt{l}} \leq \dfrac{\epsilon}{2^{dl/2} \xi_{i,j}^{\bkt{l}}}$, the global error bound for $\mathcal{H}_h$-matrices is given by
\begin{align} \label{eq:mp_thm}
 \magn{H - \widehat{H}} \lesssim \bkt{\sqrt{\bkt{\ell C^{\prime} + C^{\prime\prime} + \bkt{L-\ell} C^{\prime\prime\prime}}} 2 + 1} \epsilon \magn{H},
\end{align}
where the constants $C^{\prime}$ and $C^{\prime\prime}$ depend on the underlying dimension $d$ and the admissibility parameter $\eta$, and $C^{\prime\prime\prime}$ depends only on $d$.
\end{theorem}

\begin{proof}
In the $\widehat{H}$ representation, an admissible block $I \times J$ at level $l$ is stored in precision $u_{i,j}^{\bkt{l}}$. We consider that the leaf-level diagonal blocks $\widehat{H}^{\bkt{L}}_{I,I}$ and $\widetilde{H}^{\bkt{L}}_{I,I}$ are identical, i.e., $\widehat{H}^{\bkt{L}}_{I,I} = \widetilde{H}^{\bkt{L}}_{I,I}$. 

From Lemma 2.2 in \cite{amestoy2023mixed}, together with the definition of $\xi_{i,j}^{\bkt{l}}$, it follows that
\begin{align} \label{eq:need_mvp}
    \magn{\widetilde{H}^{(l)}_{I,J} - \widehat{H}^{(l)}_{I,J}} \leq (2 + \sqrt{p} u_{i,j}^{\bkt{l}})u_{i,j}^{\bkt{l}} \magn{\widetilde{H}^{(l)}_{I,J}} = \xi_{i,j}^{\bkt{l}} (2 + \sqrt{p} u_{i,j}^{\bkt{l}})u_{i,j}^{\bkt{l}} \magn{\widetilde{H}}.
\end{align}

\begin{align} \label{eq:amp_eq0}
    \magn{\widetilde{H}^{(l)}_{I,J} - \widehat{H}^{(l)}_{I,J}}^2 \leq \bkt{\xi_{i,j}^{\bkt{l}}}^2 \bkt{2 + \sqrt{p} u_{i,j}^{\bkt{l}}}^2 \bkt{u_{i,j}^{\bkt{l}}}^2 \magn{\widetilde{H}}^2 \lesssim 4 \bkt{\xi_{i,j}^{\bkt{l}} u_{i,j}^{\bkt{l}}}^2 \magn{\widetilde{H}}^2, 
\end{align}
where $l = 1,2, \dots ,L$.

\begin{align*} 
 \resizebox{\linewidth}{!}{$
\begin{aligned} 
   \magn{\widetilde{H} - \widehat{H}}^2 &=   \dsum_{l=1}^{\ell} \dsum_{i=1}^{2^{dl}} \dsum_{\mathcal{C}_j^{\bkt{l}} \in \mathcal{IL}_s (\mathcal{C}_i^{\bkt{l}})} \magn{\widetilde{H}^{(l)}_{I,J} - \widehat{H}^{(l)}_{I,J}}^2 + \dsum_{i=1}^{2^{d \ell}} \dsum_{\mathcal{C}_j^{\bkt{\ell}} \in \mathcal{N}_s (\mathcal{C}_i^{\bkt{\ell}})} \magn{\widetilde{H}^{(\ell)}_{I,J} - \widehat{H}^{(\ell)}_{I,J}}^2 \\
   &\quad \quad + \dsum_{l=\ell+1}^{L} \dsum_{i=1}^{2^{dl}} \dsum_{\mathcal{C}_j^{\bkt{l}} \in \mathcal{IL}_w (\mathcal{C}_i^{\bkt{l}})} \magn{\widetilde{H}^{(l)}_{I,J} - \widehat{H}^{(l)}_{I,J}}^2 \\ 
   & \lesssim \quad \dsum_{l=1}^{\ell} \dsum_{i=1}^{2^{dl}} \dsum_{\mathcal{C}_j^{\bkt{l}} \in \mathcal{IL}_s (\mathcal{C}_i^{\bkt{l}})} 4 \bkt{\xi_{i,j}^{\bkt{l}} u_{i,j}^{\bkt{l}}}^2 \magn{\widetilde{H}}^2 + \dsum_{i=1}^{2^{d \ell}} \dsum_{\mathcal{C}_j^{\bkt{\ell}} \in \mathcal{N}_s (\mathcal{C}_i^{\bkt{\ell}})}  4 \bkt{\xi_{i,j}^{\bkt{\ell}} u_{i,j}^{\bkt{\ell}}}^2 \magn{\widetilde{H}}^2 \\ 
   &\quad \quad + \dsum_{l=\ell+1}^{L} \dsum_{i=1}^{2^{dl}} \dsum_{\mathcal{C}_j^{\bkt{l}} \in \mathcal{IL}_w (\mathcal{C}_i^{\bkt{l}})}  4 \bkt{\xi_{i,j}^{\bkt{l}} u_{i,j}^{\bkt{l}}}^2 \magn{\widetilde{H}}^2 \quad \bkt{\text{By using \cref{eq:amp_eq0}}}
\end{aligned}
$}
\end{align*}

Now, setting

\begin{align} \label{eq:adaptive_prec_choice}
   u_{i,j}^{\bkt{l}} \leq \dfrac{\epsilon}{2^{dl/2} \xi_{i,j}^{\bkt{l}}} \implies \xi_{i,j}^{\bkt{l}} u_{i,j}^{\bkt{l}} \leq \dfrac{\epsilon}{2^{dl/2}},
\end{align}

the above expression becomes 

\begin{align*}
\begin{aligned}
    \magn{\widetilde{H} - \widehat{H}}^2 & \lesssim \quad \dsum_{l=1}^{\ell} \dsum_{i=1}^{2^{dl}} \dsum_{\mathcal{C}_j^{\bkt{l}} \in \mathcal{IL}_s (\mathcal{C}_i^{\bkt{l}})} \dfrac{4\epsilon^2}{2^{dl}} \magn{\widetilde{H}}^2 + \dsum_{i=1}^{2^{d \ell}} \dsum_{\mathcal{C}_j^{\bkt{\ell}} \in \mathcal{N}_s (\mathcal{C}_i^{\bkt{\ell}})}  \dfrac{4 \epsilon^2}{2^{d \ell}} \magn{\widetilde{H}}^2 \\
    &\quad \quad + \dsum_{l=\ell+1}^{L} \dsum_{i=1}^{2^{dl}} \dsum_{\mathcal{C}_j^{\bkt{l}} \in \mathcal{IL}_w (\mathcal{C}_i^{\bkt{l}})}  \dfrac{4 \epsilon^2}{2^{dl}} \magn{\widetilde{H}}^2\\ 
    & \leq \quad \dsum_{l=1}^{\ell} \dsum_{i=1}^{2^{dl}} C_{sp}' \dfrac{4\epsilon^2}{2^{dl}} \magn{\widetilde{H}}^2 + \dsum_{i=1}^{2^{d \ell}} C_{sp}'' \dfrac{4 \epsilon^2}{2^{d \ell}} \magn{\widetilde{H}}^2 + \dsum_{l=\ell+1}^{L} \dsum_{i=1}^{2^{dl}} C_{sp}'''  \dfrac{4 \epsilon^2}{2^{dl}} \magn{\widetilde{H}}^2\\ 
    & \leq \quad \bkt{\ell C^{\prime} + C^{\prime\prime} + \bkt{L-\ell} C^{\prime\prime\prime}} 4 \epsilon^2 \magn{\widetilde{H}}^2, 
\end{aligned}
\end{align*}
where the sparsity constants \cite{weak_hackbusch2004} (maximum size of admissible pairs) $C_{sp}'$, $C_{sp}''$ and $C_{sp}'''$ for a balanced tree are bounded above as follows:
\begin{enumerate}
    \item $C_{sp}' \leq C^{\prime} = (2^d-1) (1 + {2\sqrt{d}}/{\eta})^d$, \quad (see \cref{eq:standard_IL_size})
    \item $C_{sp}'' \leq C^{\prime\prime} = (1 + {2\sqrt{d}}/{\eta})^d - 1$, \quad (see \cref{eq:standard_N_size})
    \item $C_{sp}''' \leq C^{\prime\prime\prime} = (2^d-1)$, \quad (see \cref{eq:weak_IL_size}).
\end{enumerate}

Therefore, we get
\begin{align} \label{eq:amp_eq}
     \magn{\widetilde{H} - \widehat{H}} \lesssim \sqrt{\bkt{\ell C^{\prime} + C^{\prime\prime} + \bkt{L-\ell} C^{\prime\prime\prime}}} 2 \epsilon \magn{H}.
\end{align}

Applying \cref{lemma:global_error} and \cref{eq:amp_eq}, together with the triangle inequality, readily yields the result.
\end{proof}

\subsubsection*{Special cases}
The global error bound for adaptive mixed precision $\mathcal{H}_s$ and HODLR matrices can be obtained as follows.

   \begin{itemize}
       \item By setting $\ell = L$ in \cref{eq:mp_thm}, one obtains the global error bound for adaptive mixed precision $\mathcal{H}_s$-matrices:
       \begin{align} \label{eq:err2_bound}
           \magn{H - \widehat{H}} \lesssim (\sqrt{\bkt{L C^{\prime} + C^{\prime \prime}}} 2 + 1) \epsilon \magn{H}
       \end{align}
        \item By setting $\ell = 0$ in \cref{eq:mp_thm}, one obtains the global error bound for adaptive mixed precision HODLR matrices (in this case, $C^{\prime \prime}_{sp} = 0$, so $C^{\prime \prime}$ is absent):
       \begin{align} 
           \magn{H - \widehat{H}} \lesssim (\sqrt{L C^{\prime \prime \prime}} 2 + 1) \epsilon \magn{H}
       \end{align}
   \end{itemize}

\begin{remark}
The paper~\cite{carson2025mixed} gives a global error bound for HODLR matrices built using the balanced binary tree. In contrast, this work considers a more general and structurally more complex hierarchical representation based on a balanced $2^d$-tree. Furthermore, in the mixed precision HODLR representation of \cite{carson2025mixed}, the precision is fixed level-wise throughout the hierarchy. In this work, no such constraint is imposed. Instead, we assign a precision $u_{i,j}^{(l)}$ to an admissible block $I \times J$ at level $l$, i.e., at the same level, the precision may differ for different admissible blocks (see \Cref{fig:prec_color}). This block-wise precision assignment provides greater flexibility and enables the more effective use of lower precisions. Such flexibility is particularly important since the $\mathcal{H}_h$-matrices have a larger number of admissible blocks compared to the binary tree-based HODLR matrices.
\end{remark}

\subsection{Adaptive mixed precision \texorpdfstring{$\mathcal{H}_h$}{Hh}-matrix representation}
We present an adaptive mixed precision representation in \Cref{alg:algo2} based on \Cref{thm:amp_throrem}. 

\Cref{alg:algo2} takes as input the $\mathcal{H}_h$-matrix, the target accuracy $\epsilon$, the working precision $u$, and a set of available precisions $\mathcal{U} = \{ u_1, u_2, \dots, u_n \}$, satisfying $u \leq u_1 < u_2 < \cdots < u_n$. Since the rank of the admissible blocks is usually not known \emph{a priori}, the low-rank approximation is performed using the prescribed target accuracy $\epsilon$. We also assume that $u < \epsilon$, as the low-rank approximation is carried out in working precision. To satisfy the global error bound \cref{eq:mp_thm}, the low-rank factors corresponding to an admissible block $I \times J$ at level $l$ are stored adaptively in precision $u_{i,j}^{\bkt{l}}$, chosen as the lowest available precision $\Bar{u} \in u \cup \mathcal{U}$ such that

\begin{align} \label{eq:amp_prec_choice}
    \Bar{u} \leq \dfrac{\epsilon}{2^{dl/2} \xi_{i,j}^{\bkt{l}}}.
\end{align}
Finally, the algorithm outputs the adaptive mixed precision representation $\widehat{H}$. 

It should be noted that by leveraging memory accessor mechanisms, one can exploit a continuum of precision formats for storing data, rather than being limited to the floating-point formats available for computation (see \cite{mary2025approximate}, Chapter 12). This allows storing each block with exactly the right number of bits required to achieve the desired accuracy, leading to more efficient memory use and potentially reducing data movement costs further.

 \begin{algorithm}
 \small
	\caption{Adaptive mixed precision $\mathcal{H}_h$-matrix representation.}\label{alg:algo2}
	\begin{algorithmic}[1]
        \State \textbf{Input:} $\widetilde{H}$, $\epsilon$, $u$, $\mathcal{U} = \{ u_1, u_2, \dots, u_n \}$. \qquad \textbf{Output:} $\widehat{H}$.
\For{\texttt{$l=1:\ell$}} 
				\For{\texttt{$i=1:2^{dl}$}}
					\For{$\mathcal{C}^{\bkt{l}}_{j} \in$ $\mathcal{IL}_s (\mathcal{C}^{\bkt{l}}_{i})$}
                        \State $\xi^{\bkt{l}}_{i,j} \gets {\magn{\widetilde{H}^{(l)}_{I,J}}}/{\magn{\widetilde{H}}}$, \quad
                        $u^{\bkt{l}}_{i,j} \gets \displaystyle \max \Big \{\Bar{u} \leq \dfrac{\epsilon}{2^{{dl}/2} \xi^{\bkt{l}}_{i,j}} : \Bar{u} \in u \cup \mathcal{U} \Big \}$
                        \State $\widehat{H}^{\bkt{l}}_{I,J} \gets \widehat{U}^{\bkt{l}}_I (\widehat{V}^{\bkt{l}}_J)^*$ \Comment{Store the low-rank factors $\widehat{U}^{\bkt{l}}_I, \widehat{V}^{\bkt{l}}_J$ in precision $u^{\bkt{l}}_{i,j}$ and compute in working precision $u$.}
					\EndFor
				\EndFor
			\EndFor

          		\For{\texttt{$i=1:2^{d\ell}$}}
					\For{$\mathcal{C}^{\bkt{\ell}}_{j} \in$ $\mathcal{N}_s (\mathcal{C}^{\bkt{\ell}}_{i})$}
                        \State $\xi^{\bkt{\ell}}_{i,j} \gets {\magn{\widetilde{H}^{(\ell)}_{I,J}}}/{\magn{\widetilde{H}}}$, \quad
                        $u^{\bkt{\ell}}_{i,j} \gets \displaystyle \max \Big \{\Bar{u} \leq \dfrac{\epsilon}{2^{{d\ell}/2} \xi^{\bkt{\ell}}_{i,j}} : \Bar{u} \in u \cup \mathcal{U} \Big \}$
                        \State $\widehat{H}^{\bkt{\ell}}_{I,J} \gets \widehat{U}^{\bkt{\ell}}_I (\widehat{V}^{\bkt{\ell}}_J)^*$ \Comment{Store the low-rank factors $\widehat{U}^{\bkt{\ell}}_I, \widehat{V}^{\bkt{\ell}}_J$ in precision $u^{\bkt{\ell}}_{i,j}$ and compute in working precision $u$.}
					\EndFor
				\EndFor  

            \For{\texttt{$l=\ell+1:L$}} 
				\For{\texttt{$i=1:2^{dl}$}}
					\For{$\mathcal{C}^{\bkt{l}}_{j} \in$ $\mathcal{IL}_w (\mathcal{C}^{\bkt{l}}_{i})$}
                        \State $\xi^{\bkt{l}}_{i,j} \gets {\magn{\widetilde{H}^{(l)}_{I,J}}}/{\magn{\widetilde{H}}}$, \quad
                        $u^{\bkt{l}}_{i,j} \gets \displaystyle \max \Big \{\Bar{u} \leq \dfrac{\epsilon}{2^{{dl}/2} \xi^{\bkt{l}}_{i,j}} : \Bar{u} \in u \cup \mathcal{U} \Big \}$
                        \State $\widehat{H}^{\bkt{l}}_{I,J} \gets \widehat{U}^{\bkt{l}}_I (\widehat{V}^{\bkt{l}}_J)^*$ \Comment{Store the low-rank factors $\widehat{U}^{\bkt{l}}_I, \widehat{V}^{\bkt{l}}_J$ in precision $u^{\bkt{l}}_{i,j}$ and compute in working precision $u$.}
					\EndFor
				\EndFor
			\EndFor

            \For{\texttt{$i=1:2^{dL}$}}
                \State $\widehat{H}^{\bkt{L}}_{I,I} \gets \widetilde{H}^{\bkt{L}}_{I,I}$
                \Comment{Store leaf level dense diagonal block $\widehat{H}^{\bkt{L}}_{I,I}$ in working precision $u$.}
            \EndFor
	\end{algorithmic}
\end{algorithm}

\subsubsection*{Special cases}
   From \Cref{alg:algo2}, adaptive mixed precision $\mathcal{H}_s$ and HODLR representations can be obtained easily as follows:
   \begin{itemize}
       \item Setting $\ell = L$ in \Cref{alg:algo2} generates an adaptive mixed precision $\mathcal{H}_s$-matrix representation (see lines $2-15$ \& lines $24-26$ of \Cref{alg:algo2}).

       It is noteworthy that a low-rank representation that is not beneficial in uniform precision may be beneficial in the mixed precision case. For example, consider a $125 \times 125$ neighbor block whose numerical rank is $64$ for $\epsilon = 10^{-10}$. Storing the low-rank factors in fp64 is not beneficial, since $64 \times \bkt{125+125} > 125 \times 125$. However, storing the factors in lower precision, such as fp32, becomes beneficial because $64 \times \bkt{125+125} \times 0.5 < 125 \times 125$. Thus, a matrix that is not sufficiently low-rank in a uniform precision may become so in a lower or mixed precision arithmetic. Consequently, we compress the neighbor blocks (see lines $10-15$ of \Cref{alg:algo2}) whenever the low-rank representation is beneficial; otherwise, we store them in their original dense form. Clearly, this does not violate the global error bound \cref{eq:err2_bound}, which is further confirmed by the numerical results presented in \Cref{subsec:global_error}.

       
       
       \item Setting $\ell = 0$ in \Cref{alg:algo2} generates an adaptive mixed precision HODLR representation (see lines $16-26$ of \Cref{alg:algo2}). 
   \end{itemize}

\begin{remark}
    Note that as the tree depth $L$ increases, the size of the full-rank diagonal blocks decreases. Hence, if $\magn{\widetilde{H}_{I,I}^{(L)}} \ll \magn{H}$, it may be possible to store these diagonal blocks in lower precision. However, we store them in the working precision $u$, assuming it is the highest available precision.
\end{remark}

One important component of \Cref{alg:algo2} is the computation of Frobenius norms, which incurs a computational cost of $\mathcal{O}(N^2)$. If $H$ is a small or sparse matrix, the ratios $\xi^{\bkt{l}}_{i,j} = {\magn{\widetilde{H}^{(l)}_{I,J}}}/{\magn{\widetilde{H}}}$ can be computed directly as ${\magn{H^{(l)}_{I,J}}}/{\magn{H}}$, without $\widetilde{H}$. However, this becomes expensive when $H$ is large and dense. To address this, we first construct the $\widetilde{H}$ representation and compute the Frobenius norms using its low-rank factors, storing them for reuse while forming $\widehat{H}$. A detailed procedure for fast Frobenius norm computation using SVD factors is presented in \Cref{alg:Frob_norm}. A similar approach can be followed for other low-rank approximation techniques, such as truncated QR decomposition. In practical applications, $\widehat{H}$ can be formed once and reused multiple times in subsequent computations, resulting in significant computational savings.

We perform a numerical experiment using \Cref{alg:algo2} to demonstrate the suitable precision assigned to each block for different $\epsilon$ (see \Cref{fig:prec_color}). We consider $N = 6400$ points uniformly distributed in the domain $[-1,1]^2$ with tree depth $L = 4$. The dense kernel matrix $H$ is constructed using the kernel function $\log(r)$, as defined in \cref{eq:kernel_mat}. The switching level is chosen as $\ell = L-1 = 3$. The working precision is $u=\mathrm{fp64}$, and the set of available precisions is $\mathcal{U}=\{\mathrm{fp64}, \mathrm{fp32}, \mathrm{fp16}, \mathrm{bf16}, \mathrm{q43}\}$ (see \Cref{tab:precisions}). \textbf{Zoomed-in views} in \Cref{fig:hmat0_prec_color} and \Cref{fig:hmat1_prec_color} illustrate the block-wise precision selected by the adaptive mixed precision $\mathcal{H}_h$ representation for $\epsilon=10^{-4}$ and $\epsilon=10^{-2}$, respectively.

    \begin{figure}[H]
        \centering
        \subfloat[\scriptsize $\epsilon = 10^{-4}$.]{
        \includegraphics[scale=.45]{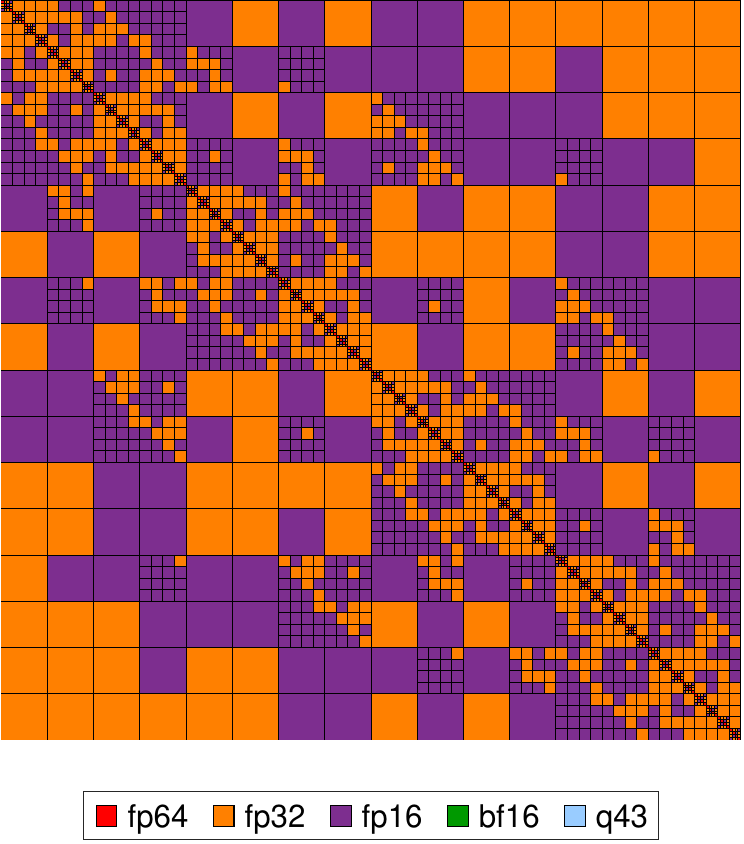}
        \label{fig:hmat0_prec_color}
        }
        \subfloat[\scriptsize $\epsilon = 10^{-2}$.]{
        \includegraphics[scale=.45]{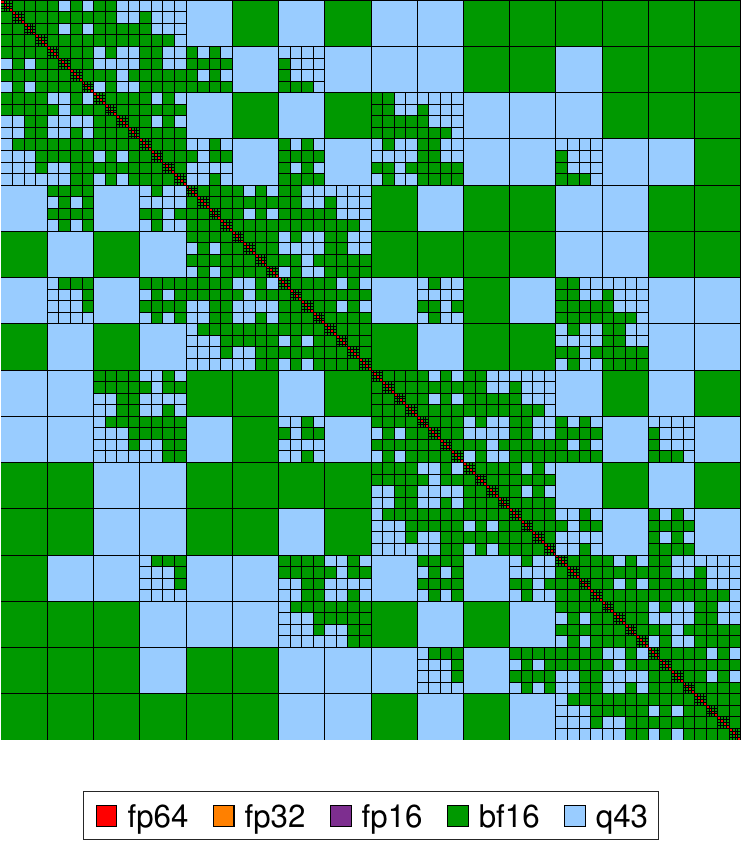}
        \label{fig:hmat1_prec_color}
        }
        \caption{Colors indicate the precision used in each block of $\widehat{H}$ for different choices of $\epsilon$.}
        \label{fig:prec_color}
    \end{figure}

\subsection{Adaptive mixed precision \texorpdfstring{$\mathcal{H}_h$}{Hh}-matrix-vector product}
The adaptive mixed precision $\mathcal{H}_h$-matrices can be applied to column vectors to perform fast matrix-vector products as described in \Cref{alg:algo3}. Its complexity is $\mathcal{O}\bkt{p N \log{N}}$, where $p$ denotes the maximum rank of the admissible blocks.

We now investigate how the working precision $u$ in \Cref{alg:algo3} (this $u$ may differ from the $u$ used in \Cref{alg:algo2}) should be chosen relative to the prescribed accuracy $\epsilon$ so that the finite precision error in the computed matrix-vector product does not exceed the approximation error. We propose the following theorem, which provides a backward error bound for adaptive mixed precision $\mathcal{H}_h$-matrix-vector product.

\begin{theorem} \label{thm:mvp}
    Let $\widehat{H}$ denote its adaptive mixed precision $\mathcal{H}_h$-matrix representation obtained by \Cref{alg:algo2}. If $b = \widehat{H} x$ is computed using \Cref{alg:algo3} with working precision $u \lesssim \epsilon$, the computed result $\widehat{b}$ satisfies
    \begin{align*}
        \widehat{b} = \emph{fl} (\widehat{H}x) = \bkt{H + \Delta H}, \quad \magn{\Delta H} \leq   2(\sqrt{\bkt{\ell C^{\prime} + C^{\prime\prime} + \bkt{L-\ell} C^{\prime\prime\prime}}} 2 + 1) \epsilon \magn{H},
    \end{align*}
    where the constants $C^{\prime}$, $C^{\prime\prime}$, and $C^{\prime\prime\prime}$ are given in \Cref{thm:amp_throrem}.
\end{theorem}

\begin{proof}
    The proof can be found in \Cref{sec:proof_mvp}.
\end{proof}

 \begin{algorithm}[H]
 \small
	\caption{$\widehat{H}$ accelerated \mvp, i.e., $b = \widehat{H} x$}\label{alg:algo3}
	\begin{algorithmic}[1]
        \State \textbf{Input:} $\widehat{H}$, $x \in \Rb^{N \times 1}$. \qquad \textbf{Output:} $b \in \Rb^{N \times 1}$.
        \State $b \gets 0 \in \Rb^{N \times 1}$
        
\For{\texttt{$l=1:\ell$}} 
				\For{\texttt{$i=1:2^{dl}$}}
					\For{$\mathcal{C}^{\bkt{l}}_{j} \in$ $\mathcal{IL}_s (\mathcal{C}^{\bkt{l}}_{i})$}
                        \State $b^{\bkt{l}}_I \gets b^{\bkt{l}}_I + \widehat{U}^{\bkt{l}}_I (\widehat{V}^{\bkt{l}}_J)^* x^{\bkt{l}}_J$ \Comment{Compute in working precision $u$.}
					\EndFor
				\EndFor
			\EndFor

          		\For{\texttt{$i=1:2^{d\ell}$}}
					\For{$\mathcal{C}^{\bkt{\ell}}_{j} \in$ $\mathcal{N}_s (\mathcal{C}^{\bkt{\ell}}_{i})$}
                        \State $b^{\bkt{\ell}}_I \gets b^{\bkt{\ell}}_I + \widehat{U}^{\bkt{\ell}}_I (\widehat{V}^{\bkt{\ell}}_J)^* x^{\bkt{\ell}}_J$ \Comment{Compute in working precision $u$.}
					\EndFor
				\EndFor  

            \For{\texttt{$l=\ell+1:L$}} 
				\For{\texttt{$i=1:2^{dl}$}}
					\For{$\mathcal{C}^{\bkt{l}}_{j} \in$ $\mathcal{IL}_w (\mathcal{C}^{\bkt{l}}_{i})$}
						\State $b^{\bkt{l}}_I \gets b^{\bkt{l}}_I + \widehat{U}^{\bkt{l}}_I (\widehat{V}^{\bkt{l}}_J)^* x^{\bkt{l}}_J$ \Comment{Compute in working precision $u$.}
					\EndFor
				\EndFor
			\EndFor

            \For{\texttt{$i=1:2^{dL}$}}
                \State $b^{\bkt{L}}_I \gets b^{\bkt{L}}_I + \widehat{H}^{\bkt{L}}_{I,I} x^{\bkt{L}}_I$ \Comment{Compute in working precision $u$.}
            \EndFor
	\end{algorithmic}
\end{algorithm}

\section{Numerical results} \label{sec:num_results}
This section presents numerical experiments involving kernel matrices commonly encountered in BEM and kernel methods to validate the theoretical results and to evaluate storage gains in the adaptive mixed precision representation. We further examine the backward errors of the matrix-vector product.

Let $\{x_i \}_{i=1}^N$ denote the locations of $N$ particles (or points), where $x_i \in \Rb^{d}$. We consider two distinct sets of particles: 

\begin{align*}
P_{\subsquare} &:\ \text{consisting of $N$ particles uniformly distributed within the hypercube } [-1,1]^d.\\
P_{\subcircle} &:\ \text{consisting of $N$ particles distributed on the surface of the unit hypersphere}.
\end{align*}

In each figure associated with an experiment, the particle distributions considered are indicated within parentheses at the end.
The numerical experiments are conducted in both two and three dimensions $(d = 2,3)$, with the admissibility parameter $\eta = \sqrt{d}$. However, it is worth noting that our publicly available code can handle any values of $d$ and $\eta$. \textbf{The numerical results associated with the particle set $P_{\subcircle}$ are presented in \Cref{sec:circ_particle_dist}}.

In \textbf{two dimensions}, we consider the kernel function

\begin{itemize}
    \item $f\bkt{r} = \log \bkt{r}$, ($2$D single-layer Laplacian)
\end{itemize}

In \textbf{three dimensions}, we consider the following kernel functions
$$
\begin{tabular}{ll}
$\bullet\; f(r) = {1}/{r}$, (3D single-layer Laplacian)
& $\bullet\; f(r) = {1}/{r^2}$ \\[8pt]
$\bullet\; f(r) = \exp (-{r^2}/{2})$, (Gaussian) (see \Cref{sec:Gaussian_kernel})
& $\bullet\; f(r) = \exp(-r)$, (Matérn)
\end{tabular}
$$
Here, $r$ denotes the Euclidean distance between two points, i.e., $r = \magn{x_i - x_j}_2$.

The $(i,j)^{th}$ entry of the kernel matrix $H \in \Rb^{N \times N}$ is given by 

\begin{align} \label{eq:kernel_mat}
    H_{i,j} = f \bkt{\magn{x_i-x_j}_2} = f \bkt{r}.
\end{align} 
If $f$ is \emph{singular}, the diagonal entries are set to zero, i.e., $H_{i,i} = 0$.

The adaptive mixed precision $\mathcal{H}_h$-matrix representations of the kernel matrix $H$ are constructed based on \Cref{alg:algo2}. The admissible blocks are compressed using the truncated SVD with a prescribed target accuracy $\epsilon$. The precisions are chosen from the set $\mathcal{U} = \{ \text{fp64}, \text{fp32}, \text{fp16}, \text{bf16}, \text{q43} \}$ (see \Cref{tab:precisions}). We simulate the lower precisions using the \texttt{chop} function of \cite{higham2019simulating}.

All the algorithms are implemented in MATLAB (R2024a), and the experiments are performed on an Intel Xeon Gold 2.6 GHz processor without parallelization.

\subsection{Global error in \texorpdfstring{$\widehat{H}$}{Hhat} representation} \label{subsec:global_error}
We illustrate how the relative global error, $\magn{H - \widehat{H}}/\magn{H}$, of adaptive mixed precision representations $\widehat{H}$ constructed using \Cref{alg:algo2} varies for different choices of the target accuracy $\epsilon$. The experiments are performed with different tree depths $L$, and we also plot the corresponding global relative error bound derived from \cref{eq:mp_thm}. Setting $\eta = \sqrt{d}$ simplifies \cref{eq:mp_thm} to \cref{eq:error_bound_2d} in two dimensions and to \cref{eq:error_bound_3d} in three dimensions, respectively. 

\begin{align} \label{eq:error_bound_2d}
         \magn{H - \widehat{H}} \lesssim (\sqrt{\bkt{27 \ell  + 8 + 3 \bkt{L-\ell}}} 2 + 1) \epsilon \magn{H},
\end{align}

\begin{align} \label{eq:error_bound_3d}
         \magn{H - \widehat{H}} \lesssim (\sqrt{\bkt{189 \ell + 26 + 7 \bkt{L-\ell}}} 2 + 1) \epsilon \magn{H}.
\end{align}

\Cref{fig:log_err_2d,fig:1r_err_3d,fig:1r2_err_3d,fig:gauss_err_3d,fig:mat_err_3d} illustrate the global relative error for varying values of $\epsilon$. The $x$-axis indicates the values of $\epsilon$, while the $y$-axis indicates the global relative error. We plot the global relative errors for the adaptive mixed precision representation of $\mathcal{H}_h$-matrices for the choices $\ell = L-1$ and $\ell = L$. The adaptive mixed precision schemes are denoted by ``amp'' in the legend. The entries ``eb1'' and ``eb2'' in the legend denote the relative error bounds corresponding to $\ell = L-1$ and $\ell = L$, respectively, obtained by substituting the respective values of $\ell$ into \cref{eq:error_bound_2d,eq:error_bound_3d}. It can be observed that $\text{eb1} < \text{eb2}$. For comparison, the global relative errors obtained in the uniform double precision (fp64) schemes are also shown. In $2$D (\Cref{fig:log_err_2d}) and $3$D (\Cref{fig:1r_err_3d,fig:1r2_err_3d,fig:gauss_err_3d,fig:mat_err_3d}), we set $N=25600$ and $N=64000$, respectively. The tree depths $L$ and particle distributions used are shown in each figure. 

All figures indicate that the error bound is satisfied for both $\ell = L-1$ and $\ell = L$, as expected.

    \begin{figure}
        \centering
        \subfloat[\scriptsize $L=3$ $(P_{\subsquare})$]{
        \includegraphics[scale=.35]{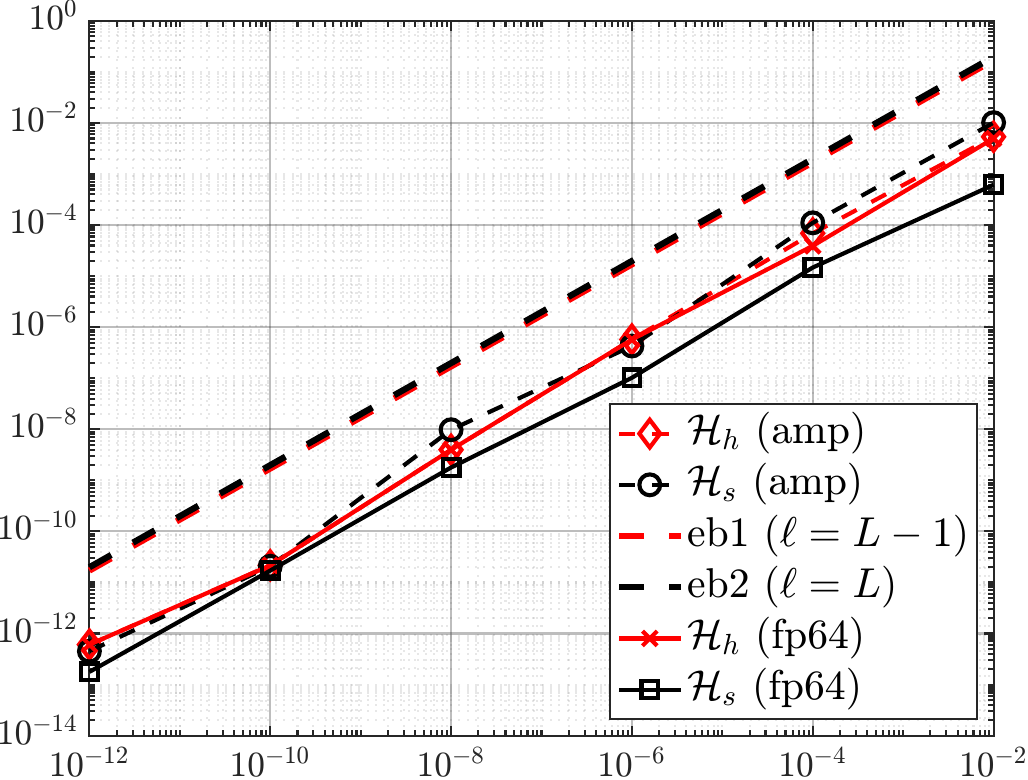}
        \label{fig:log_err_1_2d_p1}
        }
        \subfloat[\scriptsize $L=5$ $(P_{\subsquare})$]{
        \includegraphics[scale=.35]{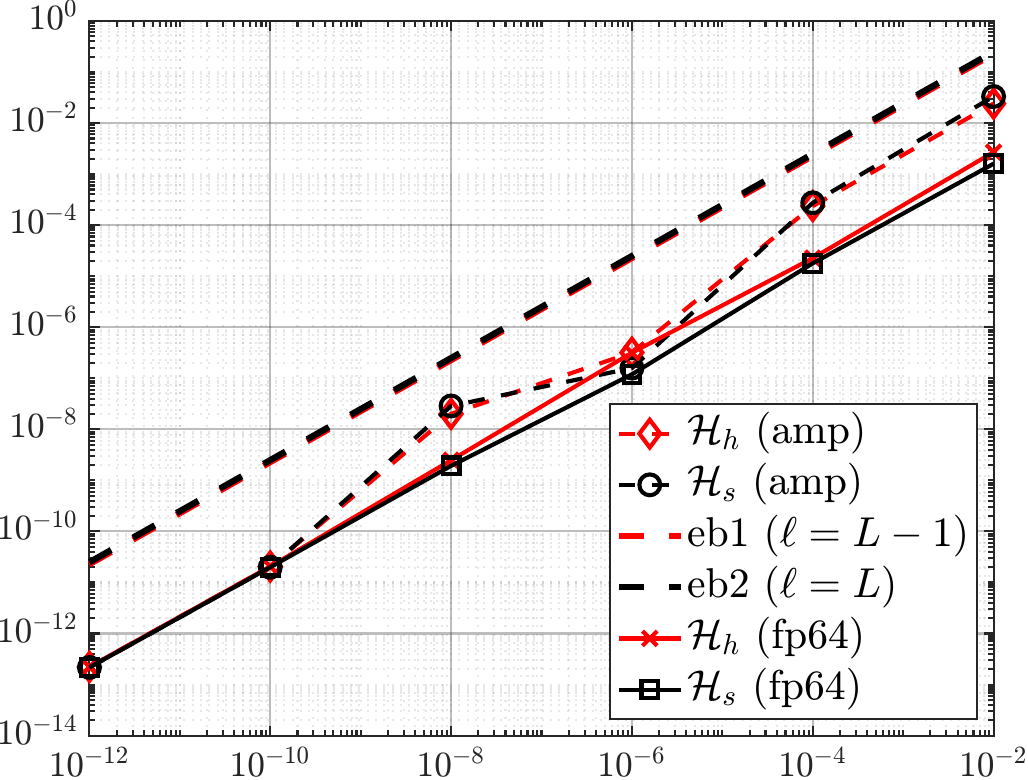}
        \label{fig:log_err_2_2d_p1}
        }
        \caption{Global relative error ($y$-axis) of adaptive mixed precision $\mathcal{H}_h$-matrices for different target accuracies ($x$-axis). We consider the kernel $\log \bkt{r}$ in $2$D, with $N=25600$.}
        \label{fig:log_err_2d}
    \end{figure}

    \begin{figure}
        \centering
        \subfloat[\scriptsize $L=3$ $(P_{\subsquare})$]{
        \includegraphics[scale=.35]{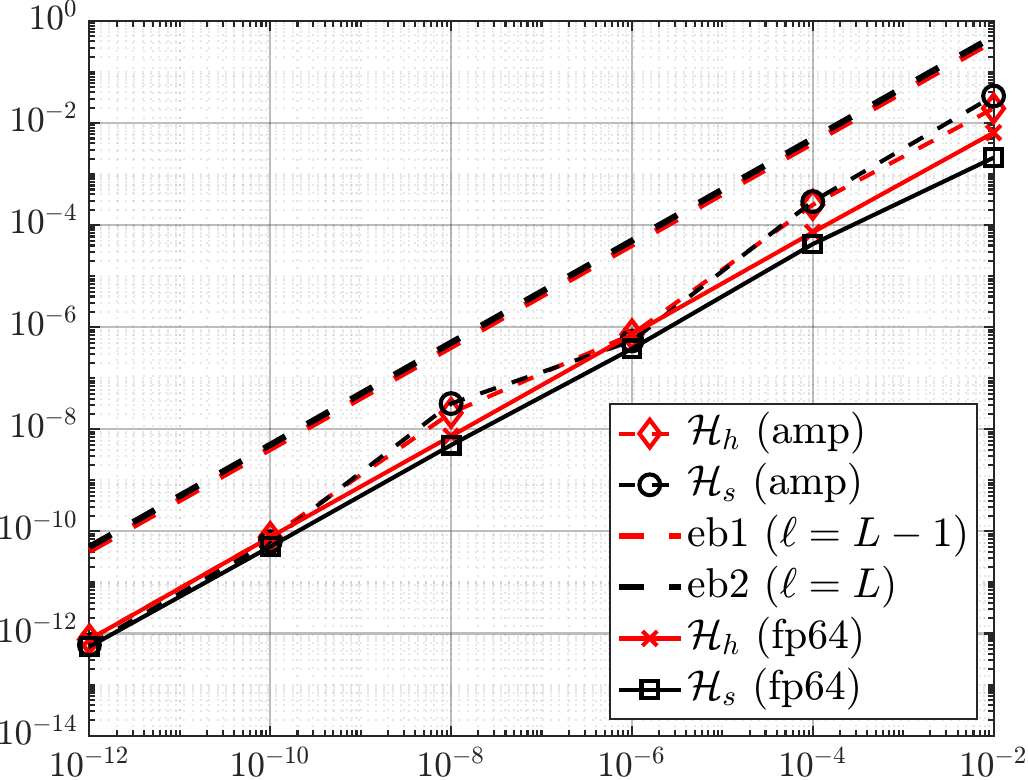}
        \label{fig:1r_err_1_3d_p1}
        }
        \subfloat[\scriptsize $L=4$ $(P_{\subsquare})$]{
        \includegraphics[scale=.35]{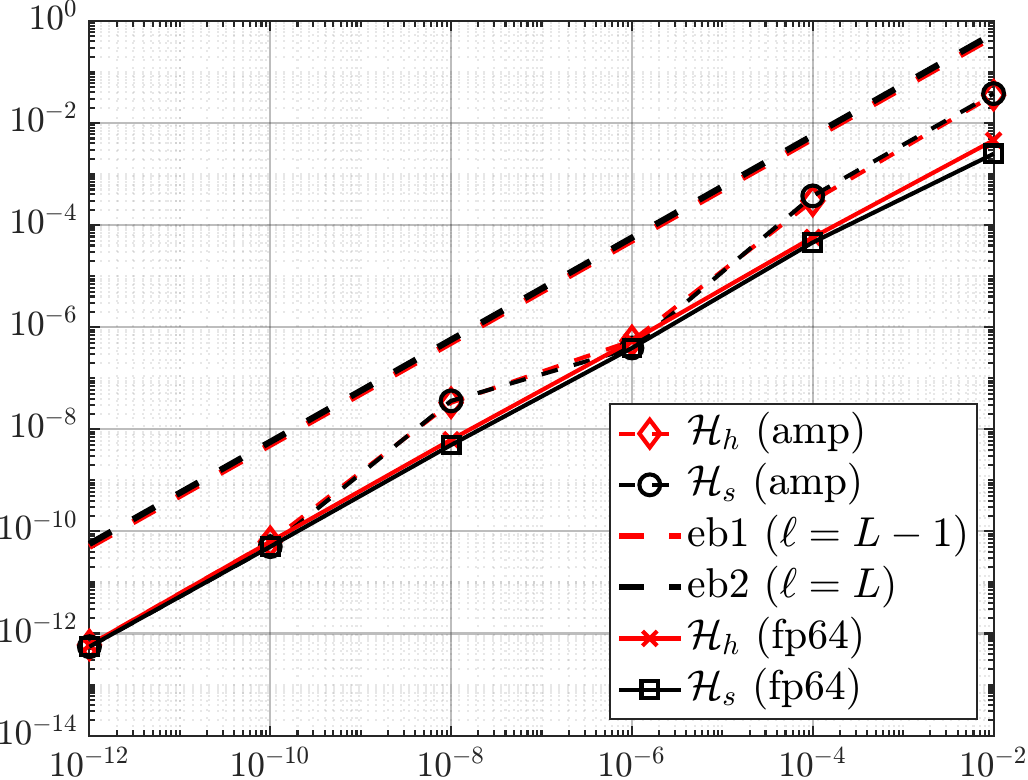}
        \label{fig:1r_err_2_3d_p1}
        }
        \caption{Global relative error ($y$-axis) of adaptive mixed precision $\mathcal{H}_h$-matrices for different target accuracies ($x$-axis). We consider the kernel $1/r$ in $3$D, with $N=64000$.}
        \label{fig:1r_err_3d}
    \end{figure}

    \begin{figure}
        \centering
        \subfloat[\scriptsize $L=3$ $(P_{\subsquare})$]{
        \includegraphics[scale=.35]{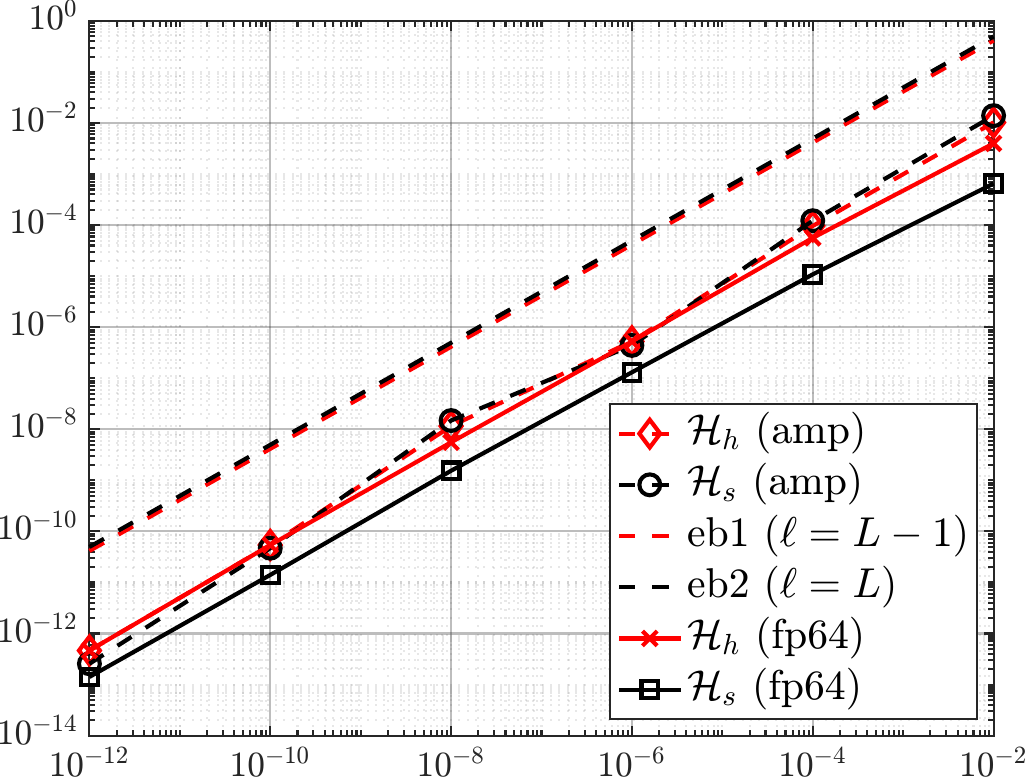}
        \label{fig:1r2_err_1_3d_p1}
        }
        \subfloat[\scriptsize $L=4$ $(P_{\subsquare})$]{
        \includegraphics[scale=.35]{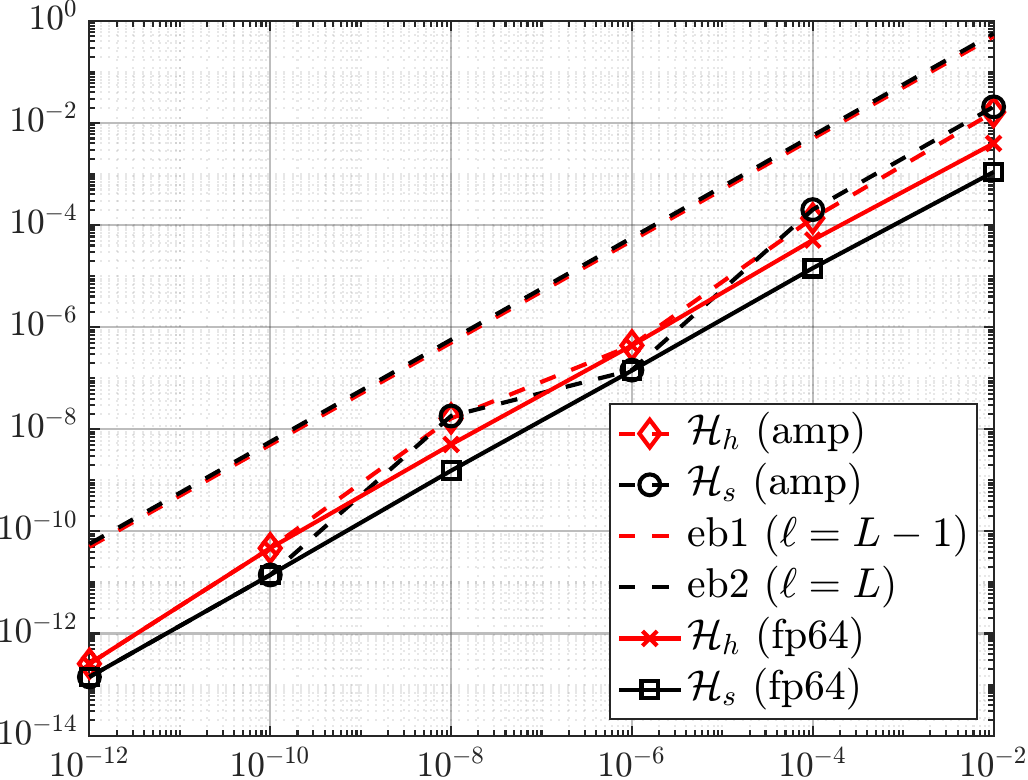}
        \label{fig:1r2_err_2_3d_p1}
        }
        \caption{Global relative error ($y$-axis) of adaptive mixed precision $\mathcal{H}_h$-matrices for different target accuracies ($x$-axis). We consider the kernel $1/r^2$ in $3$D, with $N=64000$.}
        \label{fig:1r2_err_3d}
    \end{figure}

    \begin{figure}
        \centering
        \subfloat[\scriptsize $L=3$ $(P_{\subsquare})$]{
        \includegraphics[scale=.35]{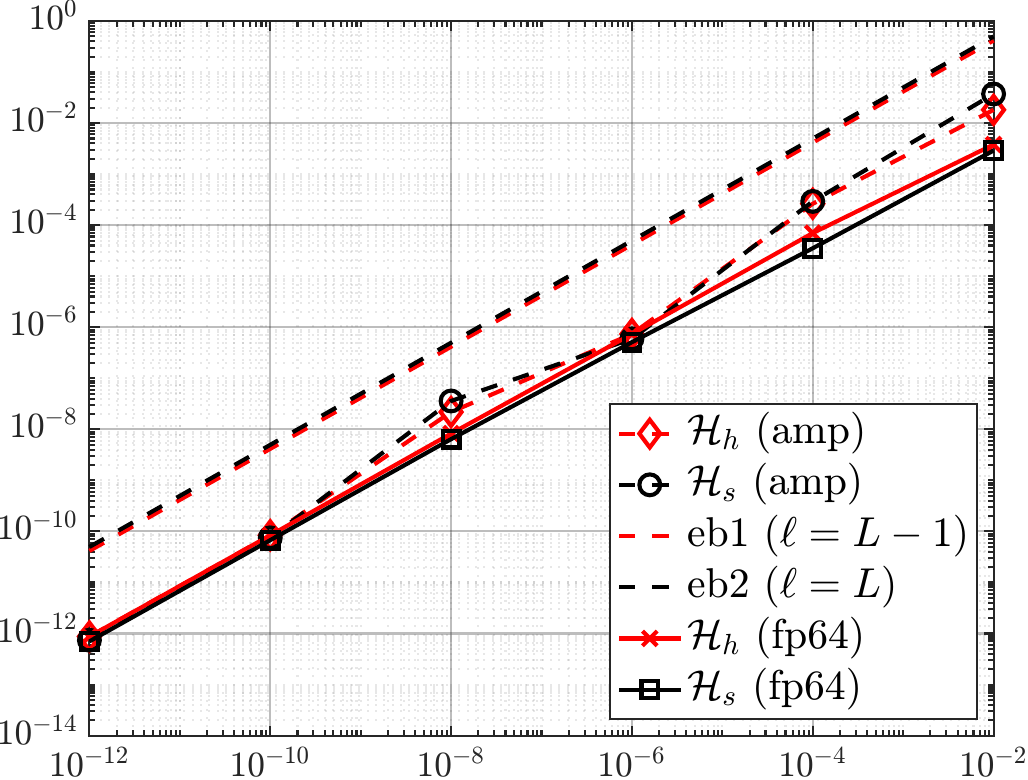}
        \label{fig:mat_err_1_3d_p1}
        }
        \subfloat[\scriptsize $L=4$ $(P_{\subsquare})$]{
        \includegraphics[scale=.35]{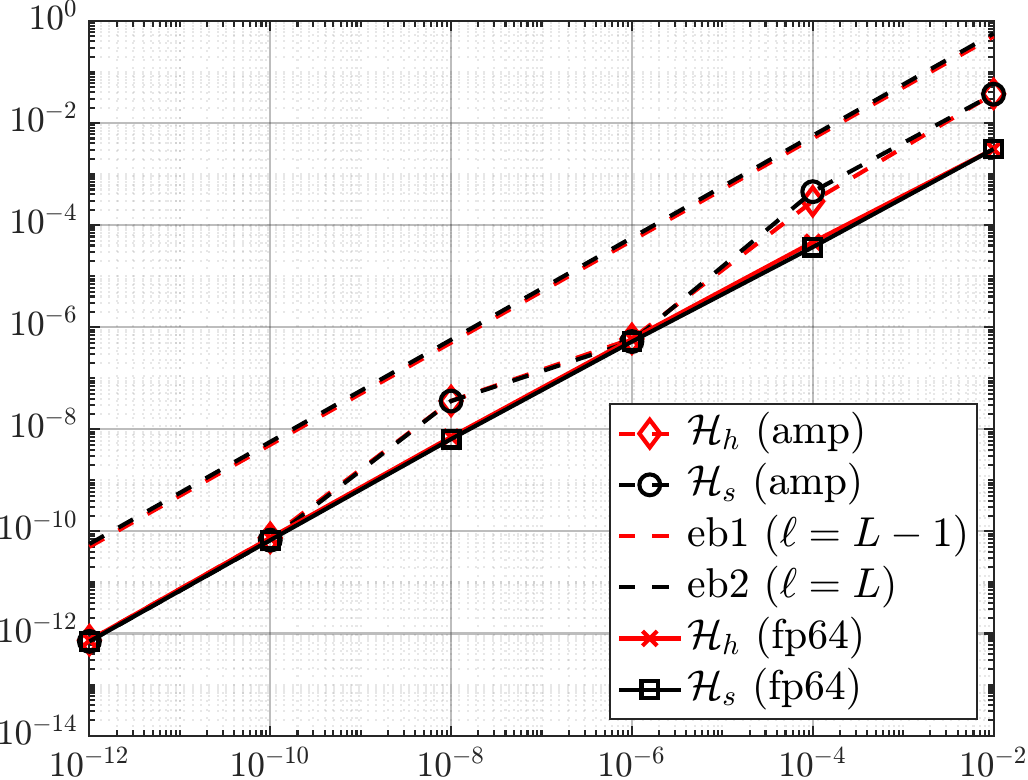}
        \label{fig:mat_err_2_3d_p1}
        }
        \caption{Global relative error ($y$-axis) of adaptive mixed precision $\mathcal{H}_h$-matrices for different target accuracies ($x$-axis). We consider the Matérn kernel in $3$D, with $N=64000$.}
        \label{fig:mat_err_3d}
    \end{figure}

\subsection{Storage gains in \texorpdfstring{$\widehat{H}$}{Hhat} representation}
In the adaptive mixed precision representation (\Cref{alg:algo2}), the low-rank factors $\widehat{U}^{\bkt{l}}_I \in \mathbb{R}^{\abs{I} \times p}, \widehat{V}^{\bkt{l}}_J \in \mathbb{R}^{\abs{J} \times p}$ of an admissible block $I \times J$ are represented in precision $u_{i,j}^{(l)}$, and the dense diagonal blocks are stored in $u =\text{fp64}$. The required bits for a low-rank (admissible) block and an $m \times n$ dense  diagonal block are given by $p\bkt{\abs{I} + \abs{J}} \times (b_{i,j}^{(l)}/64) \space \text{ and } \space mn, \text{ respectively}$, where $b_{i,j}^{(l)}$ denotes the number of bits corresponding to $u_{i,j}^{(l)}$. We divide by $64$ just to maintain consistency with the previously defined storage requirement (see \cref{eq:LR_storage}), where fp64 is the working precision. It has no impact, as the storage gain is evaluated as a ratio. Storing the low-rank factors in a precision lower than the working precision reduces the total number of bits required to store the representation $\widehat{H}$.

The storage cost accounts for the total number of bits used to store all the low-rank blocks (across all the levels) and the leaf-level dense diagonal blocks. Let $\mathcal{S}_{\mathcal{H}_h}^{\text{amp}}$, $\mathcal{S}_{\mathcal{H}_h}^{\text{fp64}}$, and $\mathcal{S}_{\mathcal{H}_s}^{\text{fp64}}$ denote the storage costs of the adaptive mixed precision $\mathcal{H}_h$-matrix, the uniform double precision $\mathcal{H}_h$-matrix and the uniform double precision $\mathcal{H}_s$-matrix, respectively. In \Cref{fig:log_mem_2d,fig:1r_mem_3d,fig:1r2_mem_3d,fig:gauss_mem_3d,fig:mat_mem_3d}, we plot the ratios $\mathcal{S}_{\mathcal{H}_h}^{\text{fp64}}/\mathcal{S}_{\mathcal{H}_h}^{\text{amp}}$ (blue line), $\mathcal{S}_{\mathcal{H}_s}^{\text{fp64}}/\mathcal{S}_{\mathcal{H}_h}^{\text{amp}}$ (magenta line) for varying values of $\epsilon$ to illustrate the storage gains achieved by the adaptive mixed precision scheme. We consider $N=25600$ and $102400$ in $2$D (\Cref{fig:log_mem_2d}), and $N=64000$ and $125000$ in $3$D (\Cref{fig:1r_mem_3d,fig:1r2_mem_3d,fig:gauss_mem_3d,fig:mat_mem_3d}), with $\ell=L-1$. The corresponding tree depth $L$ and particle distribution are indicated in each figure. The ratio $\mathcal{S}_{\mathcal{H}_h}^{\text{fp64}}/\mathcal{S}_{\mathcal{H}_h}^{\text{amp}}$ (blue line) equals $1$ for $10^{-12} \leq \epsilon \leq 10^{-10}$, since in this range the adaptive mixed precision scheme selects only double precision from $\mathcal{U}$, resulting in no storage reduction. The storage gains improve as $\epsilon$ grows. Note that in all figures the magenta line always lies above the blue line, i.e.,
$\mathcal{S}_{\mathcal{H}_s}^{\text{fp64}}/\mathcal{S}_{\mathcal{H}_h}^{\text{amp}} > \mathcal{S}_{\mathcal{H}_h}^{\text{fp64}}/\mathcal{S}_{\mathcal{H}_h}^{\text{amp}} \implies \mathcal{S}_{\mathcal{H}_s}^{\text{fp64}} > \mathcal{S}_{\mathcal{H}_h}^{\text{fp64}}$. For instance, as shown in \Cref{fig:log_mem_1_2d}, when $\epsilon = 10^{-2}$, $\mathcal{S}_{\mathcal{H}_s}^{\text{fp64}} / \mathcal{S}_{\mathcal{H}_h}^{\text{fp64}} \approx 4.8/2 = 2.4$. This again highlights the storage advantage of our hybrid hierarchical matrix in uniform (double) precision, as demonstrated in \Cref{subsec:Hs_vs_Hh}. The gap between the blue and magenta lines widens with $\epsilon$, showcasing that the hybrid hierarchical matrix becomes increasingly effective as $\epsilon$ grows. The storage gains of the adaptive mixed precision $\mathcal{H}_h$-matrix over the uniform double precision $\mathcal{H}_s$-matrix result from two reasons: the use of our hybrid admissibility condition and the use of lower precisions.

    \begin{figure}[H]
        \centering
        \subfloat[\scriptsize $N=25600, L=5$ $(P_{\subsquare})$]{
        \includegraphics[scale=.34]{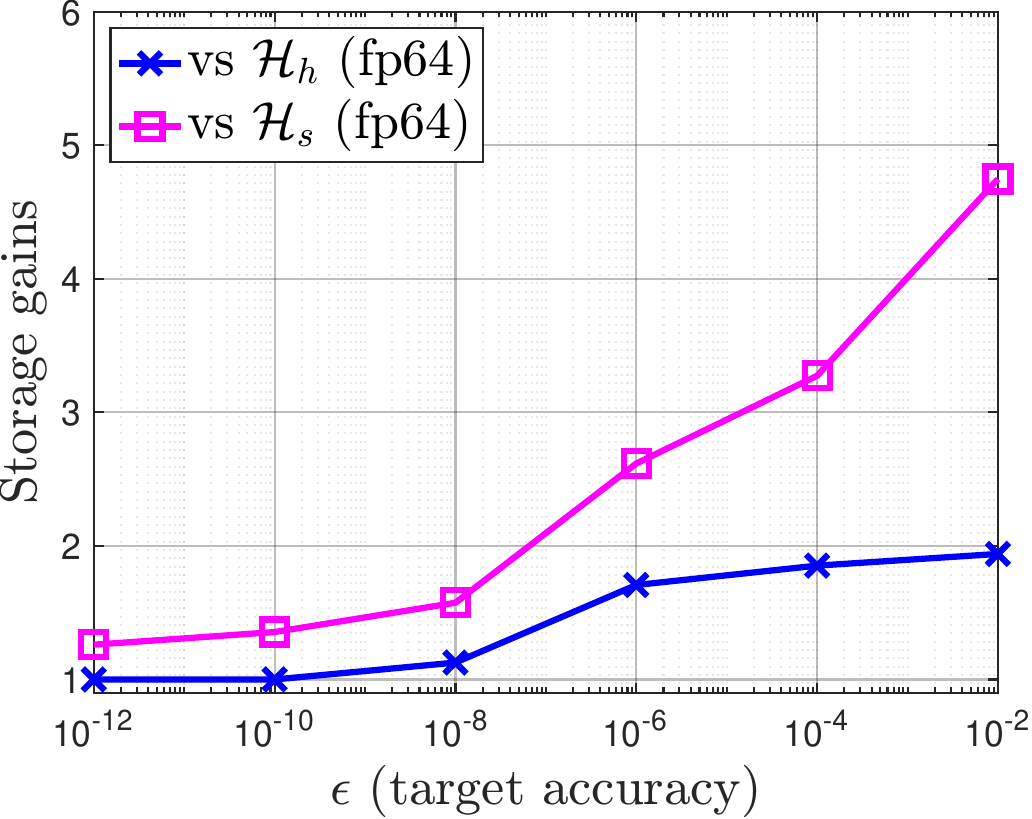}
        \label{fig:log_mem_1_2d}
        }
        \subfloat[\scriptsize $N=102400, L=6$ $(P_{\subsquare})$]{
        \includegraphics[scale=.34]{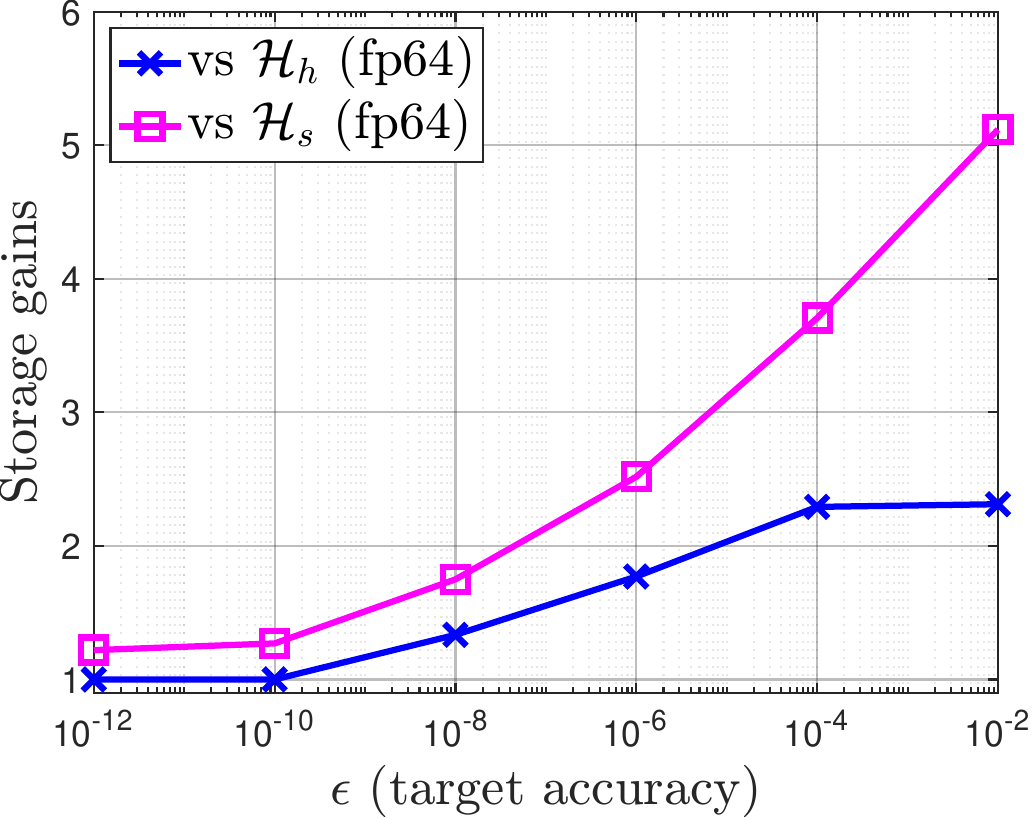}
        \label{fig:log_mem_2_2d}
        }
        \caption{Storage gains of adaptive mixed precision $\mathcal{H}_h$-matrices compared with uniform double precision $\mathcal{H}_h$ and $\mathcal{H}_{s}$ matrices for the kernel $\log \bkt{r}$ in $2$D.}
        \label{fig:log_mem_2d}
    \end{figure}

    \begin{figure}[H]
        \centering
        \subfloat[\scriptsize $N=64000, L=3$ $(P_{\subsquare})$]{
        \includegraphics[scale=.34]{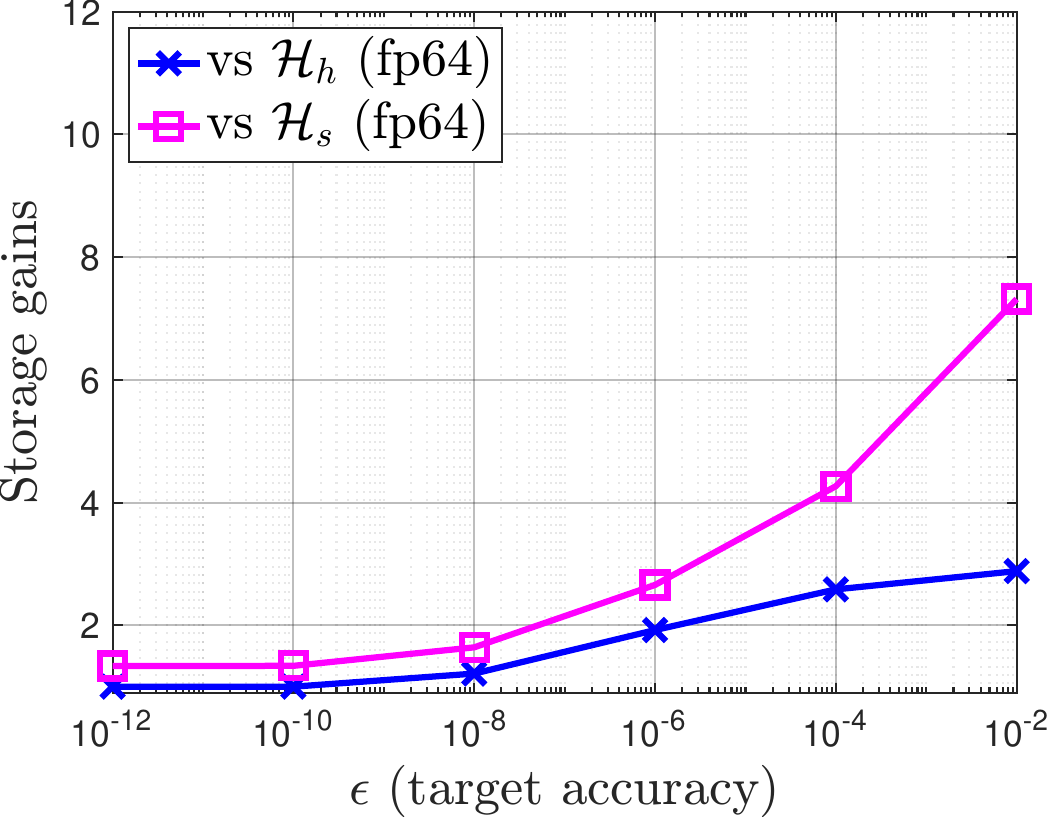}
        \label{fig:1r_mem_1_3d_p1}
        }
        \subfloat[\scriptsize $N=125000, L=4$ $(P_{\subsquare})$]{
        \includegraphics[scale=.34]{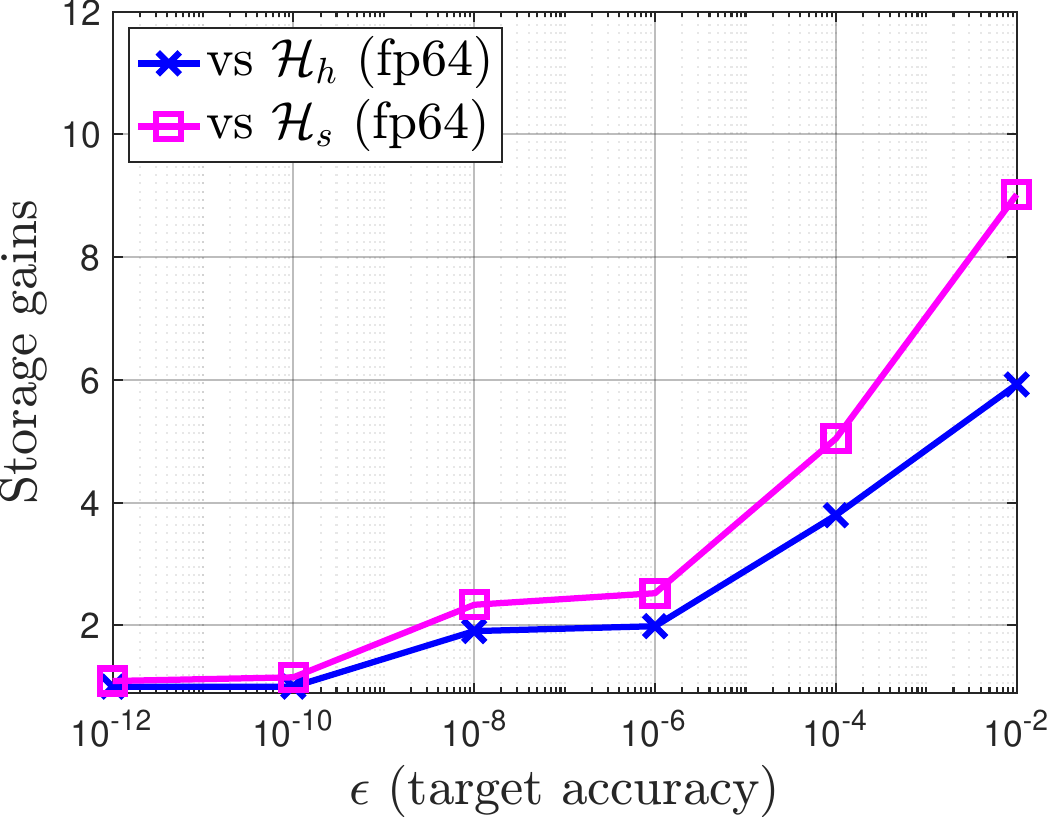}
        \label{fig:1r_mem_2_3d_p1}
        }
        \caption{Storage gains of adaptive mixed precision $\mathcal{H}_h$-matrices compared with uniform double precision $\mathcal{H}_h$ and $\mathcal{H}_{s}$ matrices for the kernel $1/r$ in $3$D.}
        \label{fig:1r_mem_3d}
    \end{figure}

    \begin{figure}[H]
        \centering
        \subfloat[\scriptsize $N=64000, L=3$ $(P_{\subsquare})$]{
        \includegraphics[scale=.34]{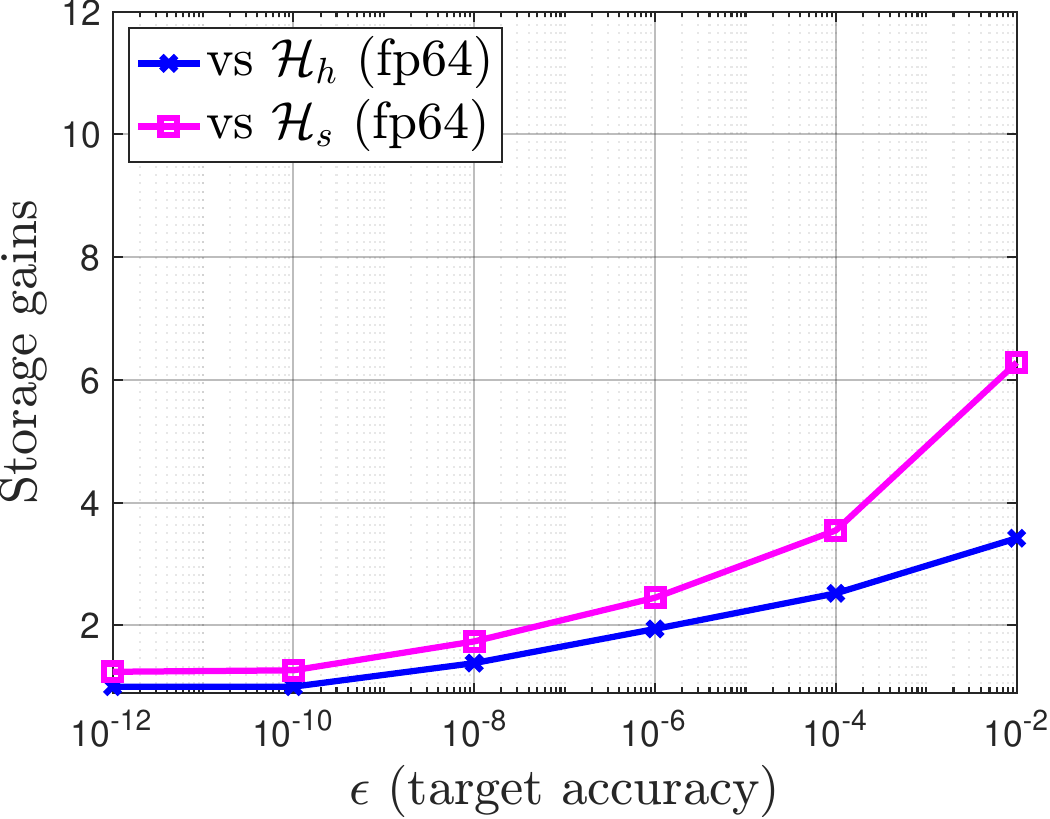}
        \label{fig:1r2_mem_1_3d_p1}
        }
        \subfloat[\scriptsize $N=125000, L=4$ $(P_{\subsquare})$]{
        \includegraphics[scale=.34]{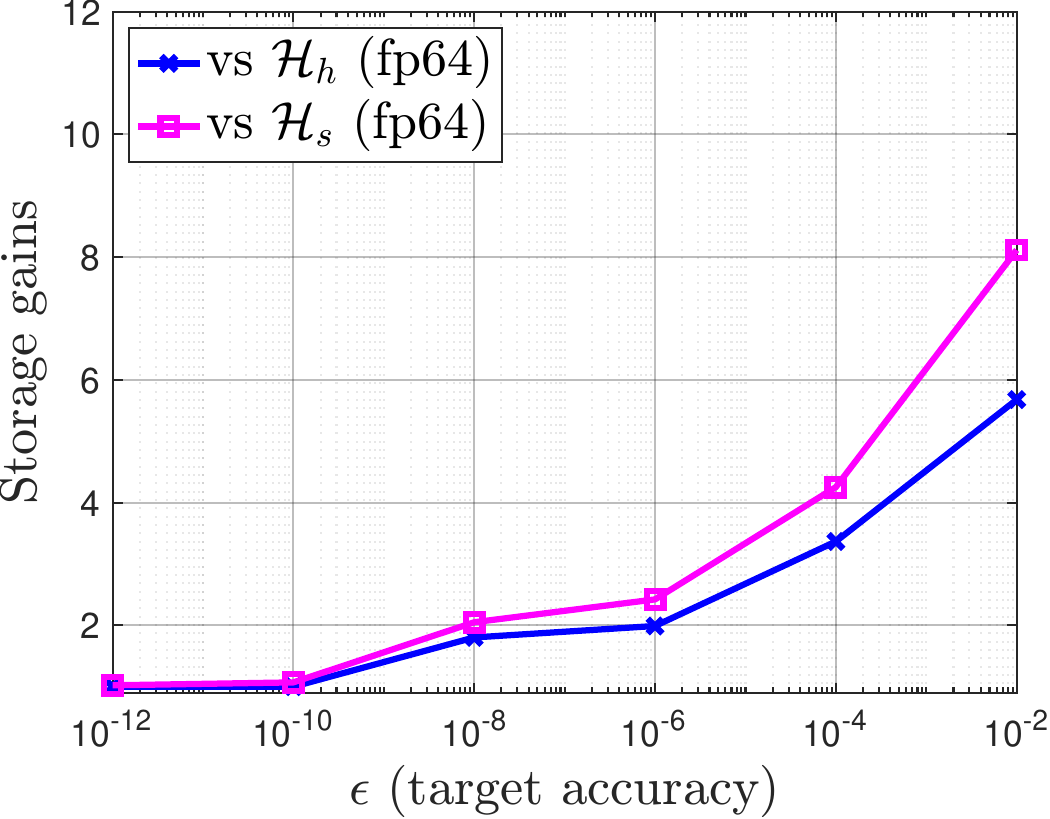}
        \label{fig:1r2_mem_2_3d_p1}
        }
        \caption{Storage gains of adaptive mixed precision $\mathcal{H}_h$-matrices compared with uniform double precision $\mathcal{H}_h$ and $\mathcal{H}_{s}$ matrices for the kernel $1/r^2$ in $3$D.}
        \label{fig:1r2_mem_3d}
    \end{figure}

    \begin{figure}[H]
        \centering
        \subfloat[\scriptsize $N=64000, L=3$ $(P_{\subsquare})$]{
        \includegraphics[scale=.34]{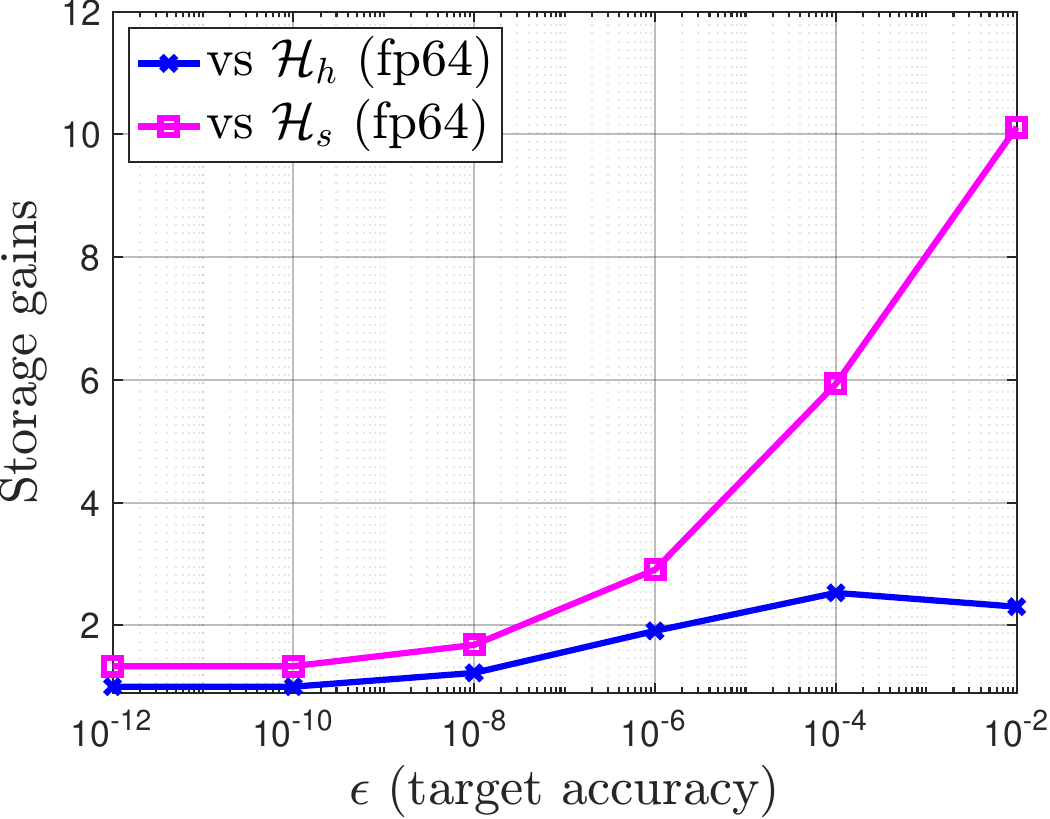}
        \label{fig:mat_mem_1_3d_p1}
        }
        \subfloat[\scriptsize $N=125000, L=4$ $(P_{\subsquare})$]{
        \includegraphics[scale=.34]{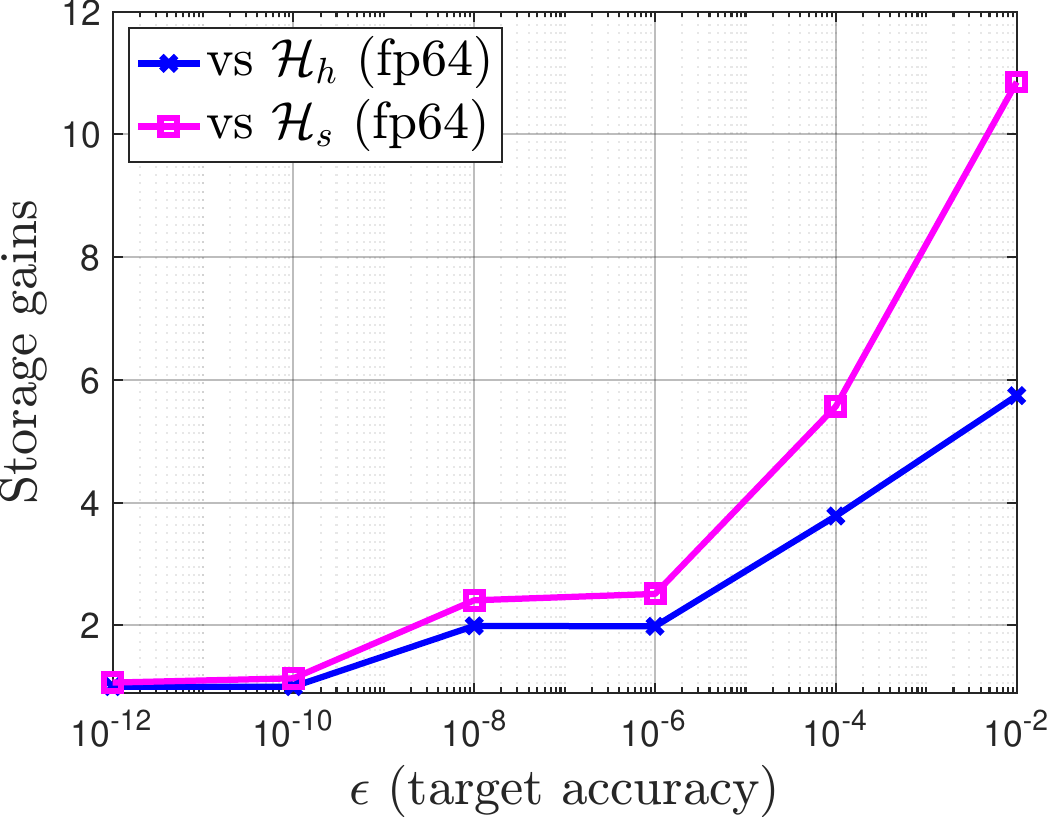}
        \label{fig:mat_mem_2_3d_p1}
        }
        \caption{Storage gains of adaptive mixed precision $\mathcal{H}_h$-matrices compared with uniform double precision $\mathcal{H}_h$ and $\mathcal{H}_{s}$ matrices for the Matérn kernel in $3$D.}
        \label{fig:mat_mem_3d}
    \end{figure}

 The storage gains plots show that the adaptive mixed precision $\mathcal{H}_h$-matrix can significantly reduce storage costs compared with the uniform double precision $\mathcal{H}_s$-matrix while maintaining accuracy. For example, for the Matérn and Gaussian kernels with $\epsilon = 10^{-2}$, the storage gains can reach $\sim 11 \times$ (see \Cref{fig:mat_mem_2_3d_p1,fig:gauss_mem_2_3d_p1}). Notably, these gains do not come at the expense of significant accuracy; the corresponding global relative errors are shown in \Cref{subsec:global_error}. 

\subsection{Backward error of matrix-vector product}
We examine the backward error of the adaptive mixed precision $\mathcal{H}_h$-matrix-vector product. We first construct the $\widehat{H}$ representation using \Cref{alg:algo2}, followed by the matrix-vector product $b = \widehat{H}x$ using \Cref{alg:algo3}.

We perform an experiment to validate \Cref{thm:mvp}. The dense kernel matrix $H \in \mathbb{R}^{N \times N}$ is formed using the kernel $1/r$, and the column vector $x \in \mathbb{R}^{N \times 1}$ is generated from a uniform distribution in $(0,1)$. We plot the relative backward error, i.e., $(\magn{Hx - b})/(\magn{H} \magn{x})$, for different target accuracies $\epsilon$. The error bound ``eb'' is plotted based on \Cref{thm:mvp}. We consider three different working precisions $u$: fp64, fp32, and fp16 in \Cref{alg:algo3}. The results are shown in \Cref{fig:mvp_error_3d}.

    \begin{figure}[H]
        \centering
        \subfloat[\scriptsize $N=8000, L=2$ $(P_{\subsquare})$]{
        \includegraphics[scale=.36]{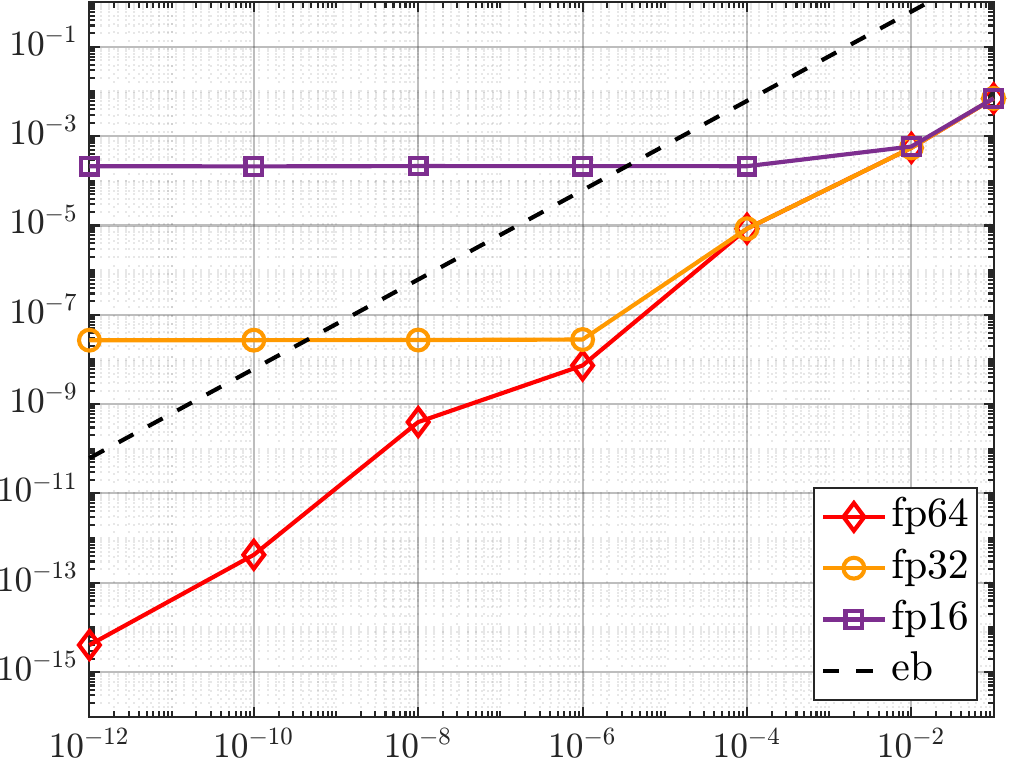}
        \label{fig:mvp_1_3d_p1}
        }
        \subfloat[\scriptsize $N=64000, L=3$ $(P_{\subsquare})$]{
        \includegraphics[scale=.36]{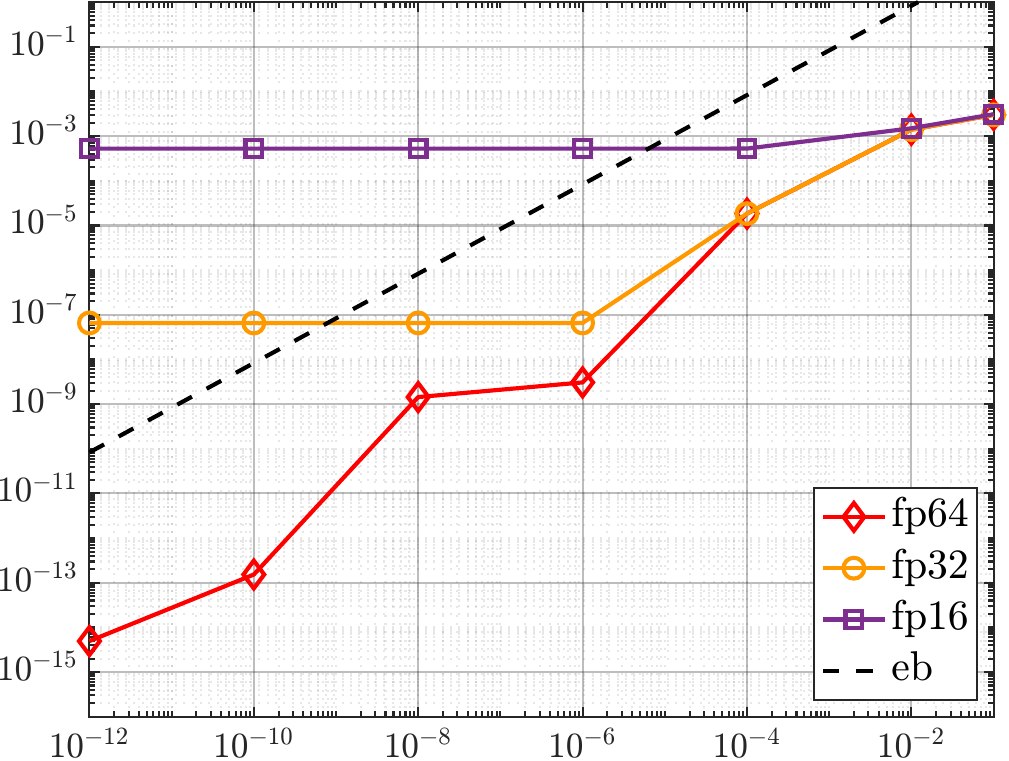}
        \label{fig:mvp_2_3d_p1}
        }
        \caption{Relative backward error of matrix-vector product ($y$-axis) using adaptive mixed precision $\mathcal{H}_h$-matrices for different target accuracies $\epsilon$ ($x$-axis). We consider the kernel $1/r$ in $3$D.}
        \label{fig:mvp_error_3d}
    \end{figure}

We can see that one should choose the working precision $u \lesssim \epsilon$ (more precisely, $u \lesssim \epsilon/N$; this is why we present results for $N=8000$ and $N=64000$) in order to avoid the effects of finite precision in the matrix-vector product computation. The larger $\epsilon$ becomes, the lower the precision that can be used without affecting the backward error. We observe similar results for other kernel functions.

The $\widehat{H}$-matrix–vector products could achieve a speedup matching the storage gains, assuming that each arithmetic operation is proportional to the number of bits involved. However, this is a more complex issue, as the time cost of various arithmetic operations depends on the hardware architecture and other factors. The $\widehat{H}$-matrix-vector product can also be leveraged to develop fast iterative solvers.

\section{Conclusion} \label{sec:conclusion}
We have proposed a novel hybrid admissibility condition that exploits the standard admissibility condition at coarser levels and the weak admissibility condition at finer levels of the balanced $2^d$-tree. Based on our hybrid admissibility condition, we have introduced hybrid hierarchical matrices, denoted as $\mathcal{H}_h$-matrices. Standard admissibility-based $\mathcal{H}_s$-matrices and weak admissibility-based HODLR matrices can be viewed as special cases of $\mathcal{H}_h$-matrices. We have also derived a criterion under which $\mathcal{H}_h$-matrices require less storage than $\mathcal{H}_s$-matrices in uniform precision setting. We have presented a rounding error analysis of $\mathcal{H}_h$-matrices and show that the admissible blocks can be stored in lower precision without sacrificing accuracy. This analysis provides an explicit rule for dynamically selecting the appropriate precision for each admissible block, and we have presented an adaptive mixed precision $\mathcal{H}_h$-matrix representation based on it. We have further investigated the backward error of the matrix–vector product with the resulting representation. Our theoretical results are validated through a series of numerical experiments on kernel matrices. The numerical results confirm that the proposed adaptive mixed precision $\mathcal{H}_h$-matrices can significantly reduce the storage costs over uniform double precision $\mathcal{H}_s$-matrices. We would like to make the implementation of the proposed algorithms publicly available at \url{https://github.com/riteshkhan/ampHmat}.

In our numerical experiments, we set $\ell = L-1$ for the $\mathcal{H}_h$-matrices. Nevertheless, storage gains can be further improved by selecting the optimal switching level. Future work may consider developing \Cref{alg:algo_optimal_level}-type approach in the mixed precision setting, along with a rigorous and comprehensive analysis.

\bibliographystyle{siam}
\bibliography{refs}

\clearpage
\appendix

\section{Proof of \texorpdfstring{\Cref{lemma:global_error}}{lemmaMVP}} \label{sec:app_proof_lemma}
In this section, we prove \Cref{lemma:global_error}, which gives a bound on the approximation error of $\mathcal{H}_h \bkt{p,\epsilon}$-matrices.

\begin{proof}
    We assume that $\widetilde{H}$ is an $\mathcal{H}_h \bkt{p,\epsilon}$-matrix. For any hierarchical matrix, the squared Frobenius norm can be computed block-wise and summed over all blocks, as follows:

\begin{align*}
\resizebox{\linewidth}{!}{$
\begin{aligned}
       \magn{H - \widetilde{H}}^2 & = \quad \dsum_{l=1}^{\ell} \dsum_{i=1}^{2^{dl}} \dsum_{\mathcal{C}_j^{\bkt{l}} \in \mathcal{IL}_s \bkt{\mathcal{C}_i^{\bkt{l}}}} \magn{H^{(l)}_{I,J} - \widetilde{H}^{(l)}_{I,J}}^2 + \dsum_{i=1}^{2^{d \ell}} \dsum_{\mathcal{C}_j^{\bkt{\ell}} \in \mathcal{N}_s \bkt{\mathcal{C}_i^{\bkt{\ell}}}} \magn{H^{(\ell)}_{I,J} - \widetilde{H}^{(\ell)}_{I,J}}^2 \\
        & \quad \quad + \dsum_{l=\ell+1}^{L} \dsum_{i=1}^{2^{dl}} \dsum_{\mathcal{C}_j^{\bkt{l}} \in \mathcal{IL}_w \bkt{\mathcal{C}_i^{\bkt{l}}}} \magn{H^{(l)}_{I,J} - \widetilde{H}^{(l)}_{I,J}}^2 \\
        & \leq \quad \epsilon^2 \dsum_{l=1}^{\ell} \dsum_{i=1}^{2^{dl}} \dsum_{\mathcal{C}_j^{\bkt{l}} \in \mathcal{IL}_s \bkt{\mathcal{C}_i^{\bkt{l}}}} \magn{H^{(l)}_{I,J}}^2 + \epsilon^2 \dsum_{i=1}^{2^{d \ell}} \dsum_{\mathcal{C}_j^{\bkt{\ell}} \in \mathcal{N}_s \bkt{\mathcal{C}_i^{\bkt{\ell}}}} \magn{H^{(\ell)}_{I,J}}^2 \\
        & \quad \quad + \epsilon^2 \dsum_{l=\ell+1}^{L} \dsum_{i=1}^{2^{dl}} \dsum_{\mathcal{C}_j^{\bkt{l}} \in \mathcal{IL}_w \bkt{\mathcal{C}_i^{\bkt{l}}}} \magn{H^{(l)}_{I,J}}^2 \qquad \bkt{\text{By applying } \cref{eq:adm_block_norm}}
       \\ 
       & \leq \quad \epsilon^2 \dsum_{l=1}^{\ell} \dsum_{i=1}^{2^{dl}} \dsum_{\mathcal{C}_j^{\bkt{l}} \in \mathcal{IL}_s \bkt{\mathcal{C}_i^{\bkt{l}}}} \magn{H^{(l)}_{I,J}}^2 + \epsilon^2 \dsum_{i=1}^{2^{d \ell}} \dsum_{\mathcal{C}_j^{\bkt{\ell}} \in \mathcal{N}_s \bkt{\mathcal{C}_i^{\bkt{\ell}}}} \magn{H^{(\ell)}_{I,J}}^2 \\
       & \quad \quad + \epsilon^2 \dsum_{l=\ell+1}^{L} \dsum_{i=1}^{2^{dl}} \dsum_{\mathcal{C}_j^{\bkt{l}} \in \mathcal{IL}_w \bkt{\mathcal{C}_i^{\bkt{l}}}} \magn{H^{(l)}_{I,J}}^2 + \epsilon^2  \dsum_{i=1}^{2^{dL}} \magn{H^{(L)}_{I,I}}^2
        = \epsilon^2 \magn{H}^2
\end{aligned}
$}
\end{align*}

Hence,
\begin{align*}
    \magn{H - \widetilde{H}} \leq \epsilon \magn{H}.
\end{align*}
\end{proof}

\section{Fast Frobenius norm computation}
Here, we present a computationally efficient approach for computing the Frobenius norm using the low-rank factors. 
The squared Frobenius norm is evaluated block-wise and followed by aggregation over all blocks, as described in \Cref{alg:Frob_norm}. 

During the construction of the $\widetilde{H}$ representation, \Cref{alg:Frob_norm} is used in conjunction with \Cref{alg:algo1}. Hence, the Frobenius norm can be computed efficiently without incurring additional computational cost. 

 \begin{algorithm} [H]
 \small
	\caption{Fast Frobenius norm computation.}\label{alg:Frob_norm}
	\begin{algorithmic}[1]
        \State \textbf{Input:} $H \in \Rb^{N \times N}$, $\mathcal{T}^L$, $0 \leq \ell \leq L$, $\epsilon$.
        \State \textbf{Output:} $\magn{\widetilde{H}^{\bkt{l}}_{I,J}}$, $\magn{\widetilde{H}}$.
		\State $\magn{\widetilde{H}} \gets 0$
\For{\texttt{$l=1:\ell$}} 
				\For{\texttt{$i=1:2^{dl}$}}
					\For{$\mathcal{C}^{\bkt{l}}_{j} \in$ $\mathcal{IL}_s \bkt{\mathcal{C}^{\bkt{l}}_{i}}$}
                        \State $\Big[\widetilde{U}^{\bkt{l}}_I, \widetilde{\Sigma}^{\bkt{l}}_{IJ}, \widetilde{V}^{\bkt{l}}_J  \Big] = \texttt{SVDcompression} \bkt{H^{\bkt{l}}_{I,J}, \epsilon}$
                        \State $\magn{\widetilde{H}^{\bkt{l}}_{I,J}}^2 \gets \texttt{trace}\bkt{\bkt{\widetilde{\Sigma}^{\bkt{l}}_{IJ}}^2}$ \Comment{The Frobenius norm of the low-rank blocks is computed using the diagonal matrix $\widetilde{\Sigma}^{\bkt{l}}_{IJ}$ obtained from SVD.}
                        \State $\magn{\widetilde{H}}^2 \gets \magn{\widetilde{H}}^2 + \magn{\widetilde{H}^{\bkt{l}}_{I,J}}^2$
                        \Comment{Add square of the Frobenius norm.}
					\EndFor
				\EndFor
			\EndFor

                \If{$\ell == L$}
            		\For{\texttt{$i=1:2^{d\ell}$}}
    					\For{$\mathcal{C}^{\bkt{\ell}}_{j} \in$ $\mathcal{N}_s \bkt{\mathcal{C}^{\bkt{\ell}}_{i}}$}
                            \State $\magn{\widetilde{H}}^2 \gets \magn{\widetilde{H}}^2 + \magn{H^{\bkt{\ell}}_{I,J}}^2$ 
                            \Comment{Direct computation of Frobenius norm.}
    					\EndFor
    				\EndFor 
                \Else
            		\For{\texttt{$i=1:2^{d\ell}$}}
    					\For{$\mathcal{C}^{\bkt{\ell}}_{j} \in$ $\mathcal{N}_s \bkt{\mathcal{C}^{\bkt{\ell}}_{i}}$}
                            \State $\Big[\widetilde{U}^{\bkt{\ell}}_I, \widetilde{\Sigma}^{\bkt{\ell}}_{IJ}, \widetilde{V}^{\bkt{\ell}}_J  \Big] = \texttt{SVDcompression} \bkt{H^{\bkt{\ell}}_{I,J}, \epsilon}$
                            \State $\magn{\widetilde{H}^{\bkt{\ell}}_{I,J}}^2 \gets \texttt{trace}\bkt{\bkt{\widetilde{\Sigma}^{\bkt{\ell}}_{IJ}}^2}$,\quad $\magn{\widetilde{H}}^2 \gets \magn{\widetilde{H}}^2 + \magn{\widetilde{H}^{\bkt{\ell}}_{I,J}}^2$ 
    					\EndFor
    				\EndFor 
                \EndIf

            \For{\texttt{$l=\ell+1:L$}} 
				\For{\texttt{$i=1:2^{dl}$}}
					\For{$\mathcal{C}^{\bkt{l}}_{j} \in$ $\mathcal{IL}_w \bkt{\mathcal{C}^{\bkt{l}}_{i}}$}
                        \State $\Big[\widetilde{U}^{\bkt{l}}_I, \widetilde{\Sigma}^{\bkt{l}}_{IJ}, \widetilde{V}^{\bkt{l}}_J  \Big] = \texttt{SVDcompression} \bkt{H^{\bkt{l}}_{I,J}, \epsilon}$
                        \State $\magn{\widetilde{H}^{\bkt{l}}_{I,J}}^2 \gets 
                        \texttt{trace}\bkt{\bkt{\widetilde{\Sigma}^{\bkt{l}}_{IJ}}^2}$, \quad $\magn{\widetilde{H}}^2 \gets \magn{\widetilde{H}}^2 + \magn{\widetilde{H}^{\bkt{l}}_{I,J}}^2$ 
					\EndFor
				\EndFor
			\EndFor

            \For{\texttt{$i=1:2^{dL}$}}
                \State $\magn{\widetilde{H}}^2 \gets \magn{\widetilde{H}}^2 + \magn{H^{\bkt{L}}_{I,I}}^2$
                \Comment{Add square of the directly computed Frobenius norm of the leaf-level dense diagonal blocks.}
            \EndFor
	\end{algorithmic}
\end{algorithm}

\section{Numerical results for the Gaussian kernel} \label{sec:Gaussian_kernel}
The global relative error (\Cref{fig:gauss_err_3d}) and the storage gains (\Cref{fig:gauss_mem_3d}) for the Gaussian kernel $(\exp (-{r^2}/{2}))$ with particle set $P_{\subsquare}$ are presented here.

    \begin{figure}[H]
        \centering
        \subfloat[\scriptsize $L=3$ $(P_{\subsquare})$]{
        \includegraphics[scale=.35]{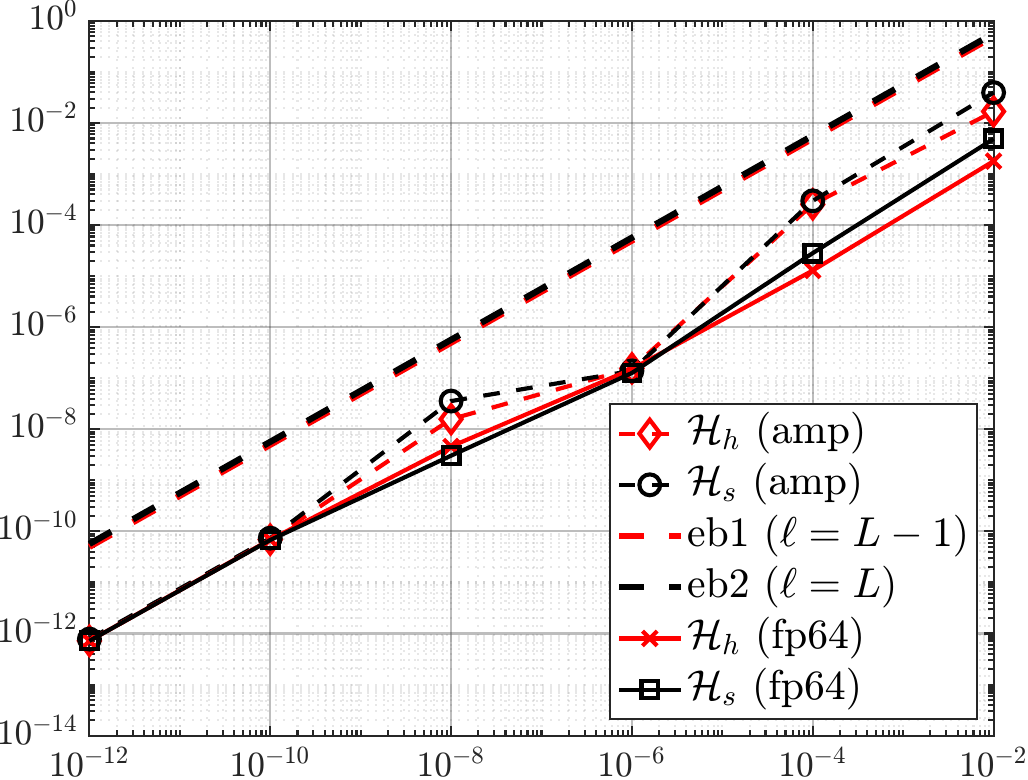}
        \label{fig:gauss_err_1_3d_p1}
        }
        \subfloat[\scriptsize $L=4$ $(P_{\subsquare})$]{
        \includegraphics[scale=.35]{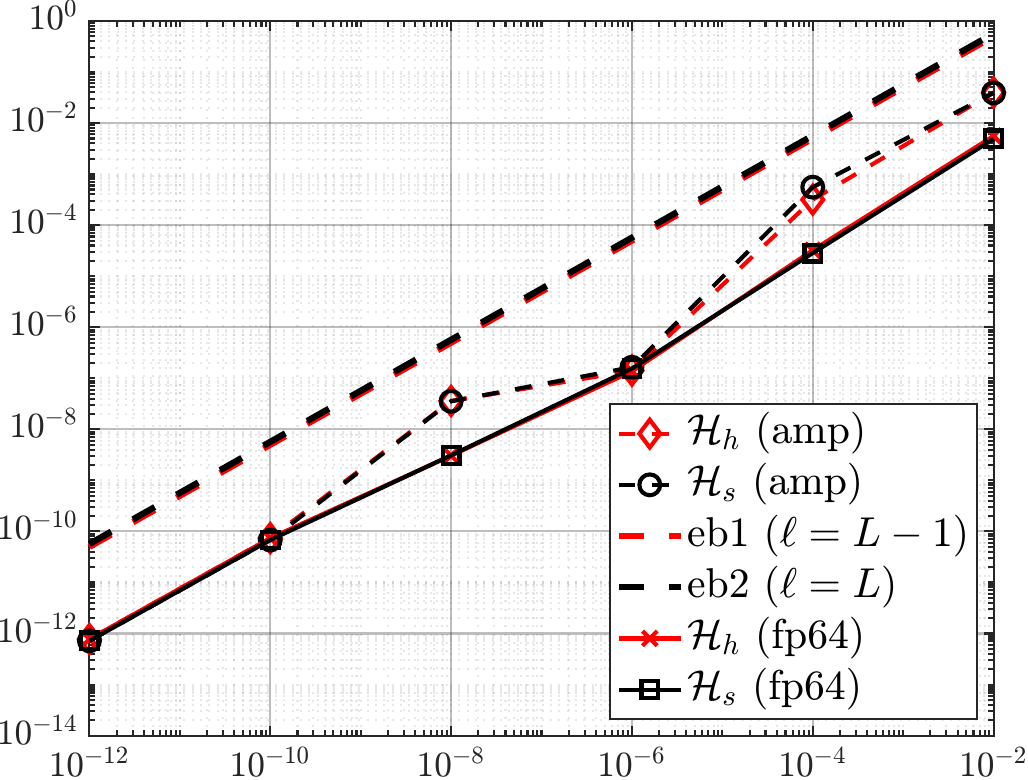}
        \label{fig:gauss_err_2_3d_p1}
        }
        \caption{Global relative error ($y$-axis) of adaptive mixed precision $\mathcal{H}_h$-matrices for different target accuracies ($x$-axis). We consider the Gaussian kernel in $3$D, with $N=64000$.}
        \label{fig:gauss_err_3d}
    \end{figure}

    \begin{figure}[H]
        \centering
        \subfloat[\scriptsize $N=64000, L=3$ $(P_{\subsquare})$]{
        \includegraphics[scale=.34]{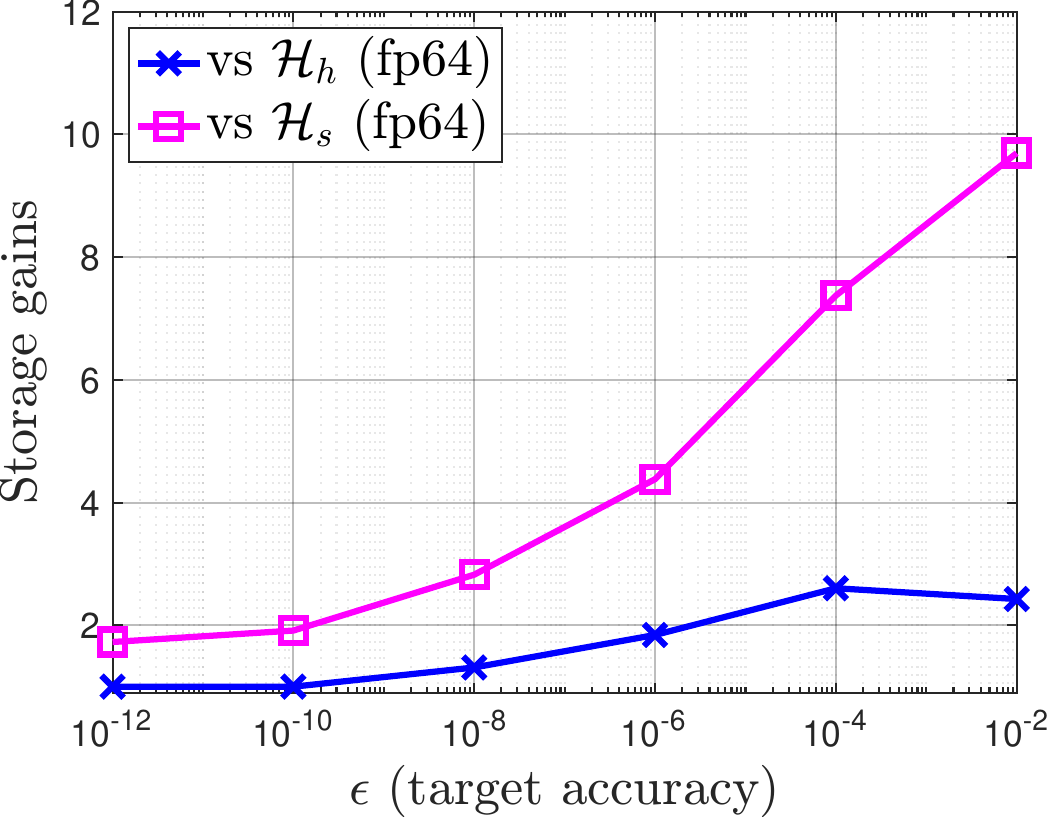}
        \label{fig:gauss_mem_1_3d_p1}
        }
        \subfloat[\scriptsize $N=125000, L=4$ $(P_{\subsquare})$]{
        \includegraphics[scale=.34]{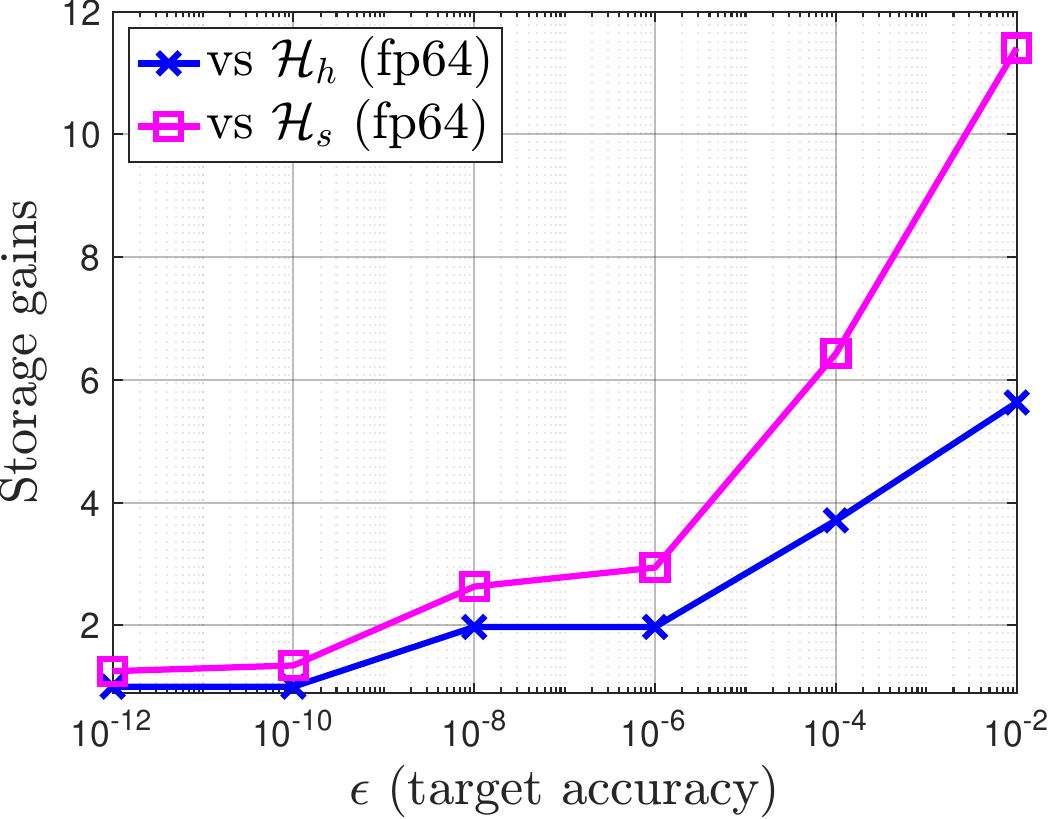}
        \label{fig:gauss_mem_2_3d_p1}
        }
        \caption{Storage gains of adaptive mixed precision $\mathcal{H}_h$-matrices compared with uniform double precision $\mathcal{H}_h$ and $\mathcal{H}_{s}$ matrices for the Gaussian kernel in $3$D.}
        \label{fig:gauss_mem_3d}
    \end{figure}

\section{Numerical results for particle set \texorpdfstring{$P_{\subcircle}$}{Pcircle}} \label{sec:circ_particle_dist}
This section presents the numerical results for the particle set $P_{\subcircle}$. We consider all five kernel functions mentioned in \Cref{sec:num_results}, and construct the corresponding kernel matrices $H$ as described therein.

\subsection{Global error in \texorpdfstring{$\widehat{H}$}{Hhat} representation}
\Cref{fig:log_err_2d_p2,fig:1r_err_3d_p2,fig:1r2_err_3d_p2,fig:gauss_err_3d_p2,fig:mat_err_3d_p2} show how the relative global error, $\magn{H - \widehat{H}}/\magn{H}$, of the adaptive mixed precision representations $\widehat{H}$ (see \Cref{alg:algo2}) varies with target accuracy $\epsilon$.

    \begin{figure}
        \centering
        \subfloat[\scriptsize $L=5$ $(P_{\subcircle})$]{
        \includegraphics[scale=.35]{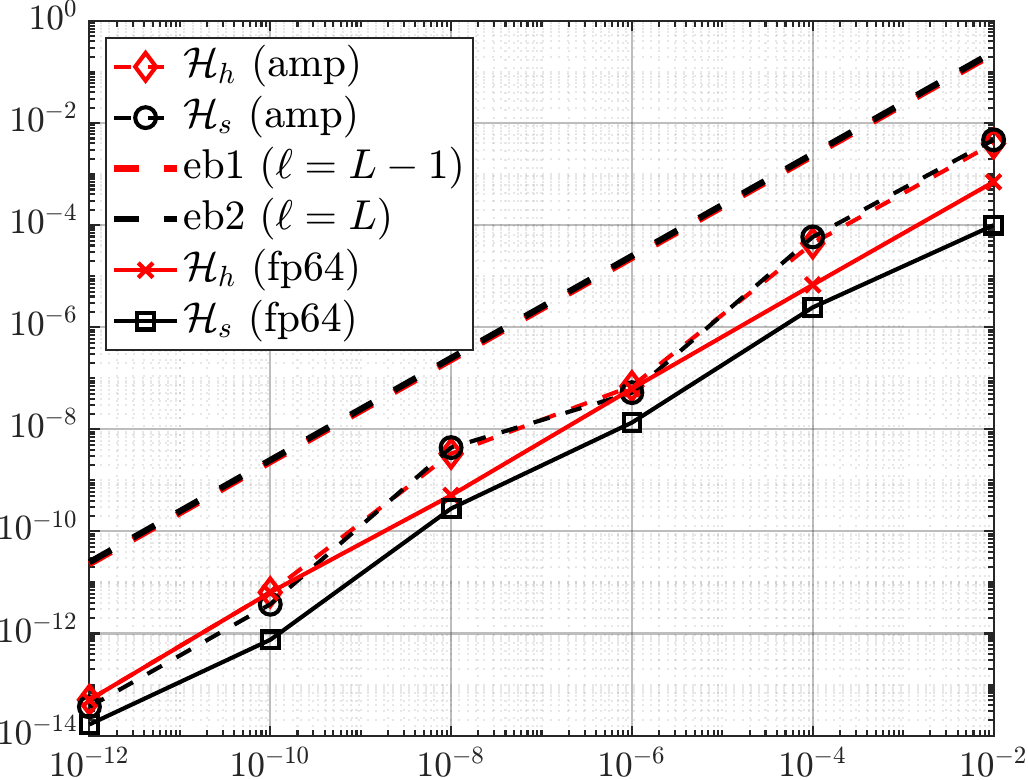}
        \label{fig:log_err_1_2d_p2}
        }
        \subfloat[\scriptsize $L=7$ $(P_{\subcircle})$]{
        \includegraphics[scale=.35]{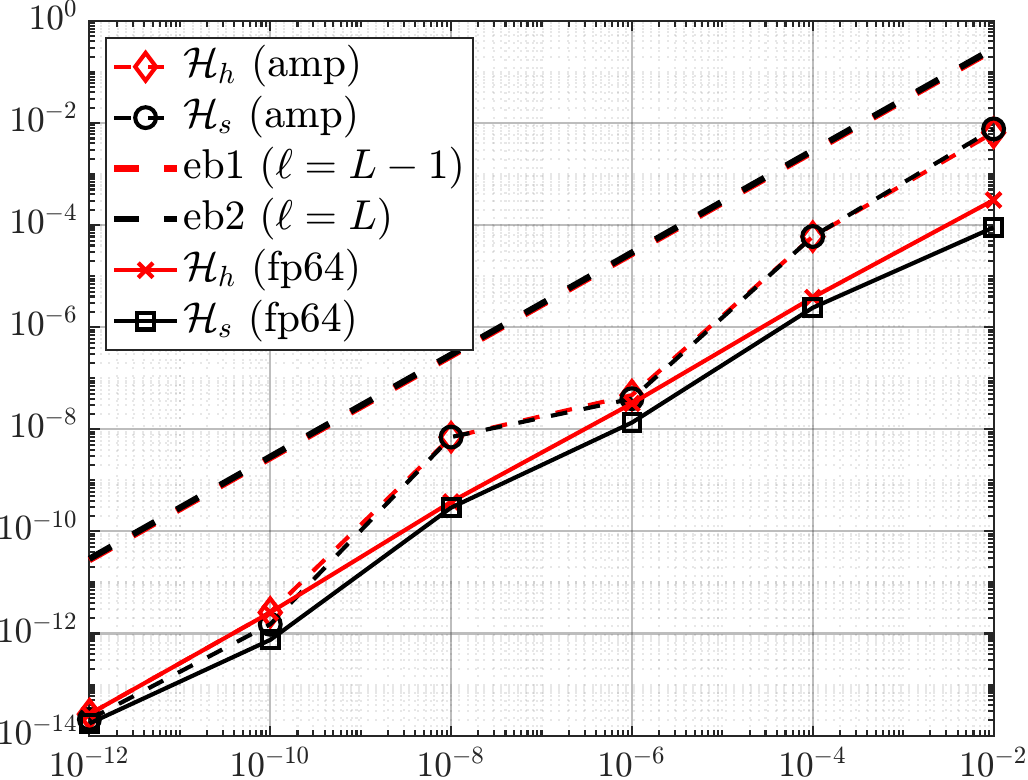}
        \label{fig:log_err_2_2d_p2}
        }
        \caption{Global relative error ($y$-axis) of adaptive mixed precision $\mathcal{H}_h$-matrices for different target accuracies ($x$-axis). We consider the kernel $\log \bkt{r}$ in $2$D, with $N=25600$.}
        \label{fig:log_err_2d_p2}
    \end{figure}

    \begin{figure}
        \centering
        \subfloat[\scriptsize $L=4$ $(P_{\subcircle})$]{
        \includegraphics[scale=.35]{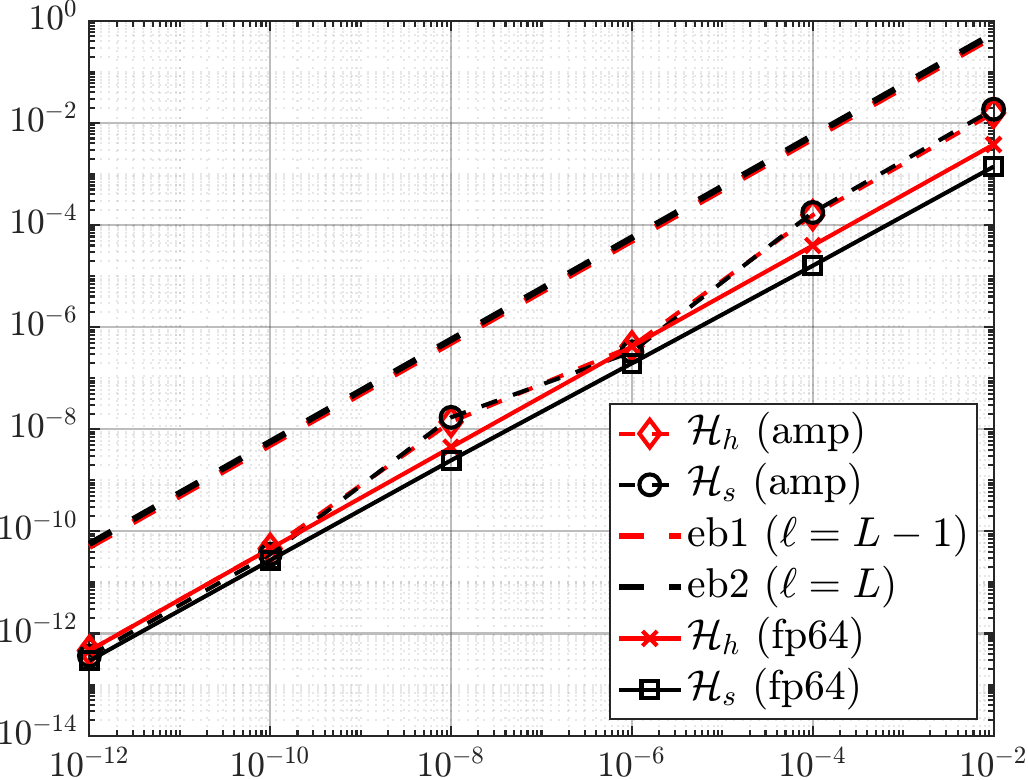}
        \label{fig:1r_err_1_3d_p2}
        }
        \subfloat[\scriptsize $L=5$ $(P_{\subcircle})$]{
        \includegraphics[scale=.35]{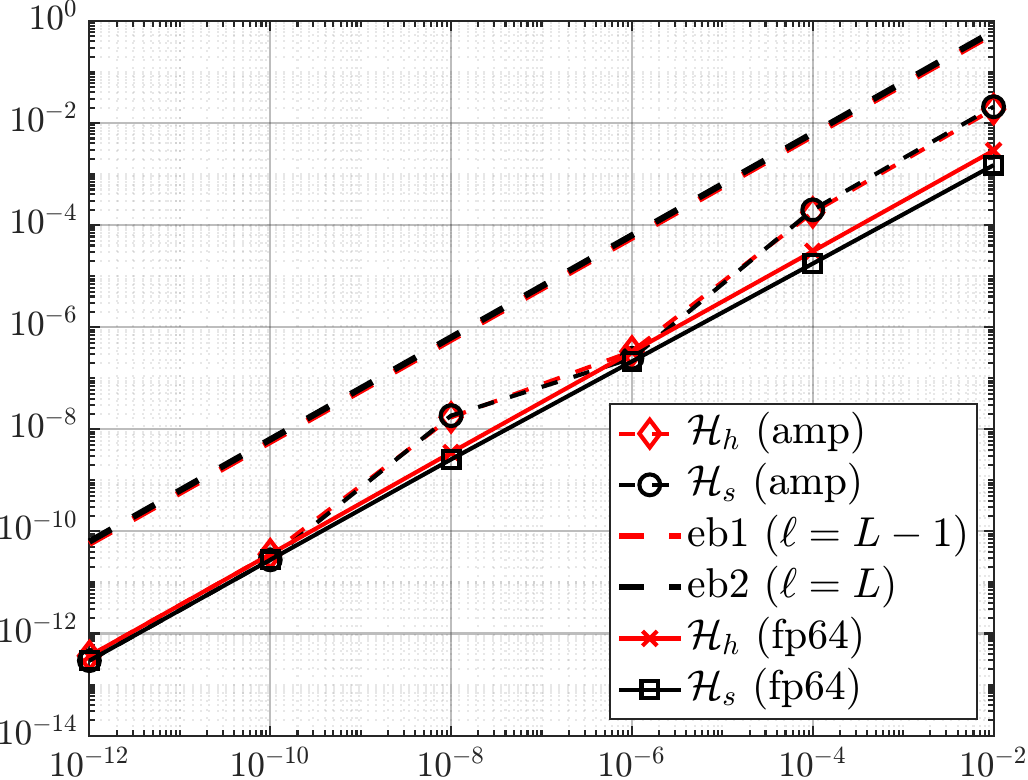}
        \label{fig:1r_err_2_3d_p2}
        }
        \caption{Global relative error ($y$-axis) of adaptive mixed precision $\mathcal{H}_h$-matrices for different target accuracies ($x$-axis). We consider the kernel $1/r$ in $3$D, with $N=64000$.}
        \label{fig:1r_err_3d_p2}
    \end{figure}

    \begin{figure}
        \centering
        \subfloat[\scriptsize $L=4$ $(P_{\subcircle})$]{
        \includegraphics[scale=.35]{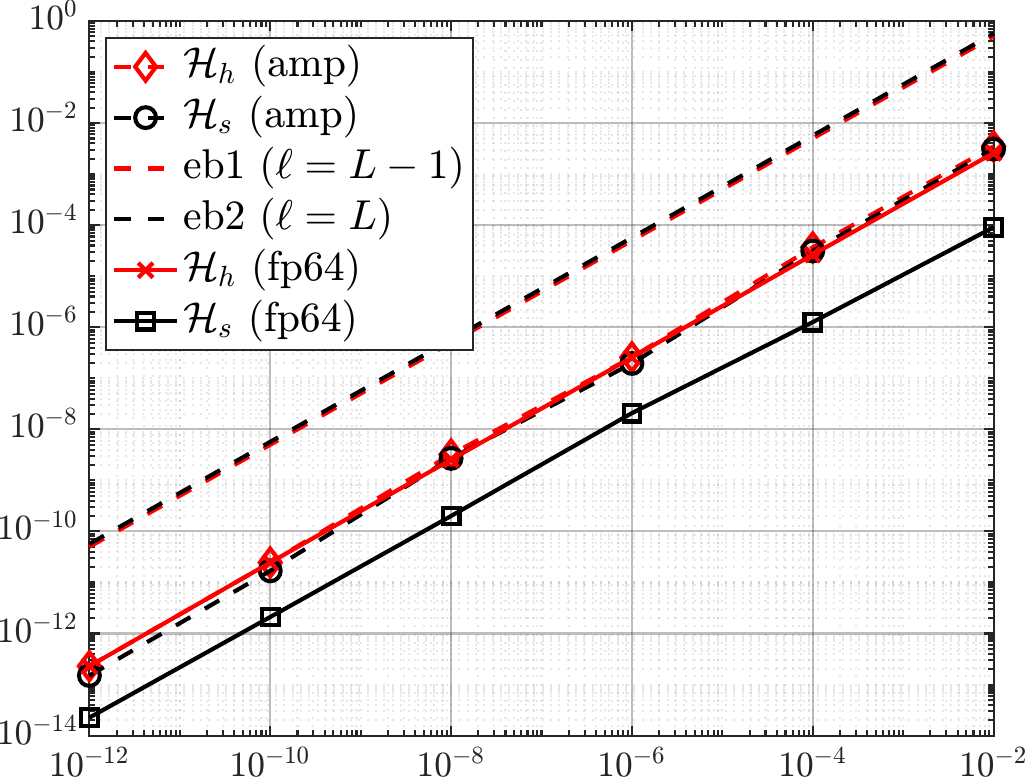}
        \label{fig:1r2_err_1_3d_p2}
        }
        \subfloat[\scriptsize $L=5$ $(P_{\subcircle})$]{
        \includegraphics[scale=.35]{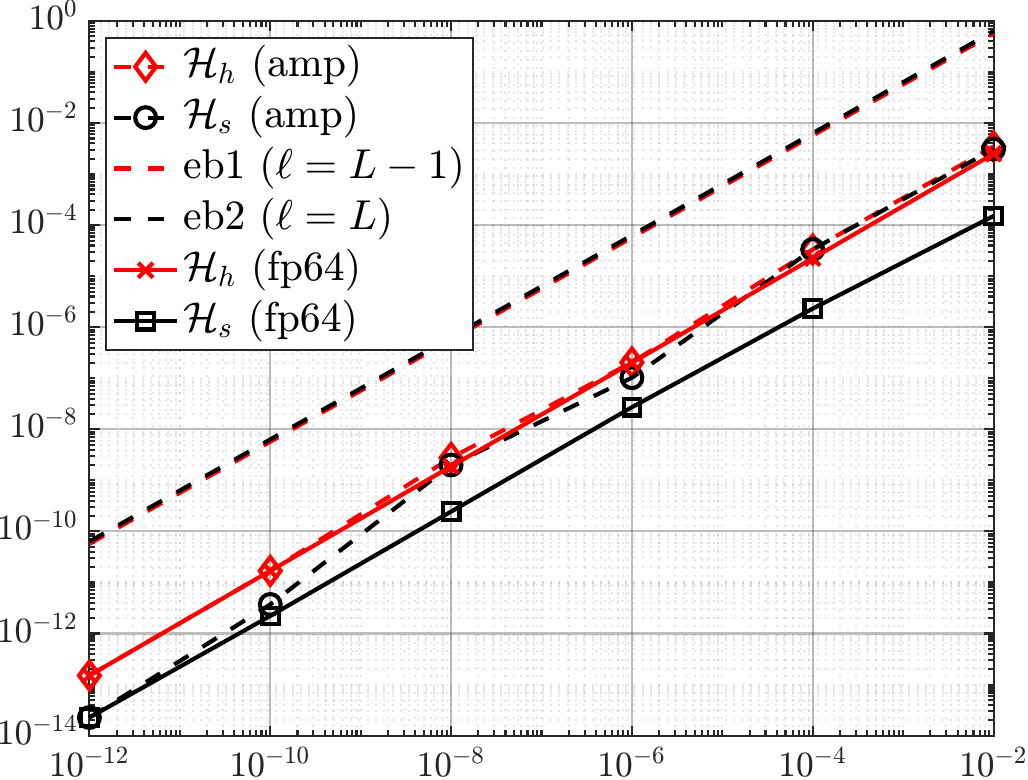}
        \label{fig:1r2_err_2_3d_p2}
        }
        \caption{Global relative error ($y$-axis) of adaptive mixed precision $\mathcal{H}_h$-matrices for different target accuracies ($x$-axis). We consider the kernel $1/r^2$ in $3$D, with $N=64000$.}
        \label{fig:1r2_err_3d_p2}
    \end{figure}

    \begin{figure}
        \centering
        \subfloat[\scriptsize $L=4$ $(P_{\subcircle})$]{
        \includegraphics[scale=.35]{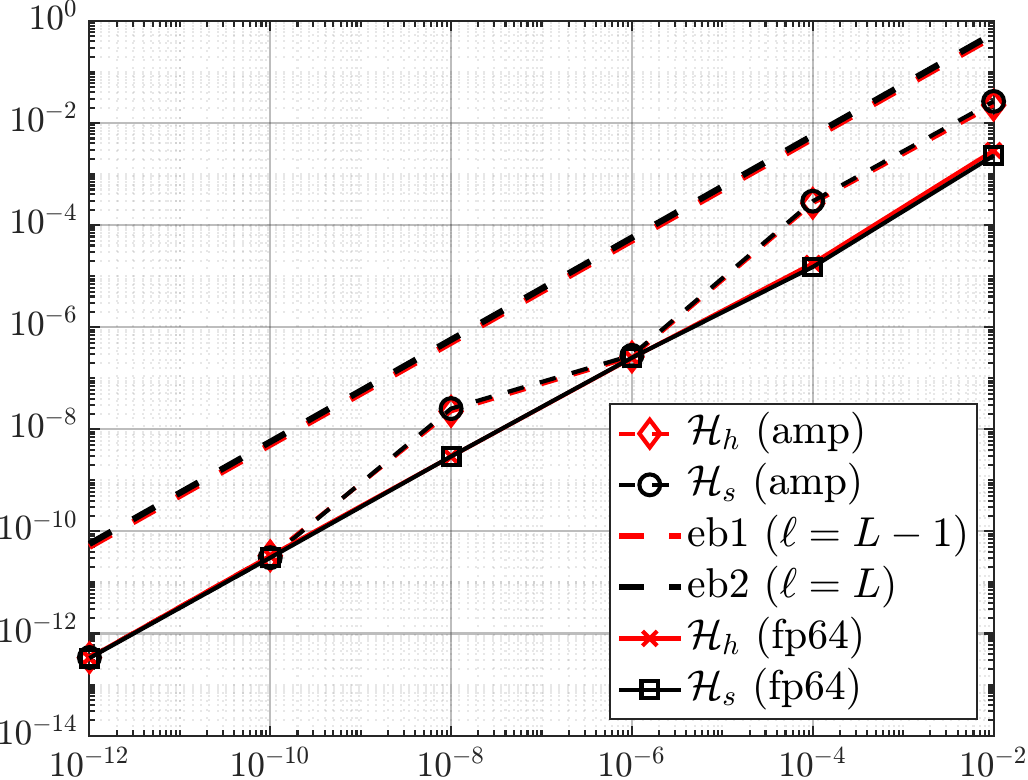}
        \label{fig:gauss_err_1_3d_p2}
        }
        \subfloat[\scriptsize $L=5$ $(P_{\subcircle})$]{
        \includegraphics[scale=.35]{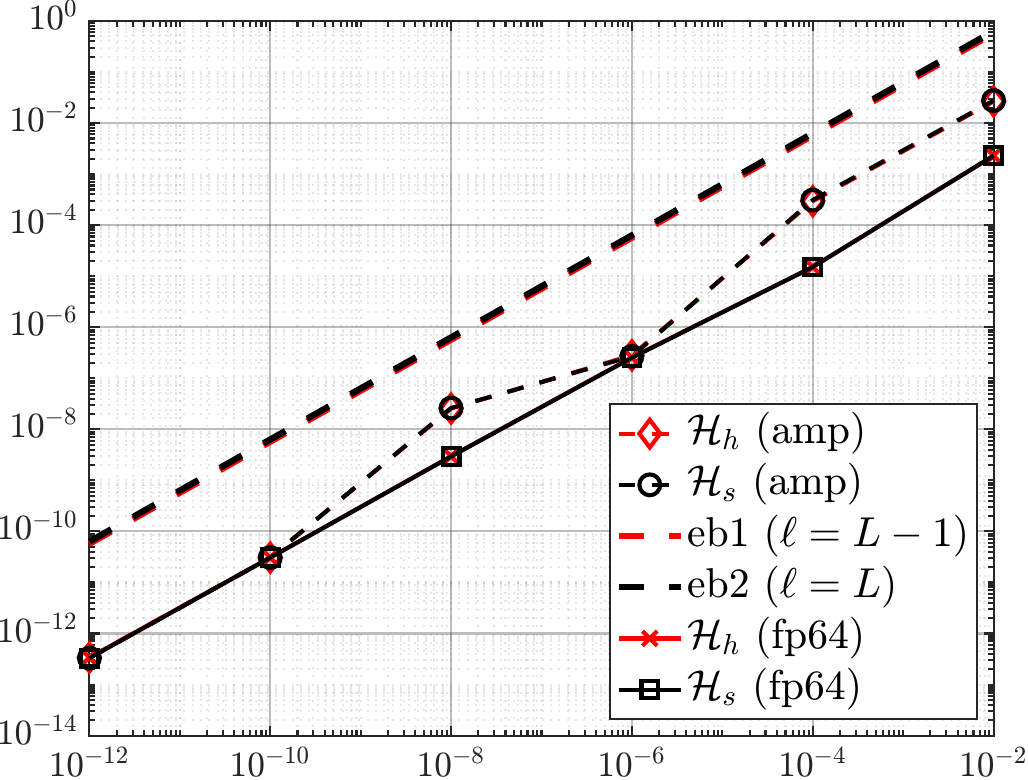}
        \label{fig:gauss_err_2_3d_p2}
        }
        \caption{Global relative error ($y$-axis) of adaptive mixed precision $\mathcal{H}_h$-matrices for different target accuracies ($x$-axis). We consider the Gaussian kernel in $3$D, with $N=64000$.}
        \label{fig:gauss_err_3d_p2}
    \end{figure}

    \begin{figure}
        \centering
        \subfloat[\scriptsize $L=4$ $(P_{\subcircle})$]{
        \includegraphics[scale=.35]{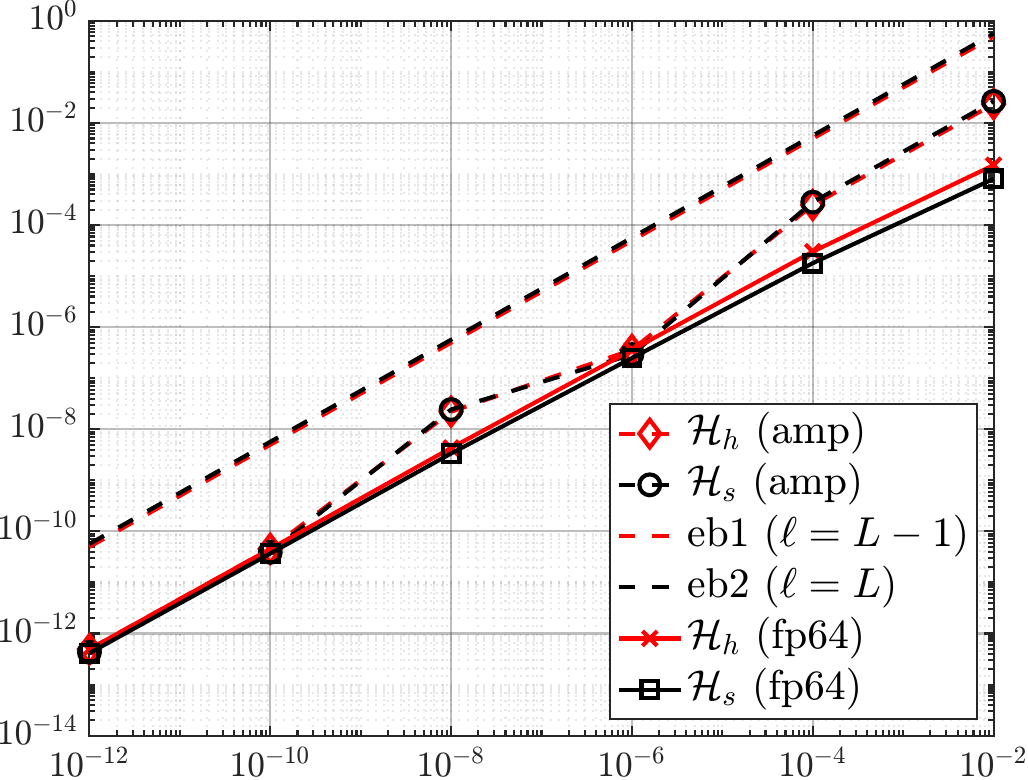}
        \label{fig:mat_err_1_3d_p2}
        }
        \subfloat[\scriptsize $L=5$ $(P_{\subcircle})$]{
        \includegraphics[scale=.35]{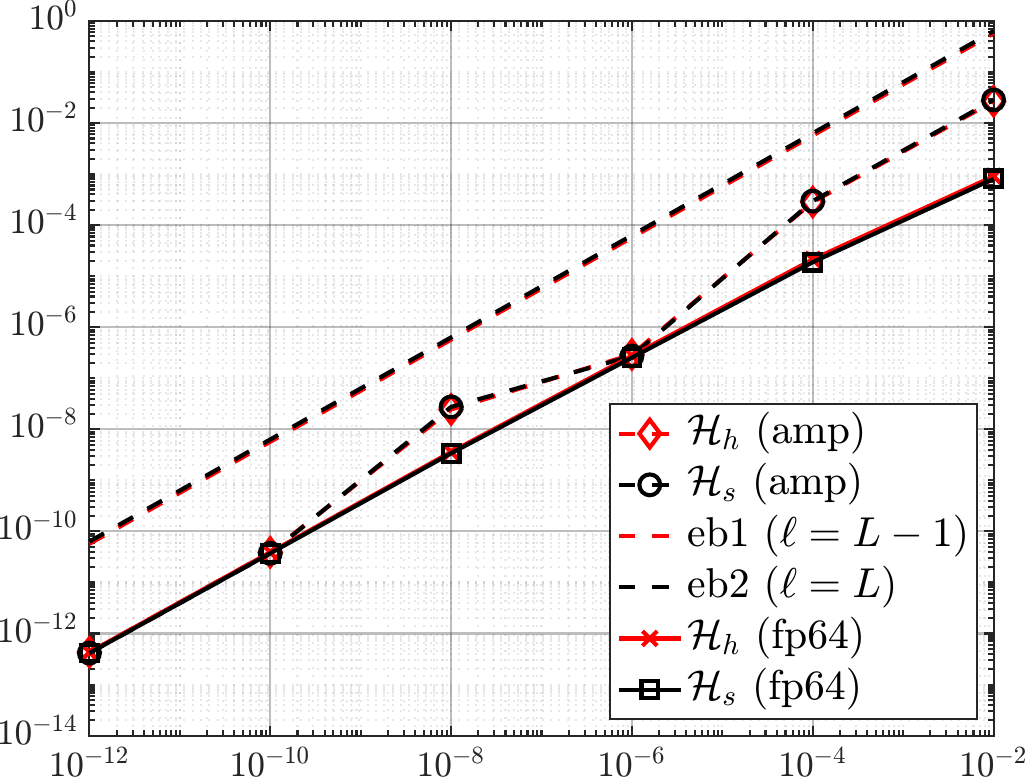}
        \label{fig:mat_err_2_3d_p2}
        }
        \caption{Global relative error ($y$-axis) of adaptive mixed precision $\mathcal{H}_h$-matrices for different target accuracies ($x$-axis). We consider the Matérn kernel in $3$D, with $N=64000$.}
        \label{fig:mat_err_3d_p2}
    \end{figure}

    \subsection{Storage gains in \texorpdfstring{$\widehat{H}$}{Hhat} representation}
    The storage advantages of adaptive mixed precision $\mathcal{H}_h$-matrices over uniform precision representations are illustrated in \Cref{fig:log_mem_2d_p2,fig:1r_mem_3d_p2,fig:1r2_mem_3d_p2,fig:gauss_mem_3d_p2,fig:mat_mem_3d_p2}.
    
    \begin{figure}
        \centering
        \subfloat[\scriptsize $N=25600, L=7$ $(P_{\subcircle})$]{
        \includegraphics[scale=.34]{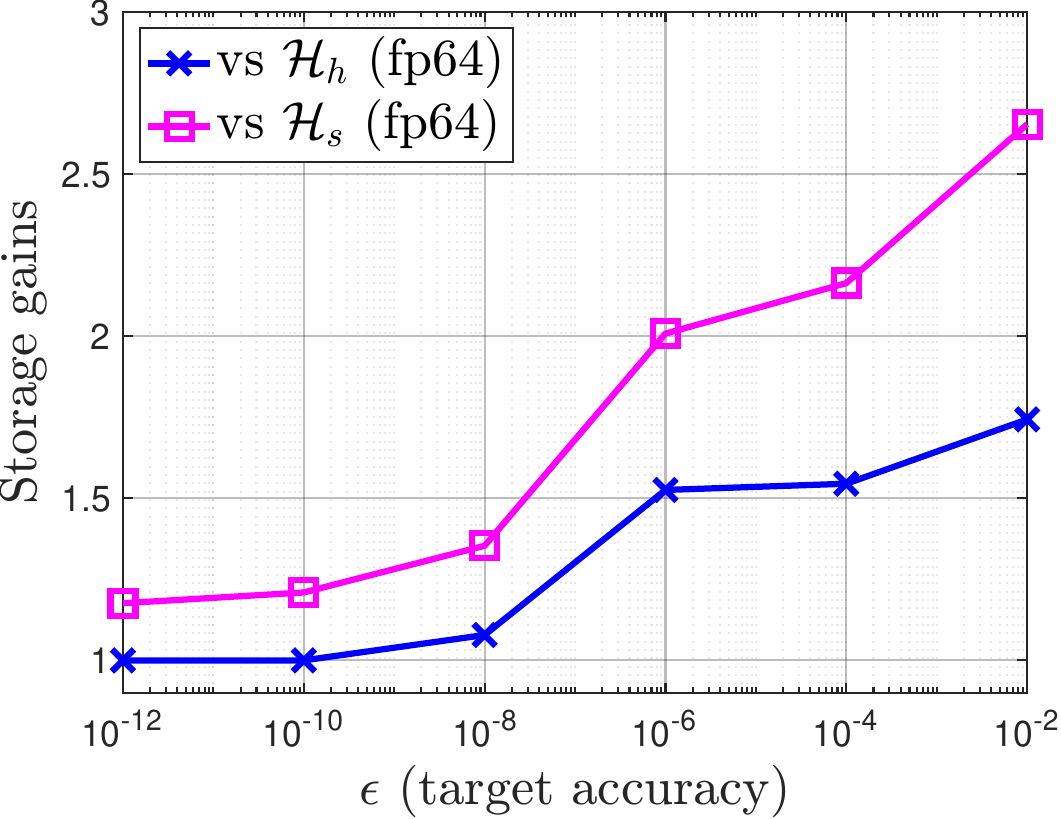}
        \label{fig:log_mem_1_2d_p2}
        }
        \subfloat[\scriptsize $N=102400, L=8$ $(P_{\subcircle})$]{
        \includegraphics[scale=.34]{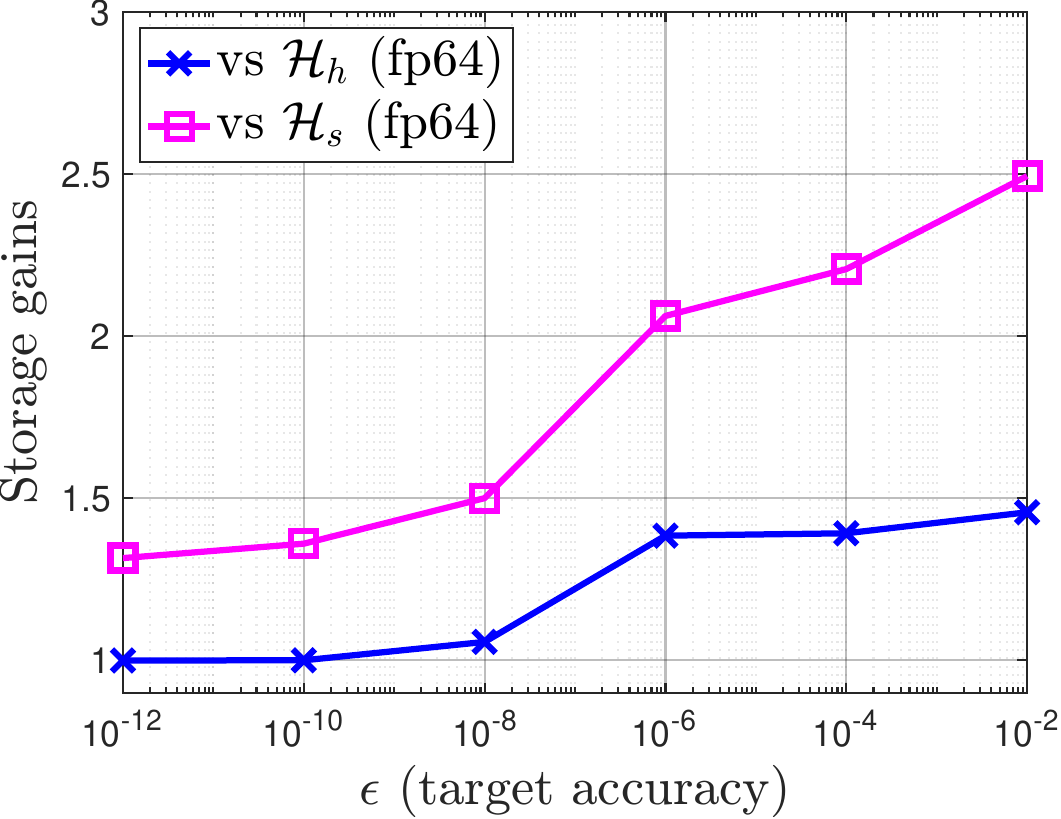}
        \label{fig:log_mem_2_2d_p2}
        }
        \caption{Storage gains of adaptive mixed precision $\mathcal{H}_h$-matrices compared with uniform double precision $\mathcal{H}_h$ and $\mathcal{H}_{s}$ matrices for the kernel $\log \bkt{r}$ in $2$D.}
        \label{fig:log_mem_2d_p2}
    \end{figure}

    \begin{figure}[H]
        \centering
        \subfloat[\scriptsize $N=64000, L=5$ $(P_{\subcircle})$]{
        \includegraphics[scale=.34]{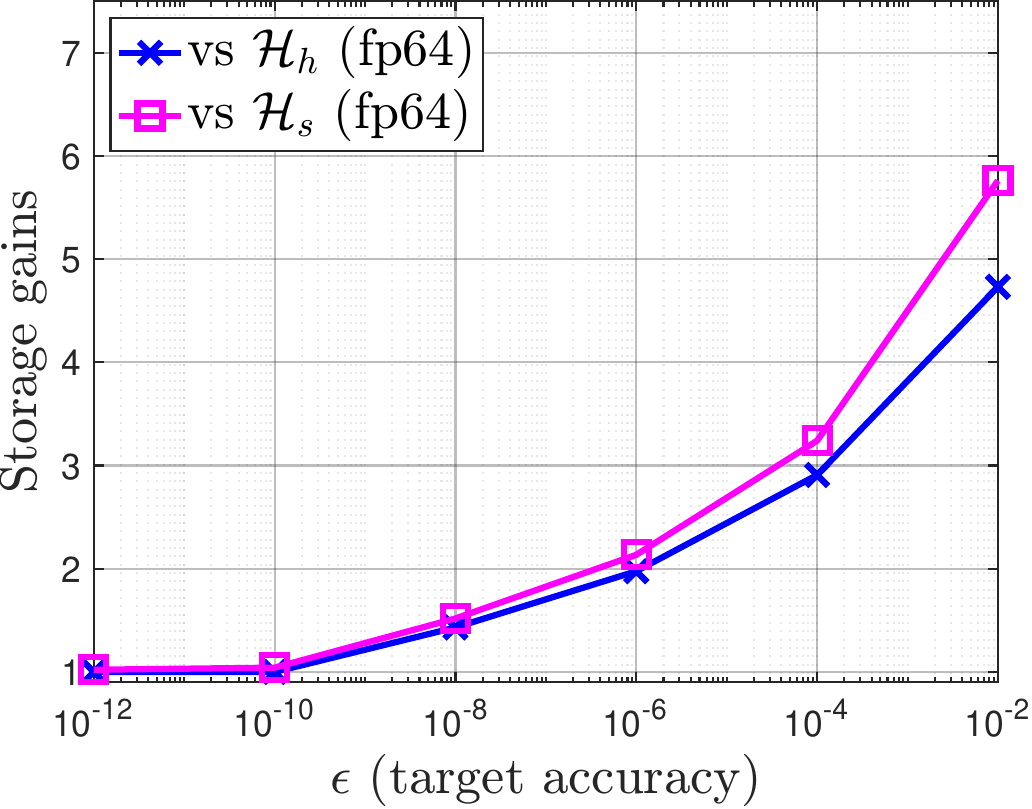}
        \label{fig:1r_mem_1_3d_p2}
        }
        \subfloat[\scriptsize $N=125000, L=5$ $(P_{\subcircle})$]{
        \includegraphics[scale=.34]{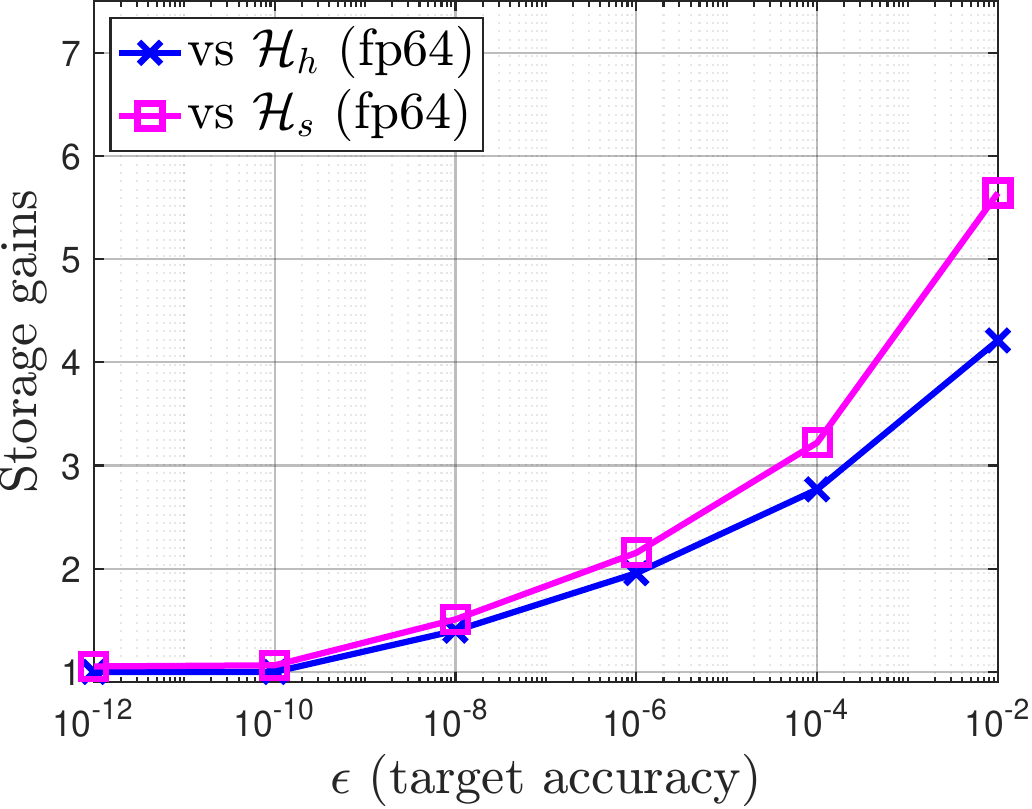}
        \label{fig:1r_mem_2_3d_p2}
        }
        \caption{Storage gains of adaptive mixed precision $\mathcal{H}_h$-matrices compared with uniform double precision $\mathcal{H}_h$ and $\mathcal{H}_{s}$ matrices for the kernel $1/r$ in $3$D.}
        \label{fig:1r_mem_3d_p2}
    \end{figure}

    \begin{figure}[H]
        \centering
        \subfloat[\scriptsize $N=64000, L=5$ $(P_{\subcircle})$]{
        \includegraphics[scale=.34]{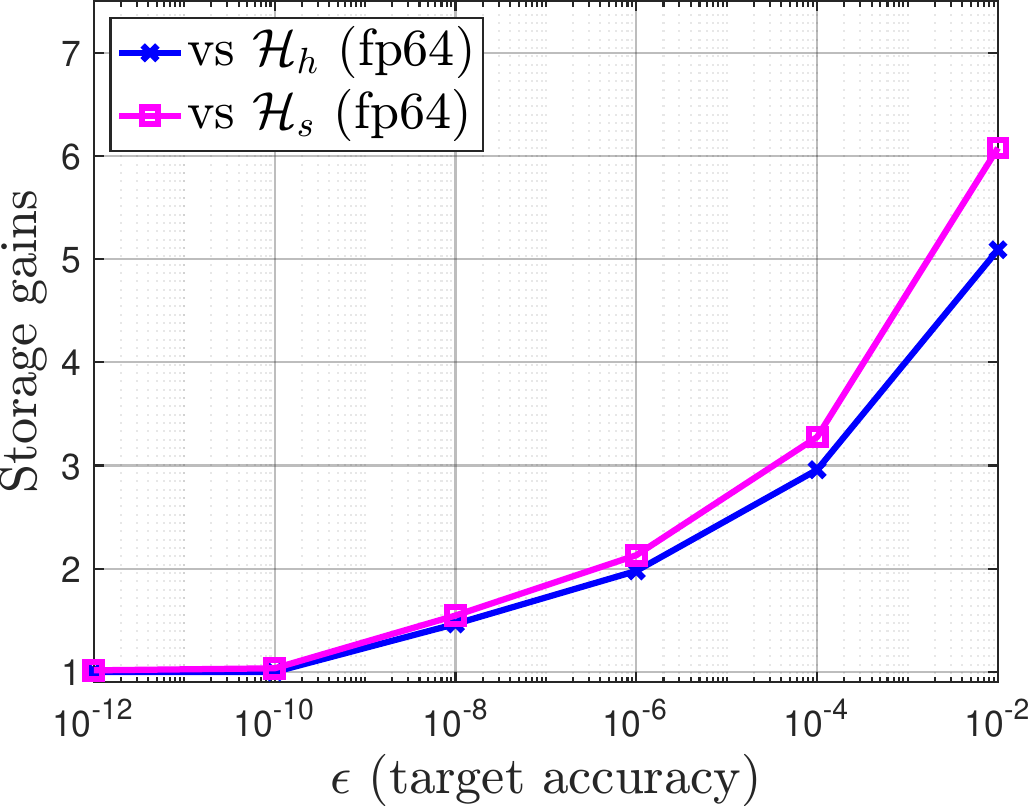}
        \label{fig:1r2_mem_1_3d_p2}
        }
        \subfloat[\scriptsize $N=125000, L=5$ $(P_{\subcircle})$]{
        \includegraphics[scale=.34]{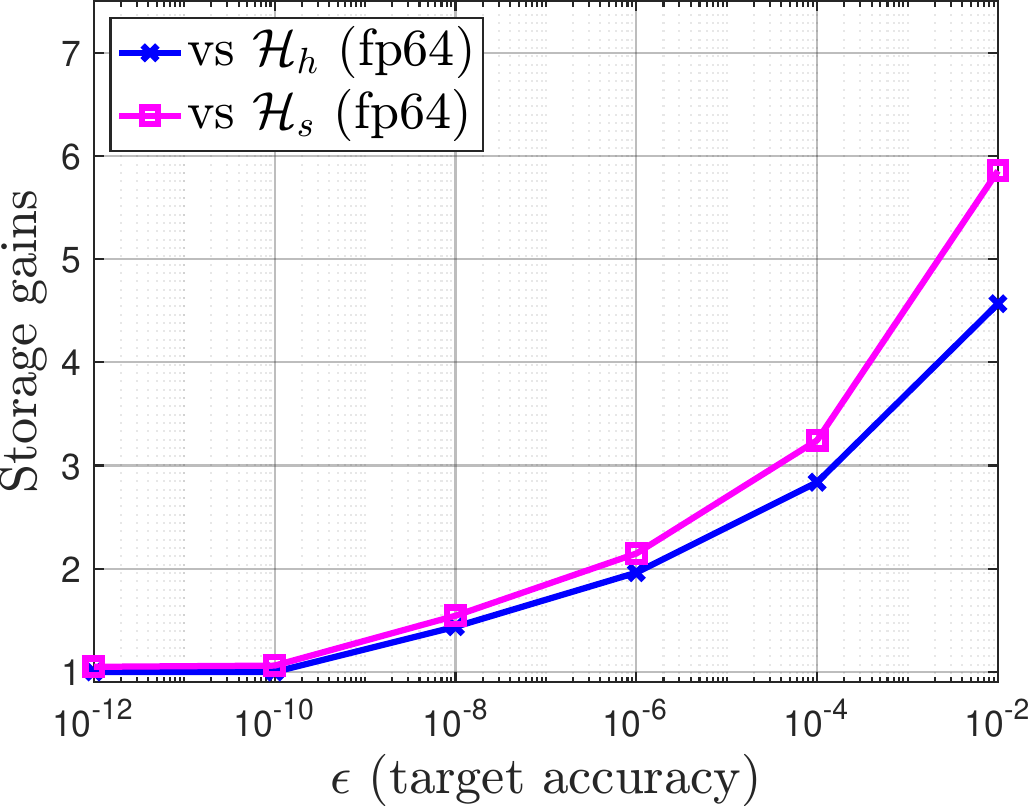}
        \label{fig:1r2_mem_2_3d_p2}
        }
        \caption{Storage gains of adaptive mixed precision $\mathcal{H}_h$-matrices compared with uniform double precision $\mathcal{H}_h$ and $\mathcal{H}_{s}$ matrices for the kernel $1/r^2$ in $3$D.}
        \label{fig:1r2_mem_3d_p2}
    \end{figure}

    \begin{figure}[H]
        \centering
        \subfloat[\scriptsize $N=64000, L=5$ $(P_{\subcircle})$]{
        \includegraphics[scale=.34]{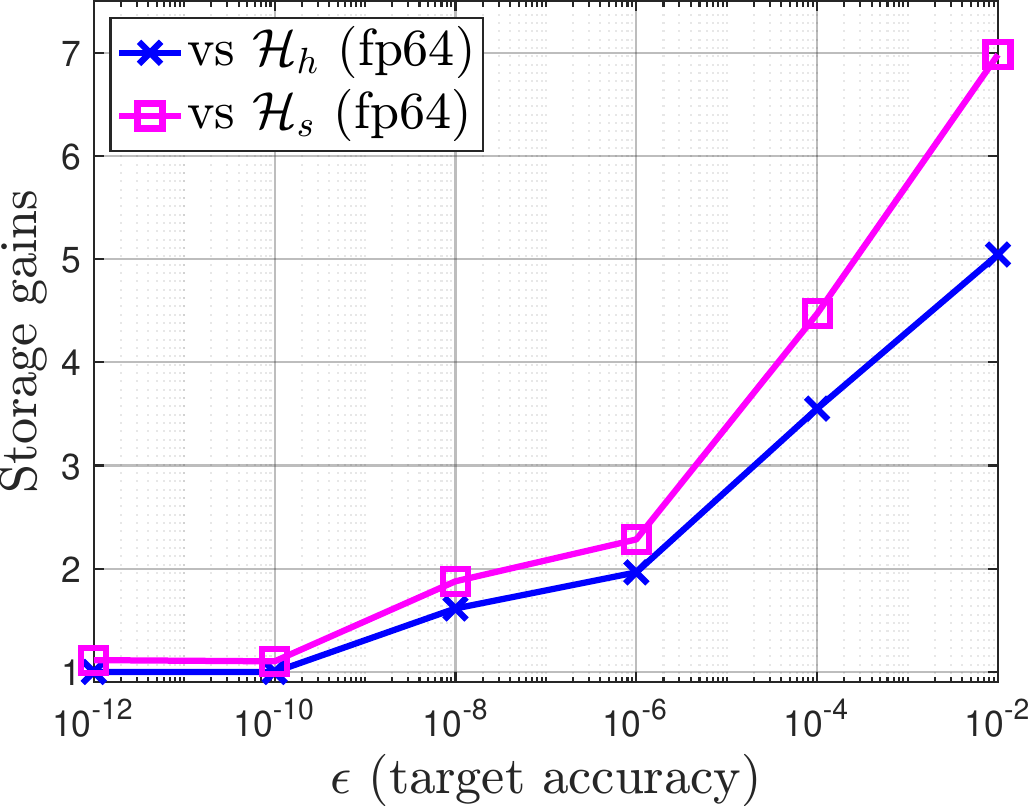}
        \label{fig:gauss_mem_1_3d_p2}
        }
        \subfloat[\scriptsize $N=125000, L=5$ $(P_{\subcircle})$]{
        \includegraphics[scale=.34]{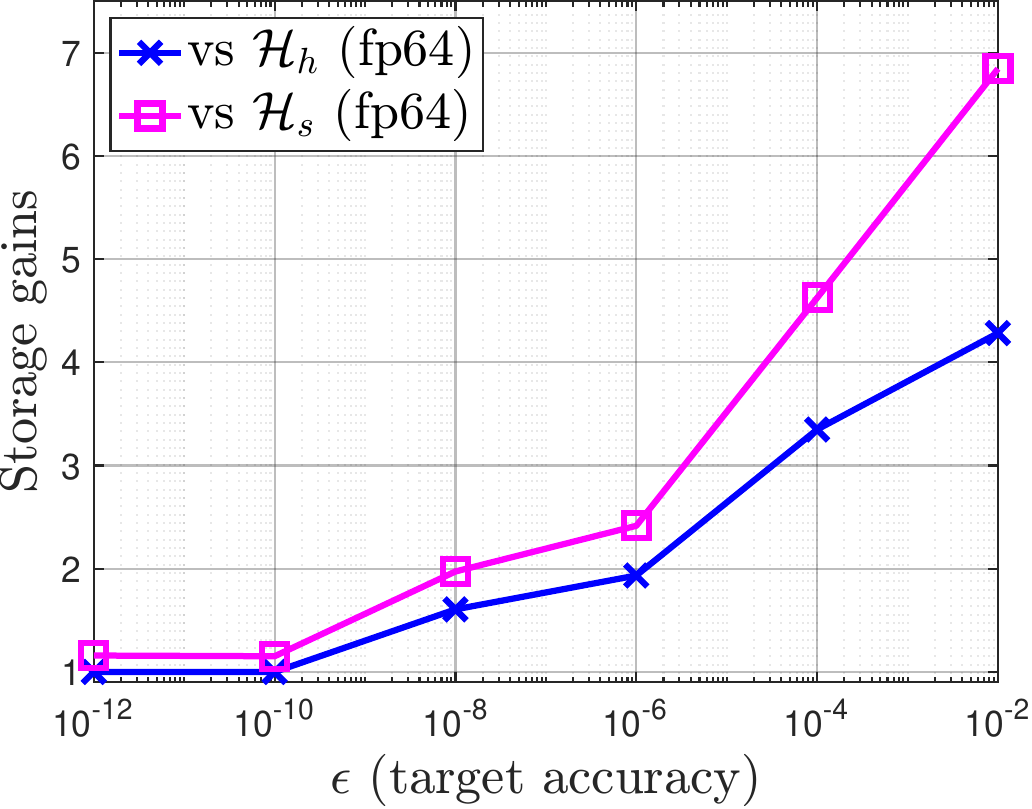}
        \label{fig:gauss_mem_2_3d_p2}
        }
        \caption{Storage gains of adaptive mixed precision $\mathcal{H}_h$-matrices compared with uniform double precision $\mathcal{H}_h$ and $\mathcal{H}_{s}$ matrices for the Gaussian kernel in $3$D.}
        \label{fig:gauss_mem_3d_p2}
    \end{figure}

    \begin{figure}[H]
        \centering
        \subfloat[\scriptsize $N=64000, L=5$ $(P_{\subcircle})$]{
        \includegraphics[scale=.34]{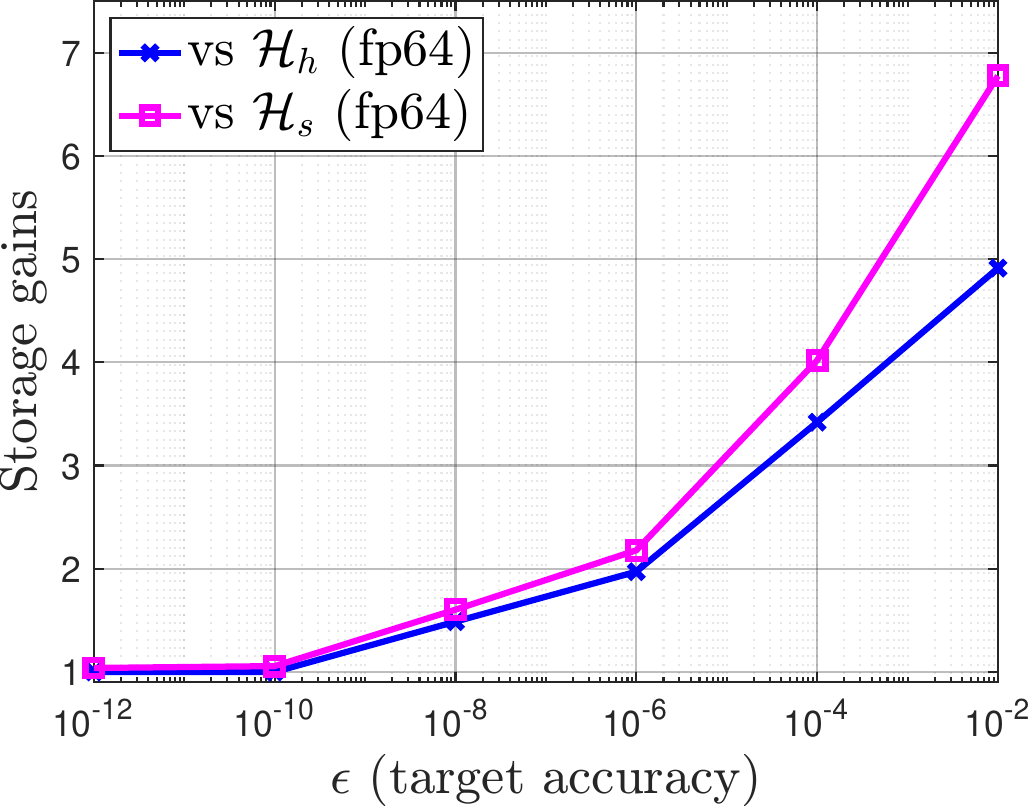}
        \label{fig:mat_mem_1_3d_p2}
        }
        \subfloat[\scriptsize $N=125000, L=5$ $(P_{\subcircle})$]{
        \includegraphics[scale=.34]{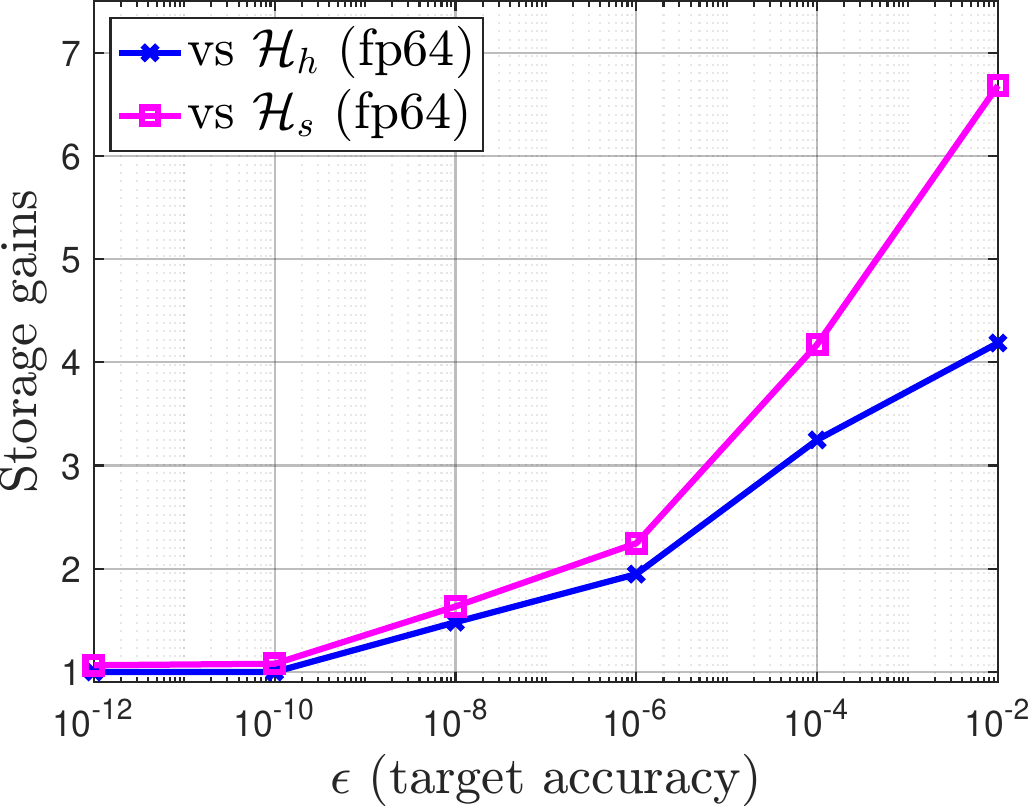}
        \label{fig:mat_mem_2_3d_p2}
        }
        \caption{Storage gains of adaptive mixed precision $\mathcal{H}_h$-matrices compared with uniform double precision $\mathcal{H}_h$ and $\mathcal{H}_{s}$ matrices for the Matérn kernel in $3$D.}
        \label{fig:mat_mem_3d_p2}
    \end{figure}

    \section{Backward error bound of \texorpdfstring{$\widehat{H}x$}{Hx}} 
    We use a lemma from \cite{carson2025mixed} (Lemma 3.4) to prove \Cref{thm:mvp}. The lemma gives a bound on working precision so that the finite precision error in the computed matrix-vector product does not exceed the error due to inexact (low-rank) representation. For convenience, we restate it as follows.

    \begin{lemma} \label{lemma:mvp}
       \emph{[\cite{carson2025mixed}, Lemma 3.4]} Let $\widehat{A}_p$ be an approximation of $A \in \mathbb{R}^{N_b \times N_b}$ such that $\magn{A - \widehat{A}_p} \approx \delta$, for some $\epsilon>0$. Then the error incurred in finite precision computation of $\widehat{y} = \emph{fl} (\widehat{A}_p x)$ does not exceed the error due to the inexact representation, provided that the working precision has unit roundoff $u \leq \delta/(N_b \magn{\widehat{A}_p})$.
    \end{lemma}

    \begin{proof}
        See Lemma 3.4 of \cite{carson2025mixed} for the proof.
    \end{proof}

\subsection{Proof of \texorpdfstring{\Cref{thm:mvp}}{thmmvp}} \label{sec:proof_mvp}
We now aim to bound the backward error of the adaptive mixed precision $\mathcal{H}_h$-matrix–vector product (\Cref{alg:algo3}), i.e., $\widehat{H}x$, by verifying whether each block in the representation satisfies the conditions of \Cref{lemma:mvp}.

\begin{proof}
For an admissible block, using \cref{eq:need_mvp} and substituting $u_{i,j}^{\bkt{l}} \leq \dfrac{\epsilon}{2^{dl/2} \xi_{i,j}^{\bkt{l}}}$, we obtain

\begin{align} \label{eq:app_mvp1}
    \magn{\widetilde{H}^{(l)}_{I,J} - \widehat{H}^{(l)}_{I,J}} \lesssim 2 \xi_{i,j}^{\bkt{l}} u_{i,j}^{\bkt{l}} \magn{\widetilde{H}} \leq 2^{1-dl/2} \epsilon \magn{\widetilde{H}} \approx 2^{1-dl/2} \epsilon \magn{H}.
\end{align}

From \cref{eq:adm_block_norm}, we get

    \begin{align} \label{eq:app_mvp2}
        \magn{H^{\bkt{l}}_{I,J} - \widetilde{H}^{\bkt{l}}_{I,J}} \leq \epsilon \magn{H^{\bkt{l}}_{I,J}} \leq \epsilon \magn{H}.
    \end{align}

Applying the triangle inequality to \cref{eq:app_mvp1} and \cref{eq:app_mvp2} gives

\begin{align}
    \magn{H^{(l)}_{I,J} - \widehat{H}^{(l)}_{I,J}} \leq \bkt{2^{1-dl/2}+1} \epsilon \magn{H}.
\end{align}

Therefore, by \Cref{lemma:mvp}, the error in $\text{fl}\bkt{\widehat{H}^{(l)}_{I,J} x_b}$, where $x_b$ is a segment/block of $x$, does not exceed the error in computing the approximation $\widehat{H}^{(l)}_{I,J}$, provided that the working precision

\begin{align}
    u \leq \frac{\bkt{2^{1-dl/2}+1} \epsilon \magn{H}}{N_b \magn{\widehat{H}^{(l)}_{I,J}}}.
\end{align}

The dense diagonal blocks at level $L$ are stored in working precision. Thus,

\begin{align}
    \magn{H^{\bkt{L}}_{I,I} - \widehat{H}^{\bkt{L}}_{I,I}} \leq u \magn{H^{\bkt{L}}_{I,I}} \leq u \magn{H}.
\end{align}

Define $H_{\mathscr{L}} = H - \widehat{H}$, with $H^{\bkt{l}}_{\mathscr{L}_{I,J}}$ denoting its block $I \times J$ at level $l$. If $u \lesssim \epsilon$, then for each block, the finite-precision error in the block-wise matrix–vector product computed using \Cref{alg:algo3} is bounded above by the corresponding low-rank approximation error. Consequently, when $u \lesssim \epsilon$ (more precisely, $u \lesssim \epsilon/N$)

\begin{align} \label{eq:mvp3}
    \widehat{b} = \text{fl} \bkt{\widehat{H}x}= \bkt{\widehat{H} + H_F} x = \bkt{H - H_{\mathscr{L}} + H_F} x, \quad \magn{H^{\bkt{l}}_{F_{I,J}}} \leq \magn{H^{\bkt{l}}_{\mathscr{L}_{I,J}}}.
\end{align}

By using \cref{eq:mvp3} and \cref{eq:mp_thm}, we obtain

\begin{align} \label{eq:mvp4}
\magn{H_F} \leq \magn{H_{\mathscr{L}}} = \magn{H - \widehat{H}} \lesssim \bkt{\sqrt{\bkt{\ell C^{\prime} + C^{\prime\prime} + \bkt{L-\ell} C^{\prime\prime\prime}}} 2 + 1} \epsilon \magn{H}
\end{align}

Using $\Delta H = H_{F} - H_{\mathscr{L}}$ and applying \cref{eq:mvp4} yields

\begin{align}
\begin{aligned}
    \magn{\Delta H} & = \magn{H_{F} - H_{\mathscr{L}}} \leq \magn{H_{F}} + \magn{H_{\mathscr{L}}}  \leq 2 \magn{H_{\mathscr{L}}} \\ & \lesssim 2 \bkt{\sqrt{\bkt{\ell C^{\prime} + C^{\prime\prime} + \bkt{L-\ell} C^{\prime\prime\prime}}} 2 + 1} \epsilon \magn{H}
\end{aligned}
\end{align}
This completes the proof.
\end{proof}

\end{document}